\newif\iffinal
\finaltrue	

\documentclass[letterpaper,11pt,reqno]{amsart} 
\RequirePackage[utf8]{inputenc}
\usepackage[portrait,margin=3cm]{geometry}
\iffinal\else\usepackage[notref,notcite]{showkeys}\fi
\usepackage{mathrsfs,soul}
\usepackage{hyperref}
\usepackage[foot]{amsaddr}
\usepackage{amssymb,amsthm,amsfonts,amsbsy,latexsym,dsfont}
\usepackage[textsize=small]{todonotes}       
\usepackage{graphicx}
\usepackage[numeric,initials,nobysame]{amsrefs}
\usepackage{upref,setspace}

\usepackage{enumerate}
\newenvironment{enumeratei}{\begin{enumerate}[\upshape (i)]}{\end{enumerate}}
\newenvironment{enumeratea}{\begin{enumerate}[\upshape (a)]}{\end{enumerate}}

\numberwithin{equation}{section}
\numberwithin{figure}{section}
\numberwithin{table}{section}

\newtheorem{thm}{Theorem}[section]
\newtheorem{lem}[thm]{Lemma}
\newtheorem{cor}[thm]{Corollary}
\newtheorem{prop}[thm]{Proposition}
\newtheorem{defn}[thm]{Definition}

\renewcommand{\leq}{\leqslant} 
\renewcommand{\geq}{\geqslant}

\newcommand{\ind}{\mathds{1}}
\newcommand{\eps}{\varepsilon}

\newcommand{\set}[1]{\left\{#1\right\}}

\newcommand{\ie}{\emph{i.e.,}}

\newcommand{\eg}{\emph{e.g.,}}
 
\newcommand{\equald}{\stackrel{\mathrm{d}}{=}}
\newcommand{\probc}{\stackrel{\mathrm{P}}{\longrightarrow}}
\newcommand{\probd}{\stackrel{\mathrm{d}}{\Longrightarrow}}

\def\qed{ \hfill $\blacksquare$}  
\let\ga=\alpha \let\gb=\beta \let\gc=\gamma  
     \let\gl=\lambda        \let\go=\omega   \let\gs=\sigma \let\gt=\tau 
  
\let\gC=\Gamma \let\gD=\Delta   
         \let\gS=\Sigma  
                                
\newcommand{\cA}{\mathcal{A}}\newcommand{\cB}{\mathcal{B}}\newcommand{\cC}{\mathcal{C}}
\newcommand{\cD}{\mathcal{D}}
\newcommand{\cI}{\mathcal{I}}

\newcommand{\cN}{\mathcal{N}}
\newcommand{\cR}{\mathcal{R}}

\newcommand{\cW}{\mathcal{W}}
  
\newcommand{\vone}{\mathbf{1}}

\newcommand{\vC}{\mathbf{C}}

\newcommand{\vR}{\mathbf{R}}
\newcommand{\vU}{\mathbf{U}}\newcommand{\vV}{\mathbf{V}}\newcommand{\vW}{\mathbf{W}}
\newcommand{\vX}{\mathbf{X}}\newcommand{\vZ}{\mathbf{Z}}

\newcommand{\vt}{\mathbf{t}}\newcommand{\vu}{\mathbf{u}}
\newcommand{\vx}{\mathbf{x}}
 
\newcommand{\mvzero}{\boldsymbol{0}}

\newcommand{\bR}{\mathbb{R}}

\newcommand{\dR}{\mathds{R}}

\newcommand{\sA}{\mathscr{A}}

\newcommand{\sE}{\mathscr{E}}

\newcommand{\sN}{\mathscr{N}}

\newcommand{\sS}{\mathscr{S}}

\DeclareMathOperator{\E}{\mathds{E}}
\DeclareMathOperator{\pr}{\mathds{P}}

\DeclareMathOperator{\var}{Var}
\DeclareMathOperator{\cov}{Cov}
\DeclareMathOperator{\argmax}{argmax}

\DeclareMathOperator{\avg}{avg} 
  
\title[Large Average Submatrix]{Energy Landscape for large average submatrix detection problems in Gaussian random matrices}

\date{}

\author[Bhamidi]{Shankar Bhamidi$^1$}
\address{$^1$Department of Statistics and Operations Research, 304 Hanes Hall, University of North Carolina, Chapel Hill, NC 27599}
\author[Dey]{Partha S.~Dey$^2$}
\address{$^2$Courant Institute of Mathematical Sciences, New York University, 251 Mercer Street, New York, NY 10012}
\author[Nobel]{Andrew B.~Nobel$^1$}
\email{bhamidi@email.unc.edu, partha@cims.nyu.edu, nobel@email.unc.edu}

\begin{document}

\begin{abstract}
The problem of finding large average submatrices of a real-valued matrix arises in the exploratory analysis of data from a variety of disciplines, ranging from genomics to social sciences.  
In this paper we provide a detailed asymptotic analysis of large average submatrices of an $n \times n$ Gaussian random matrix. 
The first part of the paper addresses global maxima.
For fixed $k$ we identify the average and the joint distribution of the $k \times k$ submatrix having 
largest average value.   
As a dual result, we establish that the size of the largest square sub-matrix with average 
bigger than a fixed positive constant is, with high probability, equal to one of two consecutive 
integers that depend on the threshold and the matrix dimension $n$. 
The second part of the paper addresses local maxima.  Specifically we consider 
submatrices with dominant row and column sums that arise as the local optima
of iterative search procedures for large average submatrices. For fixed $k$, we identify the limiting average value and joint distribution of a $k \times k$ submatrix  conditioned to be a local maxima. In order to understand the density of such local optima and explain the quick convergence of such iterative procedures, we analyze the number $L_n(k)$ of local maxima, beginning with exact asymptotic expressions for the mean and fluctuation behavior of $L_n(k)$. 
For fixed $k$, the mean of $L_{n}(k)$ is $\Theta(n^{k}/(\log{n})^{(k-1)/2})$ while the standard deviation is $\Theta(n^{2k^2/(k+1)}/(\log{n})^{k^2/(k+1)})$.
Our principal result is a Gaussian central limit theorem for $L_n(k)$ that is based on a new variant of Stein's method. 
\end{abstract}
\maketitle



\section{Introduction}
\label{sec:int}

The study of random matrices is an important and active area in modern probability.  The majority of 
the existing work on random matrices has focused on their spectral properties, often in the Gaussian
setting.  By contrast, in this paper we are interested in exploring the structural properties of random
matrices by means of their extreme submatrices, in particular, submatrices with large average.
As motivation for this point of view, we note that many of the large data sets that are now
common in biomedicine, genomics, and the study of social networks can be represented in the form 
of a data matrix with real valued entries.  A common first step in the exploratory analysis, or ``mining'', 
of such data sets is the search for unusual structures or patterns that may be of potential scientific importance.  
Structures of practical interest include distinguished submatrices of the data matrix. 
The search for such submatrices is referred to as biclustering, cf.\ \cite{madeira-survey}.  
Despite their simplicity, submatrices distinguished by having large average value have
proven useful in a number of applications.  In genomics analyses, the $(i,j)$ element of the data matrix 
typically represent the value of a measured biological quantity indexed by $i$ 
(such as gene expression or copy number) in the $j$-th sample.  In this case, a large
average submatrix may capture an interesting biological interaction between a group of samples and a group
of variables (see~\cite{shabalin2009finding} and the references therein).    
In the study of social networks, it is often meaningful to derive a data matrix whose entries 
represent the strength of interactions between different 
individuals in a network.  In this case, large average submatrices indicate groups of individuals having strong interactions  within the network, and  for the subsequent detection and identification of (potentially overlapping) communities~\cite{fortunato2010community}.

In this paper we provide a detailed asymptotic analysis of large average submatrices of a Gaussian random
matrix. We consider the case in which the random matrix and the submatrices of interest are square, \ie\ they
have the same number of rows and columns.  The first part of the paper addresses global maxima.
For fixed $k$, we identify the limiting average value and joint distribution of the $k \times k$ submatrix  
with largest average.  The proof relies in part on a refined Gaussian comparison result that
may be of independent interest.
As a dual result, we establish two-point concentration for the size of the largest 
$k \times k$ submatrix with average greater than a fixed positive constant.  

The second part of the paper addresses submatrices that are local maxima, in the sense 
that their row and column sums dominate those in the
``strips'' defined by their column and row sets, respectively.  Submatrices of this sort arise as the fixed points of a natural iterative search procedure for large average submatrices \cite{shabalin2009finding} that has proven useful in the
analysis of genomic data.  For fixed $k$, we study distributional asymptotics for a $k \times k$ submatrix conditioned to be a local maxima, and we obtain a precise asymptotic
expression for the probability that a given submatrix is a local maxima. In order to understand the density of such local optima and explain the quick convergence of such iterative procedures, we study the number of local optima, $L_n(k)$ in an $n \times n$ random matrix.  
We derive refined bounds on the expectation and variance of $L_n(k)$, showing, in particular, that
\begin{align*}
\E L_n(k) &= \Theta\left( \left(\frac{n}{\sqrt{\log n}} \right)^k \cdot \sqrt{\log n} \right) 
\\ 
\mbox{ and } \  
\var(L_n(k)) &=  \Theta\left( \left( \frac{n}{\sqrt{\log n}} \right)^{2k^2 / (k+1)} \right) .
\end{align*}
The non-standard scaling of the mean reflects unexpectedly weak dependence between 
row and column dominance and the non-standard scaling of the variance arises in part from subtle and persistent correlations 
between pairs of locally optimal submatrices.  Using these results, we establish that the average of a typical local maxima is within a factor of $1/\sqrt{2}$ of the global maxima. 
Also due to the complex correlation structure of the local maxima,
existing methods do not yield a central limit theorem for $L_n(k)$.  
Nevertheless, we establish
a central limit theorem for $L_n(k)$ using a new variant of Stein's method.

In the past several years there has been renewed interest (see \eg~\cites{mahoney2010algorithmic,mahoney2011randomized}) in the
study of local optima as a tool for exploratory data analysis.
The study of optimization problems, and properties of optimal or locally optimal configurations 
for random data, is now a flourishing subbranch of discrete probability (see \eg\ \cites{steele-book,aldous2009dynamic}) and have arisen in a wide array of models, ranging from genetics and 
NK fitness models see~\cites{durrett-limic,evans-steinsaltz,limic-pemantle} to statistical physics and spin glasses, see~\cite{montan-book}. We defer a full fledged discussion to Section~\ref{sec:open}.

\subsection{Outline of the Paper}

The principal results of the paper, and a discussion of related work, are presented in the next section.  
Results for global maxima including a two point localization phenomena are described in
Sections \ref{sec:glob-opt-res} and \ref{sec:two-pt}. We then describe an iterative search procedure used in practice for finding large average submatrices in Section \ref{sec:loc-opt-res}.  Results for local maxima are described in this Section.  
We then provide more background for the problems studied in this paper and connections between our work to existing literature in Section \ref{sec:open}. 
Section \ref{sec:est} collects some of the technical estimates we need for the proofs of the main results. The reader is urged to skim through these results and then come back to them as and when they are used.  We complete the proofs about global optima in Section \ref{sec:proofs-gopt}. We prove the structure theorem for local optima in Section \ref{sec:proofs-lopt} whilst the variance asymptotics for the number of local optima are proved in Section \ref{sec:var-asymp}. Finally we present the proof of the central limit theorem for number of local optimal sub matrices in Section~\ref{sec:clt}.

\section{Statement and Discussion of Principal Results}
\label{sec: PrincRes}

\subsection{Basic Definitions and Notation}
\label{sec:not}
For integers $a \leq b$ define $[a,b] := \{a,a+1,\ldots,b-1,b\}$; when $a=1$, the interval $[1,b]$ will be denoted by $[b]$. 
Boldface capital letters, e.g.~$\vW$, will denote matrices, with corresponding the lower case, 
e.g.~$w_{ij}$, denoting their entries.
Let $\vW = (( w_{ij} ))_{i,j\ge 1}$ be an infinite two dimensional array of independent standard normal random variables
defined on a common probability space.  Let $\vW^n = ((w_{ij}))_{i, j=1}^{n}$ be the $n \times n$ Gaussian
random matrix constituting the upper left corner of $\vW$.  In what follows $[n]$ denotes the set 
$\{1,2,\ldots,n\}$.  For $n \geq 1$ and $1 \leq k \leq n$ let 
\[
\sS_{n}(k) := \{ I \times J : I,J \subseteq [n] \mbox{ with } |I|=|J|=k \}
\]
be the family of index sets of $k \times k$ submatrices of $\vW^n$.    
For $\gl = I \times J \in \sS_{n}(k)$, let $\vW_{\gl}  = (( w_{ij}))_{i \in I, j \in J }$ be the submatrix of
$\vW^n$ (also a submatrix of $\vW$) with index set $\gl$.  
Note that $|\sS_n(k)| = {n \choose k}^2$. 
For index sets $\gl, \gc \in \sS_{n}(k)$, we write $|\gl \cap \gc| = (s,t)$ to denote the fact that $\gl$ and $\gc$ share $s$ rows 
and $t$ columns. Note that $\gl \cap \gc = \emptyset$ if and only if $|\gl \cap \gc| = (0,0)$.

For any finite, real-valued matrix $\vU = (( u_{ij} ))$ let
\[
\avg(\vU) = |\vU|^{-1} \sum u_{ij}
\]
be the average of the entries of $\vU$, where $|\vU|$ denotes the number of entries of $\vU$.  
For $x\in \dR$ let
\[
\Phi(x) := \pr(Z \leq x) \ \text{ and } \ \bar{\Phi}(x) := 1-\Phi(x) 
\]
be the cumulative distribution function and complementary cumulative distribution function, respectively,
of a standard normal random variable $Z$. 

In considering extremal submatrices of a Gaussian random matrix, we make use of, and extend, 
classical results on the extreme values of the standard normal.  In what follows,
\begin{equation}
\label{eqn:an-def}
a_N := \sqrt{2 \log N}	
\end{equation}
and 
\begin{equation}
\label{eqn:bn-def}
	b_N := \sqrt{2 \log N} - \frac{\log (4 \pi \log N)}{2 \sqrt{2 \log N}}
\end{equation}
refer to the scaling and centering constants, respectively, for the maximum of $N$ independent
standard Gaussian random variables.

\subsection{Structure Theorem for Global Optima}
\label{sec:glob-opt-res}
We begin by investigating the average value and joint distribution of the $k \times k$ submatrix of
$\vW^n$ having maximum average, which we refer to as the global optimum.  To this end let
\[
\gl_n(k) :=  \argmax\{ \avg(\vW_{\gl}) : \gl \in \sS_{n}(k) \}
\]

\noindent
be the the index set of the global optimum, and let
\[
M_{n}(k) :=  \max\{ \avg(\vW_{\gl}) : \gl \in \sS_{n}(k) \} .
\]

\noindent
be its average value.  The following theorem characterizes the structure of the global optimum.
Note that in the first two results concerning the value $M_n(k)$, the value of  
$k$ is allowed to grow with $n$. 

\vskip.2in

\begin{thm}\label{thm:globalmax}
Let $\gl_n(k)$ and $M_{n}(k)$ be the index set and value of the globally optimum $k \times k$
submatrix of $\vW^n$, and let
$N = {n \choose k}^2$.
Let $a_{N}$ and $b_{N}$ be the scaling and centering constants in \eqref{eqn:an-def} and \eqref{eqn:bn-def}. 
\begin{enumeratea}
\item\label{item:maxa}  
There exists a constant $c>0$ such that as $n$ tends to infinity, for any sequence $k = k_n$ 
with $k \leq c \log n / \log\log n$, 
\[
a_N(kM_{n}(k) - b_N) \probd -\log T 
\]
where $T\sim\text{Exp}(1)$. 
\vskip.18in

\item\label{item:maxb} 
In general, if $k = k_n$ satisfies $c \log n / \log\log n \leq k \leq \exp(o(\log n))$ 
and $\go_n$ is any sequence tending to infinity, then

\[
\pr\left( \frac{-k \, \go_{n} (\log\log n)^2}{\log n} \, \leq \, a_{N}(k M_{n}(k) - b_N) \, \leq \, \go_{n} \right) \to 1
\]
as $n$ tends to infinity.
\vskip.18in

\item\label{item:maxc} 
For each fixed integer $k \geq 1$,
\[
\vW_{\gl_n(k)} - \avg(\vW_{\gl_n(k)}) \vone \vone^\prime 
\, \probd \, 
\vW^k - \avg(\vW^k) \vone\vone^\prime,
\] 
where $\vone$ is the $k \times 1$ vector of ones.
\vskip.18in
\end{enumeratea}
\end{thm}


The matrix $\vW^n$ contains only $n^2$ independent random variables.  In spite of this,
Part~\eqref{item:maxa} of Theorem~\ref{thm:globalmax} asserts that the average of the globally optimal
$k \times k$ submatrix has the same 
distributional asymptotics as the maximum of $N = {n \choose k}^2$ independent 
$\text{N}(0,k^{-2})$ random variables, provided that $k \leq c \log{n}/\log\log{n}$.  
(We expect that the same result holds if $k \ll \log n$, but the extension in this setting appears to require new ideas.) 
Part~\eqref{item:maxb} of the theorem ensures that the first order asymptotics of $M_n(k)$ 
remain unchanged as long as $\log k \ll \log n$. 
Part~\eqref{item:maxc} asserts that the joint distribution of $\vW_{\gl_n(k)}$ is the same 
as that of a $k \times k$ Gaussian random matrix once one 
subtracts their respective sample means.  In other words, asymptotically, the only thing remarkable
about the global maximum is its average value.

The proof of Theorem~\ref{thm:globalmax} relies on two auxiliary results.  The first is a combinatorial
bound, given in Lemma \ref{lem:varestimate} below, that includes refined second moment type 
calculations for the number of $k \times k$ submatrices having average greater than $b_N$.  
The second is the following Gaussian comparison lemma, which may be of independent interest. 

\vskip.1in

\begin{lem}\label{lem:compare}
Fix $N \geq 2$ and let $(X_{1}, \ldots,X_{N})$ be jointly Gaussian random variables with 
\[
\E(X_{i})=0 \ \ \E(X_i^2)=1 \ \mbox{ and } \ \E(X_iX_j) = \gs_{ij} \in (-1,1) \mbox{ for }1\leq i < j \leq N .
\] 
Let $Z_{1}, \ldots,Z_{N}$ be independent standard Gaussian random variables. For any $u \geq 1$,
\begin{align*}
& \left| \pr \left( \max_{1 \leq i \leq N} X_i \leq u \right) - \pr \left(\max_{1 \leq i \leq N} Z_i \leq u \right) \right| \\[.2in]
& \qquad
\leq \sum_{i \neq j} 2 \min\{1, |1-\theta_{ij}| \, u(1+(1\wedge\theta_{ij})u) \} \, 
\bar{\Phi}(u) \, \bar{\Phi}((1\wedge\theta_{ij})u) \\[.05in]
& \qquad\leq  \sum_{i \neq j, \gs_{ij} \neq 0}  
2 \sqrt{\frac{1+\gs_{ij}^{+}}{1-\gs_{ij}^{+}}}\cdot \bar{\Phi}(u) ^2 \cdot e^{u^2\gs_{ij}^+/(1+\gs_{ij}^+)}
\end{align*}
where $\theta_{ij}=\sqrt{({1-\gs_{ij}})/({1+\gs_{ij}})}$ and $x^{+}=\max\{x,0\}$. 
\end{lem}

\vskip.1in

We note that related Gaussian comparison results can be found in the literature (see~\cites{berman64, gal72,ross-book,ls02}).  The more precise upper bound of Lemma \ref{lem:compare} is needed here, in particular, 
to establish parts (\ref{item:maxa})) and (\ref{item:maxb})) of Theorem \ref{thm:globalmax} for
sequences $k_n$ that tend to infinity with $n$. In contexts where one has positive correlations and the second moment method is expected to give good information on the size of the maxima, the above bounds reduce even further. More precisely, let $M(u)=\sum_{i=1}^{N} \ind\{X_i \geq u\}$, where $X_1,\ldots, X_N$ are as in the statement of
Lemma \ref{lem:compare}.  If $\gs_{ij}\geq 0$ for all $i,j$, then $\theta_{i,j} \leq 1$ for all $i,j$ and 
for each $u \geq 1$ we have
\begin{align*}
0 
& \leq \, \pr\left( \max_{1 \leq i \leq N} X_i \leq u \right) - \pr\left( \max_{1 \leq i \leq N} Z_i \leq u \right) \\[.1in]
& \leq \, \sum_{i \neq j}  \bar{\Phi}(u) \, \bar{\Phi}(\theta_{ij} u) 
 \leq \, \sum_{i \neq j}  \pr( X_i \geq u, X_j \geq u ) 
 = \, \E(M(u)^2) - N \bar{\Phi}(u) .
\end{align*}
The first inequality above follows from Slepian's lemma, the second from the first inequality in Lemma \ref{lem:compare},
and the third from Lemma \ref{lem:biv-norm-tail}. Thus if one has good control on the second moment of $M(u)$, this shows that distributional asymptotics for the maxima are the same as in the \emph{i.i.d.}~regime. This is the path we shall follow.

\subsection{Two-point localization}
\label{sec:two-pt}
For fixed $k \geq 1$, Theorem~\ref{thm:globalmax} characterizes the growth of $M_n(k)$, 
the maximum average value of a $k \times k$ submatrix of $\vW^n$, with increasing 
dimension $n$.  As a dual consideration, one may fix a threshold
$\gt > 0$ and, for each $n$, study the largest $k$ for which there exists a $k \times k$ submatrix 
of $\vW^n$ with average greater than $\gt$.  Formally, define
\[
K_n(\tau) \ = \ \max\left\{ k : k \leq n \mbox{ and} \max_{\gl \in \sS_{n}(k)} \avg(\vW_{\gl})\geq \gt \right\} .
\]
\vskip.1in
\noindent
We extend the definition of the standard binomial coefficient
to non-integer valued arguments by defining 
\begin{equation}
\label{eqn:nfact-def}
	{n \choose x}:= \frac{n!}{\Gamma(x+1)\Gamma(n-x+1)}
\end{equation}
for $x \in [0,n]$, where $\Gamma(\ga):=\int_{0}^{\infty}x^{\ga-1}e^{-x}\; dx$ 
is the usual Gamma function. 
Consider the equation   
\begin{equation}
\label{eqn:tildk-def}
	{n \choose x}^2 \bar{\Phi}(x \gt) :=1. 
\end{equation}  
It is shown in \cite{sn10} that, for $n$ sufficiently large, there is a unique solution 
$\tilde{k}_n \ = \ \tilde{k}_{n}(\gt)$ of (\ref{eqn:tildk-def}), and that $\tilde{k}_n$ satisfies the relation  
\begin{equation}
\label{eqn:tildk-asymp}
	\tilde{k}_n = \frac{4}{\gt^2}\log \frac{e\gt^2 n}{4\log n} + 
	\left(\frac{4}{\gt^2} -1\right)\frac{\log\log n}{\log n} + O\left(\frac{|\log \gt|}{\gt^2\log n}\right) .
\end{equation}
It is shown in Theorem 1 of \cite{sn10} that the integer-valued random variable $K_n(\tau)$ lies in a finite interval 
around $\tilde{k}_n$, in particular
\[
-\frac{4}{\gt^2} - \frac{12\log 2}{\gt^2} - 4 \leq K_n(\tau) -\tilde{k}_n \leq 2 .
\]
eventually almost surely.
Here we refine this result to a two point localization, with a slightly weaker form of convergence. An almost sure convergence can be easily proved by using Borel-Cantelli lemma and the given probability estimates.
Also note that, similar results are known in the random graph literature, \eg\ for size of largest cliques \cite{bollobas-clique} and the chromatic number in random graphs \cite{an05}.  
Let $k_n^*$ denote the integer closest to $\tilde{k}_n$. 

\vskip.1in

\begin{thm}
\label{thm:loc}
For fixed $\gt > 0$, 
\[
\pr( K_n(\tau) \in \{ k_n^* - 1, k_n^* \} ) \to 1
\]
as $n$ tends to infinity. 
\end{thm}

\vskip.2in

%
%

\subsection{Local Optima and Iterative Search procedures}
\label{sec:loc-opt-res}
Finding the globally optimal $k \times k$ submatrix of a given data matrix is 
computationally prohibitive.  In practice, one often resorts to iterative search procedures 
that sequentially update a sequence of candidate submatrices in order to 
increase their average value.   The Large Average Submatrix (LAS) algorithm
(\cite{shabalin2009finding}) is a simple iterative search procedure for large average
submatrices that has proven effective in a number of genomic applications.  The basic
idea of the algorithm is this: if we restrict ourselves to a given set of $k$ columns, the 
optimal $k \times k$ submatrix can be found by computing the sum of each row over these
columns, and then choosing the $k$ rows with largest sum.  An analogous property holds
for a fixed set of $k$ rows.  The algorithm alternates between these two steps, alternately
updating rows and columns, until no further improvement in the average of the candidate
submatrix is possible.  A more detailed description follows.

\vskip.15in

\noindent
\begin{tabbing}
{\bf Output: } \= something \kill
{\bf Input:}	\> An $n \times n$ matrix ${\bf X}$ and integer $1 \leq k \leq n$.\\[.1in]

{\bf Loop:}	\> Select $k$ columns $J$ at random. Iterate until convergence.\\[.1in]

		\> Let $I:=$ $k$ rows with largest sum over columns in $J$.\\[.1in]
		
		\> Let $J :=$ $k$ columns with largest sums over rows in $I$.\\[.1in]
		
{\bf Output:}	\> Submatrix associated with final index sets $I$ and $J$.\\
\end{tabbing}

In practice, the iterative search procedure is applied with many choices of initial columns,
and the output submatrix with the largest average value is reported.
Submatrices to which the algorithm converges are locally optimal in the sense that
they cannot be improved by simple operations such as row or column swaps.  In particular,
their row and column sums dominate those in the strip defined by their column and 
row sets, respectively.  We make these notions more precise in the following definition.

\vskip.1in

\begin{defn}
Fix $1 \leq k \leq n$ and let $\gl = I \times J \in \sS_{n}(k)$.  
The sub-matrix $\vW_{\gl} := ((w_{ij}))_{i \in I, j \in J}$ is \emph{row dominant} in $\vW^n$ if
\[
\min_{i \in I} \Big\{ \mbox{$\sum_{j \in J} w_{ij}$} \Big\} 
\ \geq \ 
\max_{i \in [n] \setminus I} \Big\{ \mbox{$\sum_{j \in J} w_{ij}$} \Big\} 
\]
and is \emph{column dominant} in $\vW^n$ if 
\[
\min_{j \in J} \Big\{ \mbox{$\sum_{i \in I} w_{ij}$} \Big\} 
\ \geq \ 
\max_{j \in [n] \setminus J} \Big\{ \mbox{$\sum_{i \in I} w_{ij}$} \Big\} 
\]
A submatrix that is both row and column dominant in $\vW^n$ will be called \emph{locally optimal} in~$\vW^n$.
\end{defn}
\vskip.1in

It is easy to see that a $k \times k$ submatrix $\vW_{\gl}$ is locally optimal if and only if
it is a fixed point of the LAS search procedure, and the the LAS search procedure always
yields a local maximum.
Local optima represent natural ``extreme points'' of the set of $k \times k$ submatrices.  Understanding their behavior sheds light on the landscape of $k \times k$ submatrices, 
and the structure of the random matrices themselves.
The next result identifies the limiting average and distribution 
of a submatrix conditioned to be locally optimal.  In particular, we find the probability
that a given $k \times k$ submatrix is locally optimal. Before stating the main result, we will need some notation for the ANOVA decomposition of a matrix and define some random variables which arising in describing these distributional limits. 

Given any matrix $\vU = ((u_{ij}))$, we shall let $u_{i.}$ denote the average of row $i$, $u_{.j}$ denote the average of column $j$ and $u_{..}=\avg(\vU)$. Let  $\cA(\vU)$ be the Analysis of variance  (ANOVA) decomposition of the matrix $\vU$ namely 
\begin{equation}
\label{eqn:anova-decomp}
	\cA(\vU)_{ij} = u_{ij} - u_{i.}-u_{.j}+u_{..}, 1\leq i,j\leq k.
\end{equation}
Write this as 
\begin{equation}
\label{eqn:anova-r-c-av}
	\vU:= \avg(\vU)\vone\vone' + \bar{\vR}(\vU) + \bar{\vC}(\vU) + \cA(\vU),
\end{equation}
where $\bar{\vR}(\vU)$ denotes the matrix whose $i$-th row entires are all equal to $u_{i.} - \avg(\vU) $ for all $1\leq i\leq k$ and similarly $\bar{\vC}(\vU)$ denotes the matrix whose $i$-th column entries all correspond to $u_{.i} - \avg(\vU) $ while $\cA$ denotes the ANOVA operation on the entries of the matrix $\vU$ given in~\eqref{eqn:anova-decomp}.

For the statement of the result we will need the following random variables. 
\begin{enumeratei}
\item\label{iti}
Let $(G,T,T')$ be non-negative random variables with joint density 
\[
f(g,t,t') \propto ( \log(1+t/g) \log(1+t'/g))^{k-1} g^{k-1}e^{-t-t'-2g}.
\]
\item\label{itii}
$\vU=(U_{1}, \ldots,U_{k})$ and $\vV=(V_{1},\ldots,V_{k})$ are independent Dirichlet$(1,\ldots,1)$ 
random vectors, independent of $G, T, T^{\prime}$.
\end{enumeratei}

Let us now state the result. 
\begin{thm}[Structure Theorem for Locally Optimal Submatrices]
\label{thm:st1} 
Let $\cI_{k,n}$ be the event that $\vW^k$ is locally optimal in $\vW^n$.

\begin{enumeratea} 

\vskip.1in

\item\label{item:structurea}  
For fixed $k \geq 1$ 
\[
\pr(\cI_{k,n})= \frac{\theta_k}{{n\choose k}(\log n)^{(k-1)/2}} (1+o(1)) \text{ as } n\to\infty.
\]
Here
\begin{align}\label{eq:thetakval}
\theta_{k}:=\frac{k^{2k+1/2}}{2^{2k-1}\pi^{(k-1)/2} k!^2} \E( (\log(1+Y/G') \log(1+Y'/G') )^{k-1} )
\end{align}
\vskip.06in
\noindent
where $G',Y,Y'$ are independent, $G' \sim \mbox{Gamma}(k,2)$, and $Y,Y' \sim \mbox{Exp}(1)$.

\vskip.1in

\item\label{item:structureb}  
Let $a_n$ and $b_n$ be the scaling and centering constants  given in \eqref{eqn:an-def} and \eqref{eqn:bn-def}. Consider the ANOVA decomposition of the matrix $\vW^k$ as in \eqref{eqn:anova-r-c-av}. Then conditional on the event $\cI_{k,n}$ we have as $n\to\infty$,
\begin{align*}
\bigl.	\bigl(a_n(\sqrt{k}&\avg(\vW^k) - b_n), \sqrt{k} a_n \bar{\vR}(\vW^k), \sqrt{k} a_n \bar{\vC}(\vW^k),\cA(\vW^k) \bigr) \mid \cI_{k,n} \probd \\
&\left(-\log{G},~ \log(1+T/G) \begin{bmatrix} kU_1-1 \\\vdots\\ kU_k-1 \end{bmatrix} \vone',~ \log(1+T'/G) \, \vone \begin{bmatrix} kV_1-1 \\\vdots\\ kV_k-1\end{bmatrix}',~ \cA(\vW^k)   \right)
\end{align*}
where $G,T,T',\vU,\vV$ are as given in \eqref{iti} and \eqref{itii}.
\end{enumeratea}
\end{thm}
%

\noindent
{\bf Remark:} In words, the above Theorem implies that conditional on the event $\cI_{k,n}$ that $\vW^k$ is locally optimal, as $n\to\infty$ we have 
\begin{align*}
\vW^{k} \mid \cI_{k,n}\equald 
& \left( \frac{b_n}{\sqrt{k}} - \frac{\log G}{\sqrt{k}a_n}\right)\vone\vone' + \cA(\vW^{k})  \\
&+ \frac{\log(1+T/G)}{\sqrt{k}{a_n}} \begin{bmatrix} kU_1-1 \\ kU_2-1\\\vdots\\ kU_k-1\end{bmatrix}
 \vone' + \frac{\log(1+T'/G)}{\sqrt{k}{a_n}} \vone  \begin{bmatrix} kU'_1-1 \\ kU'_2-1\\\vdots\\ kU'_k-1\end{bmatrix}'+ o_{p}(1/a_n)
\end{align*}
As a simple corollary of the theorem, we see that all the entries in a typical locally optimal submatrix are 
concentrated around $\sqrt{2\log n/k}(1+o(1))$. However we will in fact crucially need the limiting structure of the re-centered row and column averages in order to establish the central limit theorem for the number of locally optimal submatrices below.

Note that local optimality is invariant under row and column permutations, and
therefore part \ref{item:structurea} of Theorem \ref{thm:st1} gives the probability 
that {\em any} fixed $k \times k$
submatrix is locally optimal in $\vW^n$: the focus on $\vW^k$ is a matter
of notational convenience.  Clearly
\begin{align*}
\pr({\mathcal I}_{k,n}) = 
\pr( \vW^k \mbox{ is row dominant}) \, \pr( \vW^k \mbox{ is column dominant} 
\mid \vW^k \mbox{ is row dominant} ).
\end{align*}
It is easy to see, by symmetry, that $\pr( \vW^k \mbox{ is row dominant}) = {n \choose k}^{-1}$. A priori, one might imagine since conditioning on the matrix $\vW^k$ being row dominant makes the entries of this matrix ``large'', that  
\[
\pr( \vW^k \mbox{ is column dominant} 
\mid \vW^k \mbox{ is row dominant} )\ \to\ c_k
\]
  for some constant  $c_k > 0$ as $n\to\infty$. 
However our argument shows that in fact
\[
\pr( \vW^k \mbox{ is column dominant} 
 \mid  \vW^k \mbox{ is row dominant} ) 
\ = \ 
\frac{\theta_k}{(\log n)^{(k-1)/2}} (1+o(1))
\]
The fact that this conditional probability tends to zero, rather than a positive constant, is 
somewhat unexpected. 

\subsection{The Number of Local Optima}

As a first step in understanding the overall landscape of $k \times k$ locally optimal submatrices, 
it is natural to consider the number of locally optimal submatrices in $\vW^n$.  

\begin{defn}
For $n \geq 1$ and $1 \leq k \leq n$ let
\begin{equation}\label{def1}
L_{n}(k):=\sum_{\gl\in\sS_{n}(k)} \ind\{\vW_{\gl} \text{ locally optimal in } \vW^n \},
\end{equation}
be the number of $k \times k$ locally optimal submatrices of $\vW^n$. 
\end{defn}

\vskip.15in

By symmetry, the probability that a given $k \times k$ submatrix of $\vW^n$ is locally optimal
is equal to $\pr(\cI_{k,n})$, and therefore $\E(L_n(k)) = {n \choose k}^2 \pr(\cI_{k,n})$. 
Thus Part~\eqref{item:structurea} of Theorem~\ref{thm:st1} immediately yields the following result.

\begin{thm}[Mean behavior]\label{thm:mean}
For each fixed $k \geq 1$, 
\vskip-.03in
\[
\E(L_{n}(k)) = \frac{\theta_k {n\choose k}}{ (\log n)^{(k-1)/2}} (1+o(1))
\] 
\vskip.08in
\noindent
as $n$ tends to infinity, where $\theta_k > 0$ is as in~\eqref{eq:thetakval}. 
\end{thm}

\vskip.1in

Intuitively this suggests that the running time of the {\bf LAS} algorithm can be bounded by a Geometric random variable with $p=p(n) = \theta_k/(\log{n})^{(k-1)/2}$, and thus converges in $\Theta_P((\log{n})^{(k-1)/2})$ steps, and thus gives conceptual insight on empirical observations on the running time of the algorithm.  Proving this at a rigorous level seems to be beyond the scope of the techniques in this paper.

The variance behavior of $L_{n}(k)$ is more delicate.  
In particular, assessing the variance of $L_n(k)$ 
requires a careful and detailed analysis of the joint probability that two given submatrices are locally 
optimal.  We do this by considering a series of cases, depending on the number of rows and
columns that the two submatrices have in common.  It is worth noting that the dominant term 
in the variance arises from submatrices having no common rows and columns: even in this case,
the local optimality of one submatrix will influence that of the other.

\vskip.15in

\begin{thm}[Variance behavior]\label{thm:var}
For each fixed $k\geq 1$, there exists $\nu_{k}\in (0,\infty)$ such that
\[
\var(L_{n}(k))= 
\begin{cases}
\nu_{1}n (1+o(1)) &\text{ for } k=1\\[.1in]
\nu_{k} n^{2k^2/(k+1)}(\log n)^{-k^2/(k+1)} (1+o(1)) & \text{ for } k\geq 2
\end{cases}
\] 
as $n$ tends to infinity.
\end{thm}

\vskip.1in

Using the results above, we may establish a connection between the average value of 
a typical local optima and the average of the global optimum.  For $c \in \bR$ let
\[
L_{n}(k:c) := \sum_{\gl \in\sS_{n}(k)} 
\ind\{\vW_{\gl} \text{ locally optimal in } \vW^n \ \text{and} \ \avg(\vW_{\gl}) \geq c \},
\]
be the number of locally optimal submatrices with average at least $c$.

\begin{cor}
\label{cor:glob-loc}
If $c_n$ is any sequence of positive numbers such that $c_n a_n \to \infty$, then
for each fixed $k \geq 1$,
\[
\frac{L_{n}(k: k^{-1/2} b_n - c_n) }{L_n(k)} \probc 1
\]
as $n$ tends to infinity.
\end{cor}

Our final result is the asymptotic normality of the random variable $L_{n}(k)$ using Stein's method.   Although 
$L_{n}(k)$ can be expressed as a sum of indicator variables,  
standard weak dependence conditions underlying existing applications of Stein's method do
not hold in this case.  The variance of $L_n(k)$ grows rapidly, in particular 
$\var(L_n(k)) / \E L_n(k) \to \infty$ as $n \to \infty$, so that $L_n(k)$ does not exhibit standard
Poisson scaling.  This is a consequence of the fact the local optimality of a $k \times k$ submatrix 
$\vW_\lambda$ affects the local optimality of {\bf every} other $k \times k$ submatrix, regardless of 
whether or not the other submatrix has any rows or columns in common with $\vW_\lambda$.  

\begin{thm}[Central Limit Theorem for $L_{n}(k)$]\label{thm:clt}
For any fixed $k\geq 1$, we have, 
\[
\tilde{L}_{n}(k):=\frac{L_{n}(k)-\E(L_{n}(k))}{\sqrt{\var(L_{n}(k))}}\probd \text{\upshape N}(0,1)
\]
as $n \to \infty$. Moreover  we have
\[
d_{\cW}(\tilde{L}_{n}(k), \text{\upshape N}(0,1))\leq 
\begin{cases}
c_1n^{-1/2} &\text{ for } k=1\\
n^{-\frac{2(k-1)}{(k+1)(k^2+2k-1)}+O(\log\log n/\log n)} &\text{ for } k\geq 2
\end{cases}
\]
where 

\[
d_{\cW}(W, Z) := \sup \set{\left|\E(g(W)) - \E(g(Z))\right|: g(\cdot) \ 1-\text{Lipschitz}} 
\]

\noindent
is the Wasserstein distance between the distribution of random variables $W$ and $Z$. 
\end{thm}

\vskip.1in

\noindent{\bf Remarks:}
The $k=1$ case of Theorem \ref{thm:clt} follows from existing work on number of local maxima of a random function on a graph (see~\cite[Theorem~3.1]{brs89}). For $k \geq 2$ this result is not applicable due to the dependency among the matrix averages.
We have not attempted to obtain the best rate of convergence in 
Theorem~\ref{thm:clt}: for $k \geq 2$ the given  rate is likely not optimal.  However, simulation results in Figure~\ref{fig:sim} for $k=2$ and 
$n \in \{100,200\}$ with $5000$ runs indicate fast convergence to the Gaussian limit.
\begin{figure}[hbtf]
\includegraphics[height=4cm]{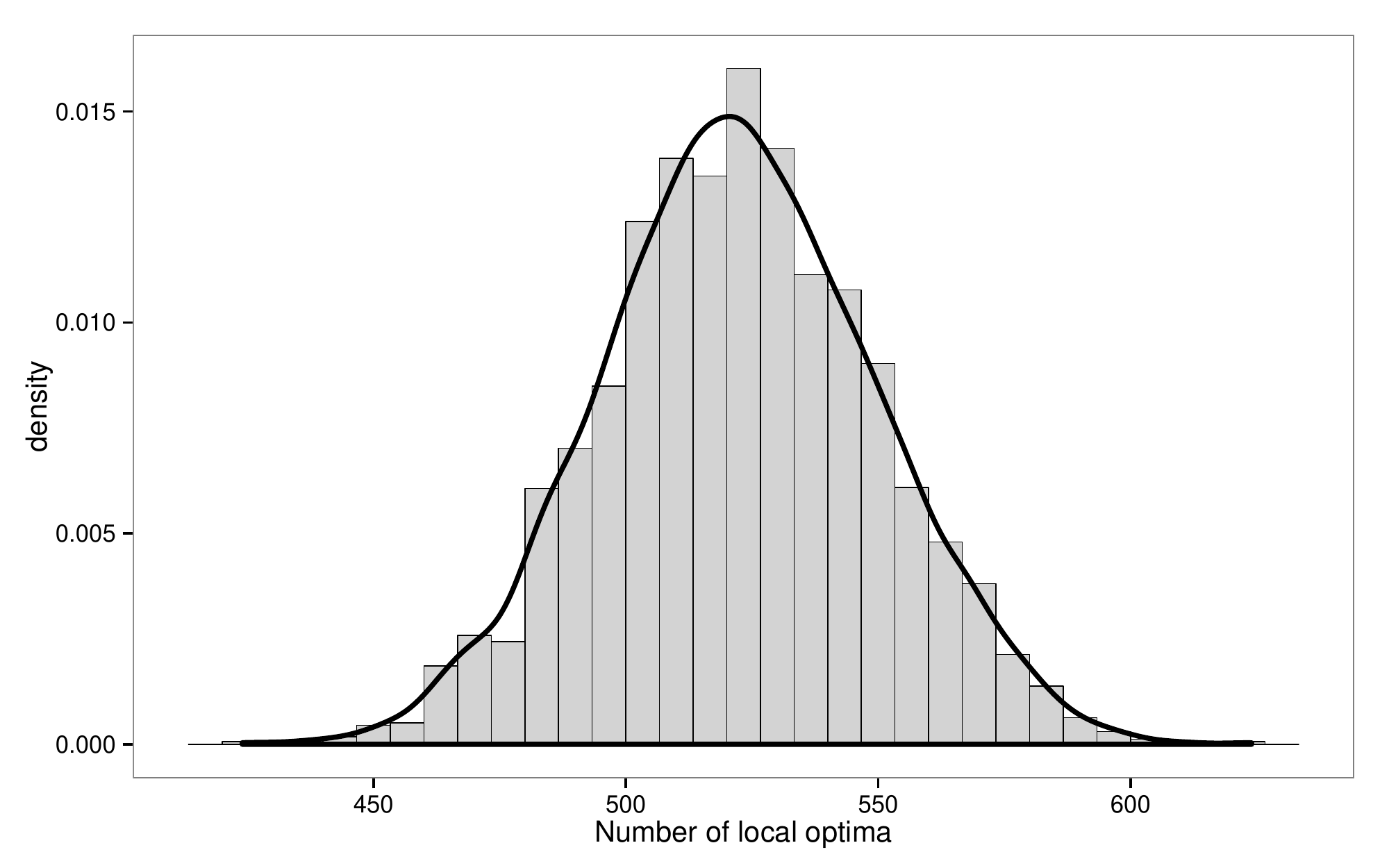}
\includegraphics[height=4cm,width=4cm]{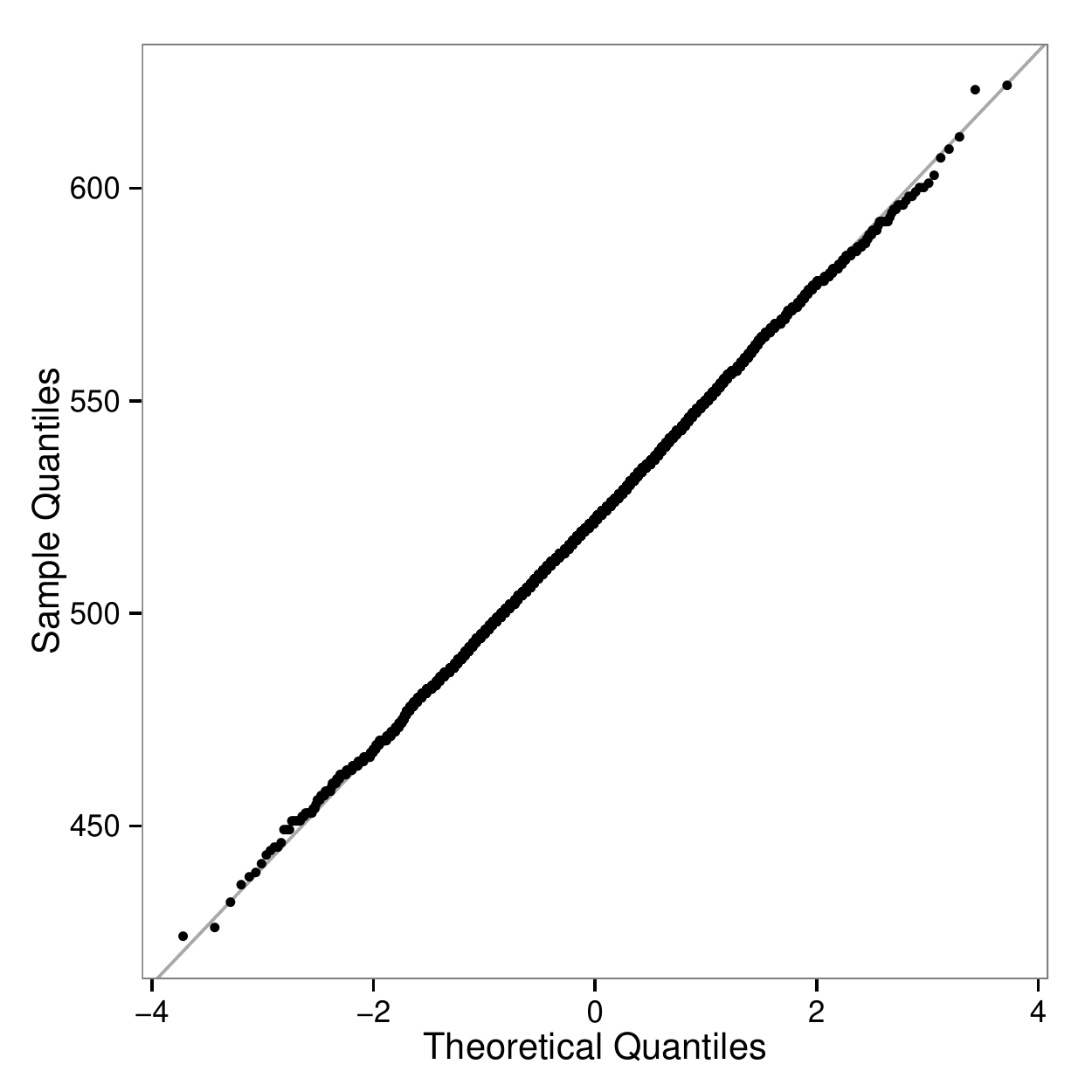}\\
\includegraphics[height=4cm]{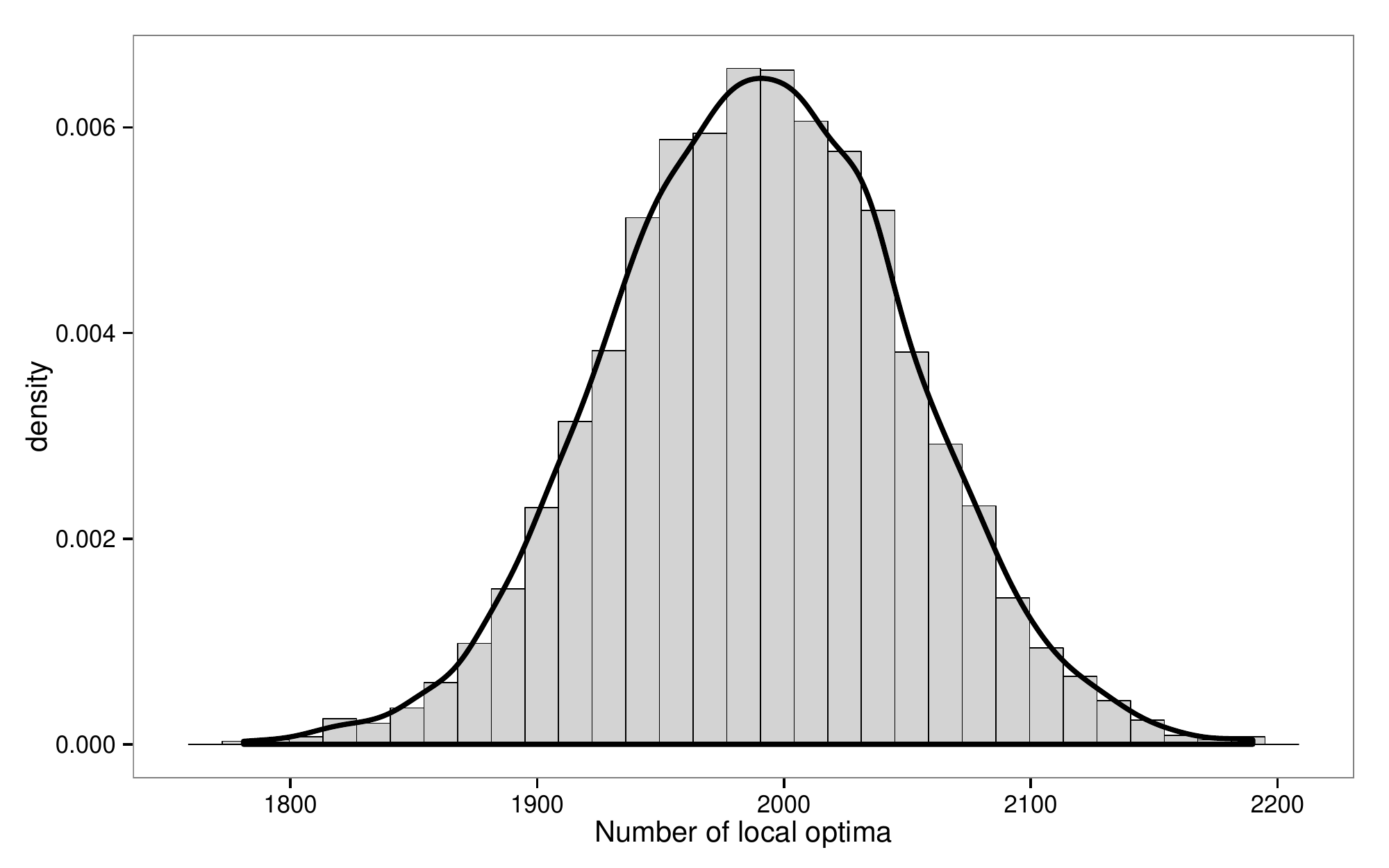}
\includegraphics[height=4cm,width=4cm]{n100_QQPlot.pdf}
\label{fig:sim}
\caption{Histogram and QQPlot for number of local optima for $k=2$ with $n=100$ (top row) and $n=200$ (bottom row) with $5000$ samples.}
\end{figure}


\subsection{Discussion}
\label{sec:open}
We now discuss the relevance of these results and related work. We start with a discussion of the general detection problem considered in this work and then expand on the techniques used in the paper. 

\subsubsection{ Finding large substructures}
As mentioned above, with the advent of large scale data in genomics, problems such as finding interesting structures in matrices has stimulated a lot of interest in a number of different communities, see e.g.~the survey \cite{madeira-survey}. In spirit, such problems are linked to another large body of work in the combinatorics community, namely the hidden clique problem  see e.g.~\cite{pittel} or~\cite{jerrum-clique} and the references therein. The simplest statement of the problem is as follows: Select a graph at random on $n$ vertices; consider the problem of detecting the largest clique (fully connected subgraph). For large $n$, it is known that the largest clique has $k(n) \sim 2\log_{2}{n}$ vertices (\cites{bollobas-clique,boll-book}). Theorem \ref{thm:loc} is very similar, in spirit to this result. However most greedy heuristics and formulated algorithms, short of complete enumeration, are only able to find cliques of size $\sim \log_{2}{n}$ and thus are off by a factor of $2$ from the optimal size. We see analogous behavior in our results; Theorem \ref{thm:globalmax}\eqref{item:maxa} implies that for fixed $k$, the average of the global optimum scales like $\sqrt{4\log{n}/k}$ whilst Theorem \ref{thm:st1} implies that the average of a typical local optima scales like $\sqrt{2\log{n}/k}$. 

\subsubsection{Planted detection problems} In the context of statistical testing of hypothesis, we have analyzed the energy landscape in the ``null'' case. One could also look at the ``alternative'' where there is some inherent structure in the data. In the last few years there has been a lot of interest in formulating statistical tests of hypothesis to distinguish between the null and the alternative,  see e.g.~\cite{candes-castro-durand} and~\cite{castro-zeitouni} for the general framework as well as application areas motivating such questions and see~\cite{berry-lugosi} and \cite{butucea2011detection} for a number of interesting general results in these contexts. In the context of the combinatorics, such questions result in the famous planted clique problem see e.g  \cites{alon-sudakov,dekel2010finding} and the references therein. 

\subsubsection{Energy landscapes} The notion of energy or fitness landscapes, incorporating a fitness or score to each element in a configuration and then exploring the ruggedness of the subsequent landscape,  arose in evolutionary biology, see~\cite{wright1932roles}, and for a nice survey, see~\cite{reidys2002combinatorial}. Our work has been partially inspired by the rigorous analysis of the NK fitness model (\cites{kauffman1989nk,weinberger1991local}) carried out in the probability community in papers such as~\cites{durrett-limic,evans-steinsaltz,limic-pemantle}. These questions have also played a major role in understanding deep underlying structures in spin glass in statistical physics, see e.g.~\cite{spin-glass-parisi}. For general modern accounts of the state of the art on combinatorial optimization in the context of random data and connections to other phenomenon in statistical physics, we refer the interested reader to~\cite{montan-book}. 

\subsection{Stein's method for normal approximations}
\label{sec:stein}
Stein's method~\cite{stein72} is a general and powerful method for proving distributional convergence with explicit rate of convergence. Here we briefly discuss the case of normal approximation. The standard Gaussian distribution can be characterized by the operator $\sA f(x):= xf(x)-f'(x)$ in the sense that, $X$ has standard Gaussian distribution iff $\E(\sA f(X))=0$ for all absolutely continuous functions $f$. Now to measure the closeness between a distribution $\nu$ and the standard Gaussian distribution $\nu_{0}$, one uses a separating class of functions $\cD$ to define a distance
\[
d_{\cD}(\nu,\nu_{0})=\sup_{h\in\cD}|\E h(X) - \E h(Z)|
\]
where $X\sim \nu,Z\sim \text{N}(0,1)$ and then attempts to show that the distance is ``small''. In this paper we will consider the $L^{1}$-Wasserstein distance in which case $\cD$ is the class of all $1$ Lipschitz functions. 

Stein's method consists of two main steps. The first step is to find solution to the equation $\sA f_{h}(x)=h(x)-\E h(Z)$ for $h\in\cD$. Assuming this can be performed, we have,
\[
\sup_{h\in\cD}|\E h(X) - \E h(Z)| \leq \sup_{f\in\cD'}|\E(Xf(X)-f'(X))|
\]
where $\cD'=\{f_{h}\mid h\in\cD\}$. The following lemma summarizes the the bounds required for Stein's method.

\begin{lem}[\cite{stein72}]\label{lem:stein}
For any $1$-Lipschitz function $h$, there is a unique function $f_{h}$ such that $\sA f_{h}=h - \E h(Z)$. Moreover we have
\[
|f_{h}|_{\infty} \leq 1, |f'_{h}|_{\infty}\leq \sqrt{2/\pi} \text{ and } |f''_{h}|_{\infty}\leq 2.
\]
\end{lem}

Thus to prove that the distribution of $X$ is close to standard Gaussian distribution it is enough to prove that
\[
\sup_{f\in \cD'}|\E f'(X) - \E Xf(X)| 
\]
is small where \begin{equation}
\label{eqn:class-fn-stein}
	\cD'=\{f\mid  |f_{h}|_{\infty} \leq 1, |f'_{h}|_{\infty}\leq \sqrt{2/\pi} \text{ and } |f''_{h}|_{\infty}\leq 2\}.
\end{equation} 
This final portion is very much problem dependent and is often the hardest to accomplish. A number of general techniques have now been formulated, \eg\ exchangeable pair approach, dependency graph approach, size-bias transform, zero-bias transform etc.~that can be used for a large class of problems. We refer the interested reader to the surveys~\cites{MR2235448,MR2732624,MR2118599,ross12} and the references therein. However, in our case because of the high degree of dependency, the above mentioned methods are difficult to apply and we develop a new variant to bound the error.

\subsubsection{Open questions}
For the sake of mathematical tractability, we assumed that the underlying matrix had gaussian entries. It would be interesting to extend this analysis to general distributions. The exact statement of the results will be different since extremal properties of the gaussian distribution play a significant role in the proofs of the main results. The results in the  paper also suggest a host of extensions and new problems. Theorem \ref{thm:globalmax} deals with the global optimum in the regime where $\log{k} = o(\log{n})$. Extending this further, especially to the regime where $k=\alpha n$ for some $0<\alpha<1$ would be quite interesting and will require new ideas; one expects that the comparison to the independence regime using Lemma \ref{lem:compare} breaks down at this stage. We also expect behavior similar  to the extrema of branching random walk (\cites{aidekon11} and references within) in this regime.  Extending the local optima results to a regime $k=k(n)\to\infty$ as opposed to the fixed $k$ regime would be interesting. This would be especially relevant in the context of detecting matrices with average above a particular threshold which by Theorem \ref{thm:loc} corresponds to the $k(n)=C\log{n}$ regime. Finally this work fixes $k$ and then tries to find submatrices with large average. It would be interesting to develop algorithms which allow one to increase $k$ to achieve large submatrices with average above a threshold $\tau$.

\section{Preliminary Results}
\label{sec:est} 
In this section we present several technical lemmas that will be used in the proofs of the main results. We urge the reader to skim these results and come back to them as and when they are used.  Lemma \ref{lem:gauss_tail} and Lemma \ref{lem:extreme} collect standard results about tails and extreme value theory for the standard normal distribution.   Lemma \ref{lem:condz} provides  estimates of normal tail probabilities arising in the extreme value regime, while Lemma \ref{lem:tail} derives tail bounds and conditional distributions for the difference between the average and the minimum of $k$ independent standard Gaussian random variables. In Section~\ref{ssec:gcomp} we prove the Gaussian Comparison Lemma~\ref{lem:compare}. We conclude the section with some combinatorial estimates. 

\subsection{Gaussian tail bounds}
The following classical bound on the tail probabilities of the standard Gaussian, see e.g.~\cite{ross-book}, will be used repeatedly in what follows. 

\begin{lem}\label{lem:gauss_tail}
For each $x > 0$, we have
\[
\frac{xe^{-x^2/2}}{\sqrt{2\pi}(1+x^2)} \ \leq \ \bar{\Phi}(x) \ \leq \ \frac{e^{-x^2/2}}{\sqrt{2\pi}x}.
\]
Moreover, $xe^{x^2/2}\bar{\Phi}(x)$ is an increasing function for $x>0$.
\end{lem}

\subsection{Extreme values}
  
Let $Z_{1}, Z_{2}, \ldots, Z_{N}$ be $N$ independent standard Gaussian random variables, and 
let $Z_{(1)} \leq Z_{(2)} \leq \cdots \leq Z_{(N)}$ be their ordered values.  
We will make use of the following standard result, see e.g.~\cite{ross-book}.

\vskip.1in

\begin{lem}\label{lem:extreme}
Let $\ell \geq 0$ be any fixed integer.   Then as $N$ tends to infinity,
\[
a_N (Z_{(N)} - b_N, Z_{(N-1)} - b_N,\ldots, Z_{(N-\ell)} - b_N) \Rightarrow (V_{1}, V_{2}, \ldots,V_{\ell}) ,
\]
where $V_{i} = - \log(T_{1}+T_{2}+\cdots+T_{i})$, and $T_1, \ldots, T_\ell$ are independent Exp$(1)$ random variables. 
\end{lem}

The next lemma analyzes properties of  conditional distribution of a standard Gaussian conditioned to be large. 

\begin{lem}\label{lem:condz}
Let $Z$ be a standard Gaussian random variable and let $\theta>0$ be a fixed real number. 
Let the scaling and centering constants $a_{n}$ and $b_{n}$ be as in \eqref{eqn:an-def} 
and \eqref{eqn:bn-def}. 
Define $\cB_n(x)$ to be the event $\{Z\geq \sqrt{\theta}(b_n + a_n^{-1}x)\}$.

\begin{enumeratea}

\item\label{item:condza}
If $c_n = o(a_n)$, then $n^{\theta} \, (\sqrt{2\pi}b_n)^{1-\theta} \, e^{x\theta} \, \pr(\cB_n(x)) \to \theta^{-1/2}$ 
uniformly for $x$ with $|x| \leq c_n$. 

\item\label{item:condzb} 
Let $x \in \dR$.  Conditional on the event $\cB_n(x)$, the random variable $a_n(Z/\sqrt{\theta} - b_n - a_n^{-1}x)$ converges in distribution to an Exp$(\theta)$ random variable.
\end{enumeratea}
\end{lem}

\begin{proof}
(\ref{item:condza}) It follows from Lemma~\ref{lem:gauss_tail} and elementary algebra that
\begin{align*}
& n^{\theta}(\sqrt{2\pi}b_n)^{1-\theta} e^{x\theta} \pr(\cB_n(x)) \\[.08in]
& = \, n^{\theta}(\sqrt{2\pi}b_n)^{1-\theta}e^{x\theta}\pr(Z\geq \sqrt{\theta}(b_n + a_n^{-1}x)) \\[.08in]
& = \, \theta^{-1/2} \left( \frac{n}{\sqrt{2 \pi} b_n} \right)^\theta 
e^{x \theta} \exp(- \theta (b_n + a_n^{-1}x)^2 / 2) (1 + o(1)) \\[.08in]
& = \, \theta^{-1/2} \left( \frac{n}{\sqrt{2 \pi} b_n} \right)^\theta \exp\{-\theta b_n^2 / 2 \} 
\exp\{ \theta x (1 - b_n / a_n) \} \exp\{ \theta (x / a_n)^2 \} (1 + o(1))
\end{align*}
The fourth and fifth terms above tend to one with increasing $n$ by definition of $a_n$ and $b_n$, 
and our assumptions on $x$.  
A straightforward calculation shows that $n(\sqrt{2\pi} b_n)^{-1}e^{-b_n^2/2}$ tends to one as $n \to \infty$,
and therefore the product of the second and third terms above tends to one as well.   

\vskip.1in

\noindent(\ref{item:condzb}) The claim follows from the fact that for each $t \geq 0$, as $n \to \infty$, \\

$\qquad\displaystyle
\pr(a_{n}(Z/\sqrt{\theta}- b_n - a_n^{-1}x) \geq t \mid \cB_n(x) )
\ = \ \frac{\pr(\cB_n(x+t))}{\pr(\cB_n(x))} 
\ \to \ e^{-\theta t} .
$
\end{proof}

In order to analyze the asymptotic behavior of the expected number of local optima $\E(L_n(k))$, we need to  understand the way in which the minimum of a set of independent Gaussian random variables deviates from its the sample mean under various conditioning events. The next lemma establishes the relevant asymptotic 
results.

\vskip.1in

\begin{lem}\label{lem:tail}
Let $Z_{1},\ldots,Z_{k}$ be independent standard Gaussian random variables with sample mean
$\bar{Z}=k^{-1}\sum_{i=1}^{k}Z_{i}$ and minimum $Z_{\min}=\min_{1\leq i\leq k} Z_{i}$.

\vskip.1in

\begin{enumeratea}

\item\label{item:taila} The random variable $\bar{Z}-Z_{\min}$ is non-negative and its cumulative distribution function
$F(x) = \pr( \bar{Z}-Z_{\min} \leq x) = \alpha_{k} x^{k-1}(1+o(1))$ as $x \downarrow 0$, where $\alpha_k > 0$
is given by
\begin{equation}
\label{alphak}
\alpha_{k} = \frac{k^{k+1/2}}{k!(2\pi)^{(k-1)/2}}.
\end{equation}

\vskip.1in

\item\label{item:tailb} 
For $\eps > 0$ let $\cB_{\eps}$ be the event $\{\bar{Z}-Z_{\min}\leq \eps\}$.  Then as $\eps \downarrow 0$,
\[
{\mathcal L} \{ \eps^{-1}(\bar{Z}-Z_1,\ldots, \bar{Z}-Z_k) \mid \cB_{\eps} \}
\, \probd \, 
(1-kU_1,\ldots, 1-kU_k )
\]
where $\vU=(U_{1},\ldots,U_{k})$ has a Dirichlet$(1,\ldots,1)$ distribution, \ie\ 
$\vU$ is uniformly distributed on the simplex 
$\{ (x_1,\ldots,x_{k})\mid x_{1}+\cdots+x_{k}=1, x_{1} \ldots, x_{k} \geq 0 \}$.

\vskip.2in

\item\label{item:tailc} There exists a positive constant $g_k > 0$ such that 
\[
\pr(\bar{Z}-Z_{\min}\geq x) = \frac{g_k}{x}\exp\left(-\frac{kx^2}{2(k-1)}\right)(1+o(1)) \text{ as } x \uparrow \infty.
\]

\vskip.1in

\end{enumeratea}
\end{lem}
\begin{proof}
\eqref{item:taila}  
Clearly, $\bar{Z}-Z_{\min}$ is non-negative, and it is easy to see that 
\[
k^{-1}(Z_{\max}-Z_{\min}) \, \leq \, \bar{Z}-Z_{\min} \, \leq \, Z_{\max}-Z_{\min}
\] 
where $Z_{\max}=\max_{1 \leq i \leq k} Z_{i}$.  Thus for all $x \geq 0$,  
\begin{equation}
\label{maxmin}
\pr(Z_{\max}-Z_{\min}\leq x) \, \leq \, \pr(\bar{Z}-Z_{\min}\leq x) \, \leq \, \pr( Z_{\max}-Z_{\min} \leq kx) 
\end{equation}
One may readily verify that 
\begin{align*}
\pr(Z_{\max}-Z_{\min}\leq x) = \int_{-\infty}^{\infty} k (\Phi(t+x)-\Phi(t))^{k-1} \phi(t) dt
\end{align*}
where $\phi(t)$ is the standard normal density. The integral above behaves like a constant $\alpha_k$
times $x^{k-1}$ as $x\downarrow 0$, and the first claim follows from (\ref{maxmin}).

We now evaluate the value of the constant $\alpha_k$.  Note that the $F(x)$ is continuous and 
that for $t \geq 0$,
\[
\E(\exp\{-t(\bar{Z}-Z_{\min})\})
\ = \ 
\int_0^\infty F(t^{-1} x) \, e^{-x} \, dx .
\]
The last equation and the behavior of $F(\cdot)$ near zero imply that
\begin{equation*}
\alpha_{k} \ = \ \lim_{t \to\infty} \frac{t^{k-1}}{(k-1)!} \E(\exp\{-t(\bar{Z}-Z_{\min})\}) .
\end{equation*}
A standard covariance calculation shows that $\bar{Z}$ is independent of  
$(\bar{Z}-Z_1,\ldots, \bar{Z}-Z_k)$, and therefore $\bar{Z}$ is independent of  
$\bar{Z}-Z_{\min}$.  It follows that for $t \geq 0$,
\begin{align}
\E(\exp\{-t(\bar{Z}-Z_{\min})\})
& = \E(e^{tZ_{\min}}) / \E(e^{t\bar{Z}}) \notag\\[.1in]
&= ke^{-t^2/2k}\int_{\dR} e^{-tx} \Phi(x)^{k-1} \phi(x) dx \notag\\[.1in]
&= ke^{(k-1)t^2/2k}\int_{\dR} \Phi(x-t)^{k-1}\phi(x) dx \label{eq:mgf}
\end{align}
where we have used the fact that $\E(e^{t\bar{Z}})=e^{t^2/2k}$ and 
$\E(e^{tZ_{\min}})= \int_{\dR} ke^{-tx} \Phi(x)^{k-1}\phi(x) dx$.
Note that 
\begin{align}
\int_{\dR} \Phi(x-t)^{k-1}\phi(x) \, dx 
&= \int_{-\infty}^{(1-1/2k)t} \Phi(x-t)^{k-1}\phi(x) \, dx +  \int_{(1-1/2k)t}^{\infty} \Phi(x-t)^{k-1} \phi(x) \, dx \notag \\[.1in]
&= \int_{-\infty}^{-t/2k} \Phi(x)^{k-1}\phi(x+t) \, dx +  O\big(\bar{\Phi}((1-1/2k)t) \big).\label{eq:2terms}
\end{align}
From \eqref{alphak}, \eqref{eq:mgf}, \eqref{eq:2terms} and the fact that
\[
\bar{\Phi}(x)=\Phi(-x)= \frac{1}{\sqrt{2\pi}x}e^{-x^2/2}(1-O(x^{-2})) \text{ for } x\to\infty
\]
we find that
\begin{align*}
\ga_{k}&= \lim_{t\to\infty} \frac{ke^{(k-1)t^2/2k}t^{k-1}}{(k-1)!}
\left(\int_{-\infty}^{-t/2k} \Phi(x)^{k-1} \phi(x+t) dx +  O(\bar{\Phi}((1-1/2k)t)\right) \\[.1in]
&=\lim_{t \to\infty} \frac{k^2e^{(k-1) t^2/2k} t^{k-1}}{k!(2\pi)^{k/2}}
\left(\int_{-\infty}^{-t/2k} |x|^{-(k-1)} e^{- \{(k-1)x^2/2 + (x+t)^2/2 \}}dx +  O\big(t^{-1}e^{-(1-1/2k)^2t^2/2} \big) \right) \\[.1in]
&=\lim_{t\to\infty} \frac{k^2}{k!(2\pi)^{k/2}}
\left(\int_{-\infty}^{-t/2k} (t/|x|)^{k-1}e^{-k(x+t/k)^2/2}dx +  O\Big( t^{k-2} e^{ \{(k-1)t^2/2k-(2k-1)^2t^2/8k^2 \} } \Big) \right) \\[.1in]
&=\lim_{t\to\infty} \frac{k^2}{k!(2\pi)^{k/2}}
\left(\int_{-\infty}^{t/2k} (1/k-x/t)^{-k+1}e^{-kx^2/2}dx +  O(t^{k-2}e^{-t^2/8k^2})\right) \\[.1in]
&= \frac{k^2}{k!(2\pi)^{k/2}} \int_{-\infty}^{\infty} k^{k-1}e^{-kx^2/2}dx 
=\frac{ k^{k+1/2}}{k!(2\pi)^{(k-1)/2}}.
\end{align*}
as desired.

\vskip.16in

\noindent\eqref{item:tailb}  
Fix $\eps> 0$ for the moment and write $\pr_\eps$ for the conditional distribution of 
$\eps^{-1}(\bar{Z}-Z_1,\bar{Z}-Z_2,\ldots, \bar{Z}-Z_k)$ given $\cB_{\eps}$. 
The distribution $\pr_\eps$ is supported on the simplex 
$\gD_{k}: = \{ (x_1,\ldots,x_{k}) \mid x_{1}+\cdots+x_{k} = 0, x_{i} \leq 1\}$ with extreme points $v_1,\ldots,v_k$, where
$v_{i}:=(1,\ldots,1, 1-k,1,\ldots,1)$ with $1-k$ in the $i$-th position. 
Note that $\pr_eps$ is invariant under coordinate permutations and that, on $\cB_\eps$, each of $Z_1,\ldots,Z_k$ is contained 
in a common interval of length $k \eps$. Clearly $\set{\pr_\eps: \eps> 0}$ is a tight family of probability measures on $\bR^k$.  
The properties above ensure that every subsequential limit of $\pr_\eps$ as $\eps\downarrow 0$ is translation invariant, 
and hence uniform, on $\gD_{k}$.  On the other hand, given a Dirichlet$(1,\ldots,1)$ random vector 
$\vU=(U_{1},\ldots,U_{k})$ the sum $\sum_{i=1}^{k}U_{i}v_{i}=(1-kU_1,1-kU_2,\ldots, 1-kU_k )$ is uniformly distributed 
on the simplex $\gD_{k}$. \\

\noindent\eqref{item:tailc} 
The relation $\bar{Z} - Z_{\min} = \max_{1 \leq i \leq k}\{\bar{Z} - Z_i\}$ implies that
\[
\pr(\bar{Z}-Z_1\geq x)\leq \pr(\bar{Z}-Z_{\min}\geq x) \leq k \pr(\bar{Z}-Z_1\geq x).
\]
The claim now follows from the fact $\bar{Z}-Z_{1}$ is normal with mean zero and variance $(k-1)/k$, 
and that
\[
\pr(Z\geq x) = \frac{e^{-x^2/2}}{\sqrt{2\pi} x} (1+o(1))
\]
as $x \uparrow \infty$.
\end{proof}

\vskip.1in

\subsection{Maxima of two correlated gaussian r.v.s}
Let $(Z,Z_\rho)$ be a bivariate gaussian random vector with $\E(Z)=\E(Z_\rho)=0$, 
$\var(Z) = \var(Z_\rho) = 1$, and $\E(ZZ_\rho) = \rho \geq 0$.  Several of our results require 
bounds on the conditional probability $\pr(Z_\rho> x \mid Z > x)$ when $x$ is large.
Without loss of generality, assume that $Z_\rho = \rho Z + \sqrt{1-\rho^2}Z'$ where 
$Z'$ is an independent copy of $Z$.   An argument like that in 
Lemma~\ref{lem:condz}(b) shows that, conditional on the event $\cA = \{Z>x\}$, 
the random variable $x(Z-x)$ is tight, and in particular, $Z$ is concentrated around $x$.
Thus, conditional on $\cA$, the event $\{ \rho Z + \sqrt{1-\rho^2}Z' > x \}$ 
is roughly the same as $\{ Z' > \theta x \}$ with $\theta = \sqrt{(1-\rho)/(1+\rho)}$.
The following result from \cite{MR2104113} makes these ideas precise.

\begin{lem}
\label{lem:biv-norm-tail}
Let $Z,Z'$ be independent standard Gaussian random variables.  For any $\rho\in [0,1]$ and $x>0$ we have 

\[
\bar{\Phi}(\theta x) \leq \pr(\rho Z + \sqrt{1-\rho^2}Z'> x \mid Z> x) \leq (1+\rho) \bar{\Phi}(\theta x) 
\]
\vskip.1in
\noindent
where $\theta = \sqrt{(1-\rho)(1+\rho)}$.
\end{lem}
\vskip.2in

\subsection{Gaussian Comparison Lemma: Proof of Lemma~\ref{lem:compare}}\label{ssec:gcomp}
Let $\gS_{1} = \{ \gs_{ij} : 1 \leq i,j \leq N \}$ be the covariance matrix of the random
vector $(X_1,\ldots,X_N)$, and let
$\gS_{0}$ be the $N \times N$ identity matrix. 
Let $\vX^0 \sim \text{N}(\mvzero,\gS_{0})$ and $\vX^1 \sim \text{N}(\mvzero,\gS_{1})$ be independent
random vectors.  For $t \in [0,1]$ define
\begin{equation}
\label{eqn:xt-def}
	\vX^{t} := \sqrt{t}\vX^{1}+\sqrt{1-t} \vX^{0} .
\end{equation} 
Note that $\vX^{t} \sim \text{N}(\mvzero,\gS_t)$,
where $\gS_{t}=t\gS_{1}+(1-t)\gS_{0}$.

Let $G(\vx)$ be a smooth function of $N$ variables $\vx=(x_1,x_2,\ldots,x_N)$. 
Let $G_{i}(\vx) = (\partial G / \partial x_i) (\vx)$ and 
$G_{ij}(\vx) = (\partial^{2} G / \partial x_i \partial x_j)(\vx)$ denote the first
and second order partial derivatives of $G$.  We claim that
\begin{equation}\label{diffegx}
\E[G(\vX^1)] - \E[G(\vX^0)] \ = \ \sum_{i<j} \gs_{ij} \int_0^1 \E(G_{ij}(\vX^t)) \, dt.
\end{equation}
To see this, note that $\vX^{t}\equald \gS_{t}^{1/2}\vX^{0}$, and therefore
\begin{align*}
\E G(\vX^1)  - \E G(\vX^0)  
&= \int_{0}^{1} \frac{d}{dt} \E(G(\gS_{t}^{1/2}\vX^{0})) \, dt \\
&= \sum_{i,j=1}^{N}\int_{0}^{1} \E\left(G_i(\gS_{t}^{1/2}\vX^{0}) \, \frac{d}{dt}(\gS_{t}^{1/2})_{ij} \, X^0_j \right) dt \\
&= \sum_{i,k=1}^{N}\int_{0}^{1} \E\biggl( G_{ik}(\gS_{t}^{1/2}\vX^{0}) \, 
              \sum_{j=1}^{k} (\gS_{t}^{1/2})_{kj} \, \frac{d}{dt}(\gS_{t}^{1/2})_{ij} \biggr) dt
\end{align*}
where the last equality follows by conditioning and Gaussian integration by parts. Using the 
symmetry of the matrix $\gS^{1/2}_{t}$ and simplifying we have
\begin{align*}
\E(G(\vX^1)) - \E(G(\vX^0))  
&=\frac{1}{2}\sum_{i,k=1}^{N}\int_{0}^{1} \E\left( \bigr( 2 \gS_{t}^{1/2}\frac{d}{dt}(\gS_t^{1/2})\bigl)_{ik}  G_{ik}(\vX^t) \right) dt \\[.1in]
&=\frac{1}{2}\sum_{i,k=1}^{N}\int_{0}^{1} \E( (\gS_1-\gS_0)_{ik}  G_{ik}(\vX^t) ) \, dt
\ = \, \sum_{i<j} \gs_{ij}\int_0^1 \E(G_{ij}(\vX^t)) \, dt,
\end{align*}
where in the second line we have used the fact that 
$2 \gS_{t}^{1/2} \, d(\gS_{t}^{1/2}) / dt=d(\gS_{t}) / dt=\gS_{1}-\gS_{0}$.
This establishes (\ref{diffegx}).  

Fix $\eps > 0$ for the moment, and let $G^{\eps}(\vx) = \prod_{i=1}^{N} \Phi (\eps^{-1}(u-x_i))$, 
where $\Phi$ is the CDF of the standard Gaussian.  Let $X_i^t$ be the $i$'th component of $\vX^{t}$. 
Applying equation (\ref{diffegx}) to $G^\eps$ yields the inequality
\begin{align*}
|\E(G^\eps(\vX^1)) &- \E(G^\eps(\vX^0))| \\
& \leq \ \sum_{i<j} |\gs_{ij}| \int_0^1 \E(\eps^{-2} \phi(\eps^{-1}(u-X_i^{t})) \, \phi(\eps^{-1}(u-X_j^{t})) ) dt\\
& = \ \sum_{i<j} |\gs_{ij}| \int_0^1 \E( f_{ij}^{t}(u+\eps Z_1,u+\eps Z_2)) dt
\end{align*}
where $f_{ij}^{t}(x,y)$ is the joint density of $(X_{i}^{t},X_{j}^{t})$ and $Z_{1},Z_{2}$ 
are independent Gaussian random variables.  
Letting $\eps$ tend to zero, and using the fact that 
\[
\lim_{\eps \to 0} G^\eps(\vx) \ = \ \ind\set{\max_{1 \leq i \leq N} x_i \leq u }
\] 
if $x_i \neq u_i$ for $i = 1,\ldots, N$, we find that
\vskip-.05in
\[ 
\left| \pr \left( \max_{1 \leq i \leq N} X_i \leq u \right) - \pr \left( \max_{1 \leq i \leq N} Z_i \leq u \right) \right| 
\ \leq \ 
\sum_{i<j} |\gs_{ij}| \int_0^1 f_{ij}^{t}(u,u) \, dt.
\]

\vskip.1in

To complete the proof, we analyze a typical term in the previous display.  For fixed $i < j$ and $u \in \bR$ we have   
\[
f_{ij}^{t}(u,u) \ = \ \frac{e^{-u^2/(1+\gs_{ij} t)}}{2\pi\sqrt{1-\gs_{ij}^2t^2}}
\]
and therefore
\begin{align*}
|\gs_{ij}| \int_0^1 f_{ij}^{t}(u,u) dt
\ \leq \ 
\left| \int_{0}^{\gs_{ij}} \frac{e^{-u^2/(1+t)}}{2\pi\sqrt{1- t^2}} \, dt \right|.
\end{align*}
Making the change of variable $x = \sqrt{(1-t)/(1+t)}$ we find that
\begin{align*}
\left| \int_{0}^{\gs_{ij}} \frac{e^{-u^2/(1+t)}}{2\pi\sqrt{1- t^2}} \, dt \right| \ 
&\leq \ \pi^{-1}e^{-u^2/2}\left| \int_{\theta_{ij}}^{1} \frac{e^{-(ux)^2/2} }{1+x^2}  dx \right|\\[.1in]
&\leq \ \pi^{-1}e^{-u^2/2}\left| \int_{\theta_{ij}}^{1} e^{-(ux)^2/2} dx \right|
\leq \ 2\bar{\Phi}(u)  |\bar{\Phi}(u) - \bar{\Phi}(\theta_{ij}u)|
\end{align*}
where $\theta_{ij}=\sqrt{(1-\gs_{ij})/(1+\gs_{ij})} \geq 0$.  
Considering separately the case $\theta_{ij} \leq 1$ and $\theta_{ij} > 1$,
the concavity of $\Phi(x)$ for $x \geq 0$, and the inequality 
$\bar{\Phi}(x)\geq \phi(x)/(1+x)$ yield
\begin{align*}
|\bar{\Phi}(u) - \bar{\Phi}(\theta_{ij}u)| \ 
& \leq \ 
\min\{ \bar{\Phi}(\theta_{ij}^1 u), |1-\theta_{ij}| u \phi(\theta_{ij}^1 u) \} \\[.1in]
& \leq \ \bar{\Phi}(\theta^{1}_{ij}u) \min\{ 1, |1-\theta_{ij}|u(1+\theta^{1}_{ij}u) \}
\end{align*}
where $\theta_{ij}^{1}=\min\{\theta_{ij},1\}$ and hence
\begin{align}\label{eq:integrated}
	\left| \int_{0}^{\gs_{ij}} \frac{e^{-u^2/(1+t)}}{2\pi\sqrt{1- t^2}} \, dt \right| \ 
&\leq 2\bar{\Phi}(u) \bar{\Phi}(\theta^{1}_{ij}u) \min\{ 1, |1-\theta_{ij}|u(1+\theta^{1}_{ij}u) \}.
\end{align}
This completes the proof of the first inequality in Lemma~\ref{lem:compare}.  The second inequality follows from the fact that $xe^{\frac{1}{2}x^2}\bar{\Phi}(x)$ is an increasing function for $x\geq 0$ and thus $\bar{\Phi}(\theta u)\leq \bar{\Phi}(u)  \theta^{-1}e^{\frac{1}{2}(1-\theta^2)u^2}$ for all $\theta\in[0,1]$.\hfill\qed

\subsection{Combinatorial estimates}
For $n \geq 1$ and $1 \leq k \leq n$ let $(n)_{k}:={n!}/{(n-k)!}$. The following bound follows easily from Stirling's approximation. 

\begin{lem}\label{lem:comb}
For $n \geq 1$ and $1 \leq k \leq \sqrt{n}$, 
\[
\frac{(n)_k}{n^k} = e^{-k^2/2n + O(k/n)}.
\]
\end{lem}

The next shows the asymptotic negligibility of a particular series which arises in deriving results about the global optima via the second moment method.  

\begin{lem}
\label{lem:varestimate}
Let $N={n\choose k}^2$ and let $a_{N}$ and $b_{N}$ be the centering and scaling constants in \eqref{eqn:an-def} 
and \eqref{eqn:bn-def}. Define $u_{N}=b_{N} - x_n / a_{N}$ where $-K \leq x_{n} \ll a_N^2$ for some constant 
$K \geq 1$. 
\begin{enumeratea}
\item\label{item:vara} 
There exists a constant $c > 0$ depending on $K$ such that for $k \leq c \log n / \log\log n$
\vskip-.05in
\begin{align*}
\sum_{\substack{1\leq s,t \leq k \\ st\neq k^2}}
{k\choose s}{k\choose t}{n-k\choose k-s}{n-k\choose k-t} {n\choose k}^{-2} \sqrt{\frac{k^2+st}{k^2-st}}\cdot  e^{stu_N^2/(k^2+st)} \to 0
\end{align*}
\vskip.06in
\noindent
as $n$ tends to infinity.

\vskip.1in

\item \label{item:varb}The same result holds if $\log k \ll \log n$ and $ k(\log\log n)^2/\log n\ll x_{n}\ll a_N^2$. 
\end{enumeratea}
\end{lem}

\begin{proof}[Proof of Lemma~\ref{lem:varestimate}]
We begin by establishing part (\ref{item:vara}) of the lemma;  a subsequent refinement yields part (\ref{item:varb}). 
Throughout the analysis we assume that $\log k= o(\log n)$, and therefore $k \ll \sqrt{n}$.
To begin, note that 
$
(k^2+st)/(k^2-st)\leq k
$
for all integers $s,t$ with  $2\leq s+t\leq 2k-1$. Thus we need to show that
\begin{align*}
I_{n}:=\sqrt{k} \sum_{\substack{1\leq s,t \leq k \\ st \neq k^2}}
{k\choose s}{k\choose t}{n-k\choose k-s}{n-k\choose k-t}  {n\choose k}^{-2} 
\exp\left\{\frac{st \, u_N^2}{(k^2 + st)} \right\} \to 0
\end{align*}
as $n$ tends to infinity.   Note that
\begin{align*}
{n-k\choose k-s} {n\choose k}^{-1} 
&= \, \frac{(k)_s (n)_{2k-s}}{(n)_k^2}
\end{align*}
and therefore by Lemma~\ref{lem:comb}, 
\begin{equation}
\label{eqn:comb1}
{n-k \choose k-s}{n-k \choose k-t}  {n \choose k}^{-2}  
\leq \, \frac{c \, (k)_s (k)_t}{n^{s+t}}
\end{equation}
for some universal constant $c>0$. Using Stirling's formula, we find
\begin{align*}
(k)_{s}(k)_{t} 
& \, \leq \, k \, (k/e)^{s+t} \, e^{-k (f(s/k) + f(t/k))}
\end{align*}
where $f(x) = -(1-x) \log(1- x)$.  As $f(x)$ is concave, the last inequality yields 
\begin{equation}
\label{eqn:comb2}
(k)_{s}(k)_{t} \, \leq \, k \, (k/e)^{s+t}e^{2kf((s+t)/2k)}.
\end{equation}
The elementary relation $4st \leq (s+t)^2$ implies that
\begin{equation}
\label{eqn:comb3}
\frac{st}{k^2+st} \leq \frac{(s+t)^2}{4k^2+(s+t)^2}.
\end{equation}
Combining (\ref{eqn:comb1}) - (\ref{eqn:comb3}) we find that 
\begin{align*}
I_{n}\ \leq \ 
& ck^{3/2} \sum_{\substack{1\leq s,t \leq k\\ st\neq k^2}}
{k\choose s}{k\choose t} \left(\frac{k}{en}\right)^{s+t} \exp\left(2kf\left(\frac{s+t}{2k}\right)
+ \frac{u_N^2(s+t)^2}{4k^2+(s+t)^2}\right).
\end{align*}
The last two terms in each summand depend on $s,t$ only through their sum; 
we decompose the outer sum accordingly.  Note that
for $0 \leq l \leq k$, 
\[
\sum_{s=0}^{l} {k\choose s}{k\choose l-s} = {2k\choose l} 
\, \leq \, 
\exp\{ 2 k \, [ f(l/2k) + f(1 - l/2k) ] \}
\]
where the inequality follows from a standard entropy bound for the binomial coefficient.
Therefore
\begin{eqnarray}
I_{n}
& \leq & 
c k^{3/2} \sum_{l=2}^{2k-1}  \left(\frac{k}{en}\right)^l \exp\left(2k \, g\Big(\frac{l}{2k}\Big)
    + \frac{u_N^2 l^2}{4k^2+l^2} \right) \nonumber \\[.1in]
\label{inq:I_n}
& = & c \, k^{3/2} \sum_{l=2}^{2k-1}  \left(\frac{k \, e^{u_N^2/4k}}{en}\right)^l 
\exp \left(2k \, g\Big(\frac{l}{2k}\Big) - \frac{u_N^2 l (2k-l)^2}{4k(4k^2+l^2)}\right)
\end{eqnarray}
\vskip.06in
\noindent
where $g(x)=-x\log x-2(1-x)\log(1-x)$.

\vskip.1in

Consider the first term in each summands in (\ref{inq:I_n}).  
The definition of $u_{N}$ ensures that 
\[
e^{u_N^2 / 2} = \frac{ N e^{-x_n r_n + o(1)} }{ \sqrt{2 \pi} \, a_N}
\]
where $r_n = 1 + x_n/2 a_N^2 \to 1$. 
Moreover, Stirling's approximation and our assumption that $k \ll \sqrt{n}$ imply that 
$N = (1 + o(1)) (en/k)^{2k} / 2 \pi k$, and therefore
\[
e^{u_N^2/4k} \leq \frac{en}{k} \left(\frac{c_{0} \, e^{-x_nr_n}}{\sqrt{k^3\log n}}\right)^{1/2k} 
\]
for some universal constant $c_0 > 0$.  Plugging this expression into inequality (\ref{inq:I_n})
yields
\begin{align*}
I_{n}
& \leq c_1 \, k^{3/2} \sum_{l=2}^{2k-1} \exp\left( 2k \, g\Big(\frac{l}{2k}\Big)
- \frac{u_N^2 \, l (2k-l)^2}{4k(4k^2+l^2)}\right) \left(\frac{e^{-x_nr_n}}{\sqrt{k^3\log n}}\right)^{l/2k}\\
& \leq c_1 \, k^{3/2} \sum_{l=2}^{2k-1} \exp\left( 2k \, g\Big(\frac{l}{2k}\Big)-\frac{u_N^2 \, l}{4k} 
\frac{(1-{l}/{2k})^2}{(1+(l/2k)^2)}\right) \left(\frac{e^{-x_nr_n}}{\sqrt{k^3\log n}}\right)^{l/2k}.
\end{align*}
 
Now consider the function
\[
\psi(x) := \frac{2g(x)}{x(1-x)^2(1+x^2)^{-1}} \ \ \ \ x \in (0,1) .
\]
It is easy to see that $\psi$ is positive and diverges to infinity as $x \to 0$ or $1$.  We claim that 
$\psi$ is convex. To see this, note that $\psi$ can be expressed as the sum of
\[
\psi_1(x) = \frac{-x \log x - x(1-x)}{x(1-x)^2(1+x^2)^{-1}} 
\ \mbox{ and } \  
\psi_2(x) = \frac{-2 (1-x) \log (1-x) + x(1-x)}{x(1-x)^2(1+x^2)^{-1}} 
\]
Taking the Taylor series expansion of 
$\psi_1$ around $1$, and $\psi_2$ around $0$, we find that the resulting power series have
non-negative coefficients for terms of degree $2$ or higher.  Thus $\psi_1$ and $\psi_2$ are
convex, and $\psi$ is convex as well.
Now note that for $x \in [1/k, 3/4]$ we have $\psi(x) \leq c \log k$ for some constant $c>0$.  
Moreover, $c \log k \leq u_N^2 / 4k$ under our assumption that $\log k = o(\log n)$. 
It follows that for $2 \leq l \leq 3k / 2$
\[
2 k g\Big( \frac{l}{2k} \Big) \ \leq \ l \frac{(1-l/2k)^{2}}{(1+l^2/4k^2) } \cdot \frac{u_{N}^{2}}{8k},
\]
and therefore
\begin{align*}
&k^{3/2} \sum_{l=2}^{3k/2} 
\exp\left( 2k \, g\Big(\frac{l}{2k}\Big) - 
\frac{u_N^2}{2} \frac{l}{2k}\frac{(1-{l}/{2k})^2}{(1+(l/2k)^2)}\right) \left(\frac{e^{-x_nr_n}}{\sqrt{k^3\log n}}\right)^{l/2k} \\[.1in]
&\leq k^{3/2} \sum_{l=2}^{3k/2} \exp\left(-\frac{u_N^2 l}{8k} \frac{(1-{l}/{2k})^2}{(1+(l/2k)^2)}\right) \left(\frac{e^{-x_nr_n}}{\sqrt{k^3\log n}}\right)^{l/2k}\\[.1in]
&\leq k^{3/2} \sum_{l=2}^{3k/2} \exp\left(-\frac{u_N^2}{2^8k} \cdot l\right) 
\ \leq \ \frac{k^{3/2} \exp(-{u_N^2}/{2^7k})  }{1- \exp(-{u_N^2}/{2^8k}) } .
\end{align*}
In the second inequality we dropped the final term in each summand, which is less than one.
As $u_{N}^{2}/k \approx \log n$, the final term above tends to zero as $n$ tends to infinity. 

We now consider the remaining terms in the sum $I_n$.
For $x \in [3/4,1-1/2k]$ the function $\psi(x) \leq c' k\log k$.  Moreover,
if $k \leq c \log n / \log\log n$ for a sufficiently small constant $c > 0$,
then $c' k \log k \leq u_N^2 / 4k$. In this case, arguments like those above show that
\begin{align*}
&k^{3/2} \sum_{l=3k/2}^{2k-1} \exp\left(2kg\left(\frac{l}{2k}\right)-\frac{u_N^2}{2(1+(l/2k)^2)} \frac{l}{2k} \left(1-\frac{l}{2k}\right)^2\right) \left(\frac{e^{-x_nr_n}}{\sqrt{k^3\log n}}\right)^{l/2k}\\[.1in]
&\leq k^{3/2} \sum_{3k/2}^{2k-1} \exp\left(-\frac{u_N^2}{2^6k^2} (2k-l)^2\right) \frac{1}{(k^3\log n)^{3/8}}\leq \frac{k^{3/8}}{(\log n)^{3/8}}\sum_{l=1}^{k/2} \exp\left(-\frac{u_N^2}{2^6k^2} l^2\right)\to 0 
\end{align*}
as $\log n/k, u_{N}^2/k^2\to \infty$ as $n\to\infty$.  This completes the proof of part (\ref{item:vara})
 
Suppose now that $c \log n / \log\log n \leq k$ and that 
$k(\log\log n)^2/\log n \ll x_n \ll a_N^2$.  In this case we need to break the 
final part of the sum defining $I_n$ into two parts: from $3k/2$ to $2k s_n(k)$, and from 
$2k s_n(k)$ to $2k-1$, where
$s_n(k) = 1 - \log \log n/(2c\log n)$.
For $x \in [3/4, s_n(k)]$ the function $\psi(x) \leq c'\log n \leq u_N^2/4k$.  It follows that
\begin{align*}
&k^{3/2} \sum_{l=3k/2}^{2k s_n(k)} \exp\left(2kg\left(\frac{l}{2k}\right)-\frac{u_N^2}{2(1+(l/2k)^2)} \frac{l}{2k} \left(1-\frac{l}{2k}\right)^2\right) \left(\frac{e^{-x_nr_n}}{\sqrt{k^3\log n}}\right)^{l/2k}\\[.1in]
& \leq 
k^{3/2} \sum_{3k/2}^{2k s_n(k)} 
\exp\left(-\frac{u_N^2}{2^6k^2} (2k-l)^2\right) \frac{e^{-3x_n/4}}{(k^3\log n)^{3/8}}\\[.1in]
& \leq 
\frac{k^{3/8}e^{-3x_n/4}}{(\log n)^{3/8}}\sum_{l=k\log\log n/c\log n}^{k/2} 
\exp\left(-\frac{u_N^2}{2^6k^2} l^2\right) \leq 
\frac{k^{11/8}e^{-3x_n/4}}{(\log n)^{3/8}}\exp\left(-\frac{c'k (\log\log n)^2}{\log n} \right)
\end{align*}
Our assumptions on $k$ and $x$ ensure that the last term tends to zero with increasing $n$.
For the remainder of the sum, note that
 \begin{align*}
& k^{3/2} \sum_{l = 2k s_n(k) }^{2k-1} \exp\left(2kg\left(\frac{l}{2k}\right)-
\frac{u_N^2}{2(1+(l/2k)^2)} \frac{l}{2k} \left(1-\frac{l}{2k}\right)^2\right) 
\left(\frac{e^{-x_n r_n}}{\sqrt{k^3\log n}}\right)^{l/2k}\\[.1in]
& \leq k^{5/2} \exp\{ 2k g( s_n(k) ) \} {e^{-x_n(1+o(1))}}(k^3\log n)^{-1/2+o(1)}\\[.1in]
& \leq k^{5/2} \exp\left(c'k(\log\log n)^2/\log n - x_n(1+o(1)) \right) (k^3\log n)^{-1/2+o(1)} .
\end{align*}
Our assumptions on $k$ and $x_{n}$ ensure that the final term tends to zero with increasing $n$.  This completes the proof.
\end{proof}

\section{Proofs for Global Maxima}
\label{sec:proofs-gopt}

\begin{proof}[Proof of Theorem~\ref{thm:globalmax}:] 
We begin with the proof of part~\eqref{item:maxa}. Recall that $N:=|\sS_n(k)|= {n\choose k}^2$.  For fixed $x \in \dR$,
\[
\pr( a_N (k M_n(k) - b_N) \leq x ) = \pr \left( \max_{\gl \in \sS_{n}(k)} k \avg(\vW_\gl) \leq u_N \right)
\]
where $u_{N} = b_{N} + a_{N}^{-1} x$.
Note that $k \avg(\vW_\gl) \sim \text{N}(0,1)$ for all $\gl \in \sS_{n}(k)$, and therefore the second term above concerns
the maximum of $N$ correlated standard Gaussians.  It follows from Lemma \ref{lem:extreme} that if $Z_1, Z_2, \ldots$ 
are independent $\text{N}(0,1)$ then 
\[
\pr \left( \max_{1 \leq i \leq N} Z_i  \leq u_N \right) \to \pr( - \log T \leq x) 
\]
as $N$ tends to infinity, where $T \sim \text{Exp}(1)$. Thus it suffices to show that
\[
\left| \pr\left( \max_{\gl \in \sS_{n}(k)} k \avg(\vW_\gl) \leq u_N \right) 
- \pr\left( \max_{1 \leq i \leq N} Z_i  \leq u_N \right) \right| 
\]
tends to zero as $n$ (and therefore $N$) tends to infinity.  By Lemma~\ref{lem:compare}, the absolute
value above is at most
\begin{equation}
\label{cmpbnd}
\sum_{\lambda \neq \lambda', \, \gs(\lambda, \lambda') \neq 0}  
2 \sqrt{\frac{1 + \gs(\lambda, \lambda')}{1 - \gs(\lambda, \lambda')}} \cdot 
\bar{\Phi}(u_N) ^2 \cdot e^{u_N^2 \gs(\lambda, \lambda') / (1+\gs(\lambda, \lambda') )} \\
\end{equation}
where the sum is over index sets $\lambda, \lambda' \in \sS_{n}(k)$, and
\[
\gs(\lambda, \lambda')
\ = \ 
\cov(k\avg(\vW_\gl),k\avg(\vW_{\gl'}))
\ = \ 
st \, k^{-2}
\]
if $\gl$ and $\gl'$ share $s$ rows and $t$ columns.  By a straightforward combinatorial argument, the expression
(\ref{cmpbnd}) reduces to
\begin{align*}
N^2\bar{\Phi}(u_N)^2\sum_{\substack{1\leq s,t\leq k\\ st\neq k^2}}{k\choose s}{k\choose t}{n-k\choose k-s}{n-k\choose k-t} {n\choose k}^{-2} \sqrt{\frac{k^2+st}{k^2-st}}\cdot  e^{stu_N^2/(k^2+st)}
\end{align*}
It is easy to check that $N\bar{\Phi}(u_N)\to e^{-x}$ as $N\to\infty$, and the sum tends to zero by 
Lemma~\ref{lem:varestimate}\eqref{item:vara}. 

Part \eqref{item:maxb} of the theorem follows in a similar fashion, using Slepian's lemma and 
Lemma~\ref{lem:varestimate}\eqref{item:varb}. We omit the details.

We now turn to the proof of part \eqref{item:maxc}. Fix $k \geq 1$ and $x \in \vR$.  It follows from 
part \eqref{item:maxa} of the theorem that 
\[
\pr(a_{N}( k M_{n}(k) - b_{N} ) \geq x) \to \pr( -\log T \geq x)
\] 
as $n$ and $N={n\choose k}^2$ tend to infinity, where $T$ is an exponential rate one random variable.
Given a $k \times k$ matrix $\vU$ let $\hat{\vU}$ denote the centered matrix 
$\vU - \avg(\vU) \vone\vone^{\prime}$.  Define 
\[
x_{n}=(b_{N}+x/a_{N})/k.
\]
Let $S \in \cB(\dR^{k \times k})$ be a measurable set of $k \times k$ submatrices.  It suffices to show
that 
\begin{equation}
\label{deltan}
\Delta_n := \pr( M_n(k) \geq x_n, \hat{\vW}_{\gl_n(k)} \in S) 
\ - \ \pr(M_n(k) \geq x_n) \pr(\hat{\vW}^k \in S)
\ \to \ 0 .
\end{equation}
as $n$ tends to infinity.   To see this, note that if $\vZ$ is a $k \times k$ Gaussian random matrix
independent of $\vW$, then
\[
\pr(\hat{\vW}_{\gl_n(k)} \in S) - \pr( \hat{\vZ} \in S) 
\ \leq \ 
2 \pr( M_n(k) < x_n) + | \Delta_n | .
\]
The first term on the right can be made arbitrarily small by choosing $x$ large and negative.
We now turn our attention to (\ref{deltan}).  To begin, note that by symmetry 
\begin{align*}
p:=  \ \pr(M_n(k) \geq x_n, \hat{\vW}_{\gl_n(k)} \in S) 
& =  \sum_{\gc \in \sS_{n}(k)} \pr(\gl^*(k)=\gc, \, \avg(\vW_{\gc}) \geq x_n, \hat{\vW}_\gc \in S) \\[.1in]
&=\  N \cdot \pr(\gl_n(k) = \gc_0, \, \avg(\vW^k) \geq x_n, \hat{\vW}^k \in S)
\end{align*}
where $\vW^k$ is the upper left corner submatrix, with index set $\gc_0 = [k] \times [k]$.  Let
\[
\sE_0 : =\{\gc \in \sS_{n}(k) : \gc \cap \gc_0 = \emptyset\}
\]
be the index sets of $k \times k$ submatrices that do not overlap $\vW^k$, and let
\[
\sN_0 : = \{\gc \in \sS_{n}(k) : \gc \cap \gc_0 \neq \emptyset \mbox{ and } \gc \neq \gc_0  \} 
\]
Thus $\sS_n (k) = \sE_0 \cup \sN_0 \cup \{ \gc_0 \}$.  Define events $A, B, C,  D$ and $E$ as follows:
\[
A = \left\{ \max_{\gc\in \sE_0} \avg(\vW_{\gc}) < \avg(\vW_{\gl}) \right\}, \ \ \ \ 
B = \left\{ \max_{\gc\in \sN_0} \avg(\vW_{\gc}) < \avg(\vW_{\gl}) \right\}
\]

\[
C = \{ \avg(\vW^k) \geq x_n \},
\ \ \ \ 
D = \{ \hat{\vW}^k \in S \},
\ \ \ \ 
E = \{ M_n(k) \geq x_n \}. \\
\]
\vskip.1in
\noindent
Note that
$
\{\gl_n(k) = \gc_0 \} = A \cap B
$.
Moreover, as $\avg(\vW^k)$, $\hat{\vW}^k$, and $\max_{\gc \in \sE_0} \avg(\vW_{\gc})$ are independent, 
\begin{align*}
N^{-1}p \ 
& = \  \pr(A \cap B \cap C \cap D) \\
& =  \ \pr(A \cap C \cap D) - \pr(A \cap B^c \cap C \cap D) \\
& =  \ \pr(A \cap C) \pr(D) - \pr(A \cap B^c \cap C \cap D) ,
\end{align*}
and therefore
\begin{equation}
\label{dac}
\big| N^{-1}p - \pr(D) \pr(A \cap C) \big| \ \leq \ \pr(B^c \cap C) .
\end{equation}
Setting $S = \dR^{k\times k}$ the last inequality yields
\begin{equation}
\label{eac}
\big| N^{-1} P(E) - \pr(A \cap C) \big| \ \leq \ \pr(B^c \cap C) 
\end{equation}
Using (\ref{dac}) and (\ref{eac}) we obtain the bound
\begin{eqnarray*}
| \Delta_n | 
& = & 
\big| \, p - \pr(D) \pr(E) \big| \\[.1in]
& \leq &
N \big| \, N^{-1}  p - \pr(A \cap C) \pr(D) \big| \ + \ N \pr(D) \big| N^{-1} \pr(E) - \pr(A \cap C) \big| \\[.1in]
& \leq & 
2 N \pr(B^c \cap C) 
 =  
2 N \pr(C) \pr(B^c \, | \, C)  .
\end{eqnarray*}
Now $N \pr(C) = N \pr(\avg(\vW^k) \geq x_n) \to e^{-x}$, so it suffices to show that
\[
\pr(B^c \, | \, C) \ = \  
\pr\left( \max_{\gc \in \sN_0} \avg(\vW_{\gc}) \geq \avg(\vW^k) \, \big| \, \avg(\vW^k) \geq x_n \right) 
\ \to \ 0
\]
as $n$ tends to infinity.  For $1\leq s,t \leq k$ define
\[
\sN_0 (s,t) \ = \ \{ \gc \in \sN_0 : | \gc \cap \gc_0 | = (s,t) \}
\]
It follows from the union bound that
\begin{align*}
& \pr\left( \max_{\gc \in \sN_0} \avg(\vW_{\gc}) \geq \avg(\vW^k) \, \big| \, \avg(\vW^k) \geq x_n \right) \ \leq \\[.05in]
& \sum_{1 \leq s,t \leq k, st \neq k^2} 
\pr\left( \max_{\gc \in \sN_0 (s,t)} \avg(\vW_{\gc}) \geq \avg(\vW^k) \, \big| \, \avg(\vW^k) \geq x_n \right) 
\ =: \sum_{1 \leq s,t \leq k, st \neq k^2} p_n(s,t) 
\end{align*}
Fix $1 \leq s, t \leq k$ with $st \neq k^2$.  For each $\gc \in \sN_0 (s,t)$ let
\[
F(\vW_{\gc}) \ = \ \frac{1}{| \gamma \setminus \gamma_0 |} 
\sum_{(i,j) \in \gamma \setminus \gamma_0} W_{i,j}
\]
be the average of the entries of $\vW_\gc$ that lie outside $\vW^k$, and note that 
$| \gamma \setminus \gamma_0 | = k^2 - st$.  A straightforward argument shows that
\[
\max_{\gc \in \sN_0 (s,t)} \avg(\vW_{\gc}) \ \geq \ \avg(\vW^k)
\]
implies
\[
\max_{\gc \in \sN_0 (s,t)} F(\vW_{\gc}) \ \geq \ \avg(\vW^k)  - (k^2 - st)^{-1} st  \max_{i,j \in [k]} \hat{W}_{ij}
\]
where $\hat{W}_{ij}$ are the entries of $\vW^k$.  As $\vW^k$ is independent of $\avg(\vW^k)$ and
$F(\vW_{\gc})$, the last relation implies that
\begin{align*}
p_n(s,t) \ \leq \ \pr\left( \max_{\gc \in \sN_0 (s,t)} F(\vW_{\gc}) \geq x_n  - (k^2 - st)^{-1} st \max_{i,j \in [k]} \hat{W}_{ij} \right)
\end{align*}
By Slepian's lemma, for each $x \in \dR$,
\[
\pr\left( \max_{\gc\in \sN_\gl(s,t)} F(\vW_{\gc}) \geq x \right) \leq \pr\Big( V_{|\sN_0 (s,t)|} \geq x \sqrt{k^2-st} \Big)
\]
where $V_{n}$ denotes the maximum of $n$ independent standard Gaussians. Therefore
\begin{align}
p_n(s,t) & 
\leq \ \pr\left( V_{|\sN_0 (s,t)|} \geq x_n \sqrt{k^2-st}  - (k^2-st)^{-1/2}st \max_{i,j\in [k]} \hat{W}_{ij} \right) \notag\\[.1in]
& \leq \ |\sN_\gl(s,t)| \, \pr\left( V_{1} \geq x_n \sqrt{k^2-st}  - (k^2-st)^{-1/2}st \max_{i,j \in [k]} \hat{W}_{ij} \right).
\label{eq:tozero}
\end{align}  
Now note that
\[
|\sN_{\gl}(s,t)|={k\choose s}{k\choose t}{n-k\choose k-s}{n-k\choose k-t}=O(n^{2k-s-t})
\] 
and that 
\[
x_n\sqrt{k^2-st} \approx \sqrt{(2k-2st/k)2\log n}
\]
An elementary argument shows that $2st/k \leq (s+t)^2/2k \leq (1-1/2k)(s+t)\leq s+t-1/k$.  Therefore
the probability in \eqref{eq:tozero} converges to zero as $n$ tends to infinity, and this completes the proof. 
\end{proof}

\subsection{Two point localization}

Fix $\gt> 0$ and recall the definition of $\tilde{k}$ and $k^*$ from \eqref{eqn:tildk-def}.  
For $1 \leq m \leq n$ let $N_n(m)$ denote the number of $m \times m$ sub-matrices 
with average greater than $\gt$, 
\[
N_n(m) := \sum_{\gl\in \sS_n(m)} \ind\set{\avg(\vW_{\gl})> \gt}.
\]
Note that if there is an $m \times m$ submatrix with average greather than $\gt$, then there exists 
an $(m-1) \times (m-1)$ submatrix with average greater than $\gt$. 
Thus Theorem \ref{thm:loc} is a corollary of the following Proposition~\ref{prop:first-second-mom}.

\begin{prop} \label{prop:first-second-mom}
Let $\gt>0$ be fixed. 
\begin{enumeratei}
\item\label{ia} 
$\pr(N_n(k_n^* + 1) > 0) \to 0$ as $n \to \infty$.
 
\vskip.1in 
 
\item\label{ib} 
$\pr(N_n(k_n^* - 1) > 0) \to 1$ as  $n \to \infty$.
\end{enumeratei}
\end{prop}

\begin{proof}[Proof of Proposition~\ref{prop:first-second-mom}]
For $x\in [1,n]$ define the function 
\[
f_{n}(x) := {n \choose x}^2 \bar{\Phi}(x\gt) 
\]
It is easy to see that $\E(N_{n}(m))=f_{n}(m)$. 
It follows from Lemma~\ref{lem:gauss_tail} and Stirling's approximation
$
\Gamma(x+1)=\sqrt{2\pi}x^{x+1/2}e^{-x+ O(1/x)} \text{ for } x \geq 1
$
that for any $c \in \dR$ and $x$ such that $1 \leq x + c \leq n - 1$,
\begin{align*}
&\frac{f_n(x+c)}{f_n(x)} \\[.1in]
&= \left(\frac{\Gamma(x+1)\Gamma(n-x+1)}{\Gamma(x+c+1)\Gamma(n-x-c+1)}\right)^2 
\cdot \frac{\bar{\Phi}((x+c)\gt) }{\bar{\Phi}(x \gt) }\\[.1in]
&= \frac{x^{2x+2}(n-x)^{2n-2x+1}}{(x+c)^{2x+2c+2}(n-x-c)^{2n-2x-2c+1}}
\cdot \exp\left\{ - \frac{(2cx+c^2) \gt^2}{2} + O(1/x) \right\} \\[.1in]
&= \frac{(1-x/n)^{2c}}{(1+c/x)^{2x+2c+2}(1-c/(n-x))^{2n-2x-2c+1}}
\cdot \left(\frac{en}{x} e^{- x\gt^2}\right)^c 
\exp\left\{ \frac{-c - c^2 \gt^2}{2} + O(1/x) \right\}.
\end{align*}
Using the relation \eqref{eqn:tildk-asymp} we have $f_{n}(\tilde{k}_{n})=1$. Moreover one can easily check that
\[
\frac{en}{\tilde{k_n}} \exp(-\tilde{k}_n \gt^2/4)\to 1 \text{ as } n \to \infty. 
\]
Thus $f_n(\tilde{k}_{n}+c) =n^{-3c(1+o(1))}$ tends to $0$ or infinity when
$c$ is positive or negative, respectively.  In particular, 
$\E(N_{n}(k_n^*+1)) \to 0$ and $\E(N_{n}(k_n^*-1)) \to \infty$ as $n$ tends to infinity. Note that, the distance between $\tilde{k}_{n}$ and $k_n^*+1$ is more than $1/2$. Thus $\E(N_{n}(k_n^*+1)) \leq n^{-3/2+o(1)}$. Part \eqref{ia} now follows easily from the fact that $\pr(N_n(k_n^*+1)>0) \leq \E(N_{n}(k_n^*+1)) \leq n^{-3/2+o(1)}$. By Borel-Cantelli lemma we have $N_n(k_n^*+1)=0$ eventually a.s.

To prove \eqref{ib}, we will make use of the second moment method. To simplify notation, 
let $k = k_n^*-1$.  We have already proved that $\E(N_{n}(k))\ge n^{3/2+o(1)}\to\infty$ as $n\to\infty$.  
By a standard second moment argument,
\[
\pr(N_n(k) = 0) \ \leq \ \frac{\var(N_n(k))}{(\E N_n(k))^2}.
\]
To this end, note that the collection of random variables $\{ \avg(\vW_{\gl}): \gl \in \sS_n(k) \}$ is transitive,
in the sense that for any $\gl_0,\gl_1 \in \sS_n(k)$ there exists a permutation $\pi:\sS_n(k) \to \sS_n(k)$ 
such that $\pi(\gl_{0}) = \gl_{1}$ and
\[
\{ \avg(\vW_{\gl}): \gl\in \sS_n(k) \} \equald \{ \avg(\vW_{\pi(\gl)}): \gl \in \sS_n(k) \}.
\]
A simple calculation using this transitivity shows that, in order to prove the second assertion, 
it is enough to show that for some fixed $\gl_{0} \in \sS_n(k)$,
\begin{equation}\label{eqn:in-to-0}
I_n := \sum_{\gl\in \sS_n(k), \gl\neq \gl_0, \gl\cap \gl_0 \neq \emptyset} \frac{\pr(\avg(\vW_{\gl})> \tau\mid \avg(\vW_{\gl_0}) > \tau)}{\E(N_n(k))} \to 0
\end{equation}
as $n$ tends to infinity. Moreover, we have
\[
\pr(N_n(k) = 0) \ \leq \  I_n.
\]

Note that the vector 
$(k\avg(\vW_{\gl}), k\avg(\vW_{\gl_0}))$ has a bivariate normal distribution with variance 
one and correlation $st / k^2$, where $s$ is the number of rows shared by $\gl,\gl_{0}$, and 
$t$ is the number of common columns shared by $\gl,\gl_{0}$. 

For $1\leq s,t \leq k$ define the quantity
\[
E(s,t):= {n\choose k}^{-2} {k\choose s}{n-k\choose k-s} 
{k\choose t}{n-k\choose k-t} \frac{\pr( Z_{st} \geq k \gt \mid Z \geq k \gt )}{\bar{\Phi}(k \gt)}
\]
where $Z_{st}=k^{-2}stZ + \sqrt{1-k^{-4}s^2t^2}Z'$ and $Z,Z'$ are independent standard Gaussians.  Thus we have
\begin{align}\label{eq:In}
I_n:= \sum_{s=1}^{k}\sum_{t=1}^{k} E(s,t)
\end{align}
Clearly 
$
E(k,k)= 1/\E(N_{n}(k))\to 0 \text{ as } n\to\infty. 
$
We need to estimate $E(s,t)$ for $st\neq k^2$. 

Using Lemma~\ref{lem:biv-norm-tail} and Lemma~\ref{lem:gauss_tail} with $\theta_{st}:=\sqrt{\frac{k^2-st}{k^2+st}}$, we have
\begin{align*}
\frac{\pr(Z_{st}\geq k\gt\mid Z\geq k\gt)}{\bar{\Phi}(k\gt)}
\leq 
\frac{2\bar{\Phi}(\theta_{st} k\gt)}{\bar{\Phi}( k\gt)}
\leq 2\sqrt{\frac{k^2+st}{k^2-st}}\exp\left(\frac{stk^2 \gt^2}{k^2+st}\right)
\end{align*}
for $st\leq k(k-1)$ and $k\gt>1$. Now we use Lemma~\ref{lem:varestimate}(b) with $N={n\choose k}^2, x_{n}=(b_{N}-k\gt) a_{N}\approx \frac{3}{4}(\tilde{k}_{n}-k)k\gt^2$ to bound
\[
I_{n}\le 2\sum_{\substack{1\le s,t\le k\\ st\neq k^2}}  {k\choose s}{n-k\choose k-s} {k\choose t}{n-k\choose k-t} {n\choose k}^{-2}  \sqrt{\frac{k^2+st}{k^2-st}}\exp\left(\frac{stk^2 \gt^2}{k^2+st}\right) + \frac{1}{\E(N_n(k))}.
\]
Thus we have $I_n\to 0$ as $n\to\infty$ and we are done.
\end{proof}

\section{Structure Theorem for Local Optima}
\label{sec:proofs-lopt}

\begin{proof}[Proof of Theorem \ref{thm:st1}]
Let $\cR_{k}$ and $\cC_{k}$ be the events that the sub-matrix $\vW^k$ 
is row optimal and is column optimal, respectively. 
Clearly, $\pr(\cR_{k}) = \pr(\cC_{k}) = {n\choose k}^{-1}$.   We wish to find the probability of the event 
$\cI_{k,n} := \cR_{k} \cap \cC_{k}$.  To begin, we fix some notation.   Let $\vC = \vW^k$ and let 
$c_{\cdot\cdot}=k^{-2} \sum_{i,j=1}^k C_{ij}$ be the average of the entries of $\vC$.  For $1\leq i,j \leq n$
let $c_{i \cdot} = k^{-1} \sum_{j=1}^{k} W_{ij}$ be the 
average of the first $k$ entries in the $i$th row of $\vW$.  Define the column averages 
$c_{\cdot j} = k^{-1} \sum_{i=1}^{k} W_{ij}$ in a similar fashion.  Note that $c_{i \cdot}$, $c_{\cdot j}$
are defined for rows and columns outside $\vC$.  Letting 
$\tilde{c}_{ij}=c_{ij}-c_{i\cdot}-c_{\cdot j} + c_{\cdot\cdot}$, we may write each entry of $\vC$ in
terms of its ANOVA decomposition 
\[
C_{ij} = \tilde{c}_{ij} + (c_{i\cdot}-c_{\cdot\cdot}) + (c_{\cdot j} - c_{\cdot\cdot}) + c_{\cdot\cdot}
\]
In the Gaussian setting under study, the families of random variables 
\[
\tilde{\vC} = \{ \tilde{c}_{ij} : 1 \leq i,j \leq k \} \ \ \ 
\{ c_{i \cdot} - c_{\cdot\cdot} : 1 \leq i \leq k \} \ \ \  
\{ c_{\cdot j} - c_{\cdot\cdot} : 1 \leq j \leq k \} \ \ \  
c_{\cdot\cdot}
\]
{are independent, and obviously independent of the families  
\[
\{ c_{i \cdot} : k+1 \leq i \leq n \} \ \ \ \   
\{ c_{\cdot j} : k+1\leq j \leq n \}.
\]
Note that the events $\cR_{k}$ and $\cC_{k}$ can be written as
\begin{align*}
\cR_{k} &: = 
\left\{  \min_{1 \leq j \leq k} c_{\cdot j} \geq \max_{k < j \leq n} c_{\cdot j} \right\}
\, =\, 
\left\{ c_{\cdot\cdot} - \max_{1 \leq j \leq k} \{ c_{\cdot \cdot} - c_{\cdot j} \} \geq \max_{k < j\leq n} c_{\cdot j} \right\}\\[.1in]
\cC_{k} &:=
\left\{ \min_{1 \leq i \leq k} c_{i \cdot } \geq \max_{k < i \leq n} c_{i \cdot}  \right\}
\, = \, 
\left\{ c_{\cdot\cdot} - \max_{1 \leq i \leq k} \{ c_{\cdot\cdot} - c_{i \cdot }\}\geq \max_{k<i\leq n} c_{i\cdot}  \right\}
\end{align*}
and therefore
\begin{align*}
\cI_{k,n}
=
\cR_{k} \cap \cC_{k} 
&= \left\{ c_{\cdot\cdot} \geq \max_{k < j \leq n} \{ c_{\cdot j}, c_{j\cdot} \} \right\} 
\ \cap \ 
\left\{ \max_{1\leq j\leq k}\{ c_{\cdot\cdot}-c_{\cdot j}\} \leq c_{\cdot\cdot}- \max_{k<j\leq n} c_{\cdot j} \right\} \\
&\qquad\qquad\cap 
\left\{\max_{1\leq i\leq k}\{ c_{\cdot\cdot} - c_{i \cdot }\} \leq  c_{\cdot\cdot} -  \max_{k<i\leq n} c_{i\cdot} \right\}.
\end{align*}
Now let 
\[
M_{n-k} = a_n \Big(\sqrt{k} \max_{k < j \leq n} c_{\cdot j} - b_n\Big)
\ \ \ \  
M'_{n-k} = a_{n} \Big(\sqrt{k} \max_{k < i \leq n} c_{i \cdot} - b_{n}\Big)
\] 
be the recentered and rescaled maxima of the row and column averages outside $\vC$.
It follows from Lemma~\ref{lem:extreme} that $(M_{n-k}, M'_{n-k})$ converges in distribution to 
$(-\log T,-\log T')$, where $T,T'$ are independent Exponential rate one random variables. 
Using the previous displays, one may express the eveng $\cI_{k,n}$ as follows:
\begin{align}
\cI_{k,n}
& = 
\big\{ \sqrt{k}c_{\cdot\cdot} \geq b_n  + a_{n}^{-1}\max\{M_{n-k},M'_{n-k}\} \big\} \notag\\
&\qquad\cap 
\left\{ \sqrt{k}\max_{1\leq j\leq k}\{ c_{\cdot\cdot}-c_{\cdot j}\} \leq \sqrt{k}c_{\cdot\cdot}- (b_n + a_{n}^{-1}M_{n-k}) \right\}\notag\\
&\qquad\qquad\cap 
\left\{ \sqrt{k}\max_{1\leq i\leq k}\{ c_{\cdot\cdot} - c_{i \cdot }\} \leq  \sqrt{k}c_{\cdot\cdot}- (b_n+ a_{n}^{-1}M'_{n-k}) \right\}.
\label{eqn:ik-decomp}
\end{align}
Note that $kc_{\cdot\cdot},\sqrt{k}c_{i\cdot},\sqrt{k}c_{\cdot j}$ are standard Gaussian random variables. 

For $k \geq 1$ and $x \in \bR$ define $F_{k}(x) = \pr(\max_{1\leq i\leq k}(\bar{Z}-Z_{i}) \leq x)$ where $Z_1,Z_{2},\ldots,Z_{k}$ are independent standard Gaussians and $\bar{Z}=k^{-1}\sum_{i=1}^{k}Z_{i}$. 
Using the independence arising in the ANOVA decomposition and \eqref{eqn:ik-decomp} we find
\begin{align}
\pr(\cI_{k,n}) = \E \bigl( &F_k(k^{-1/2}Z - b_n - a_{n}^{-1}M_{n-k}) \notag \\
&\quad F_k(k^{-1/2}Z - b_n - a_{n}^{-1}M'_{n-k}) \notag \\
&\qquad\ind\{ k^{-1/2}Z\geq b_n  + a_{n}^{-1}\max\{M_{n-k},M'_{n-k}\} \} \bigr), \label{eq:ev}
\end{align}
where $Z$ is independent of $M_{n-k}$ and $M'_{n-k}$.
Define $R_n = a_n (k^{-1/2} Z - b_n)$, and consider the event 
\[
A_n = \{ R_n \geq \max\{M_{n-k},M'_{n-k}\} \}
\]
appearing in (\ref{eq:ev}). 
Lemma~\ref{lem:condz}(\ref{item:condza}) with $\theta=k$ implies that
\[
\sqrt{k}(\sqrt{2 \pi} b_n)^{1-k} \, n^k \, \pr( R_n \geq x) \, \to \,  e^{-k x}
\]
uniformly over $|x| \ll a_n$. Using the fact that $\max\{M_{n-k},M'_{n-k}\}$ 
converges in distribution to $-\log(T/2)$ where $T \sim \text{Exp}(1)$, 
it follows that
\begin{equation}
\label{pAn}
\sqrt{k}(\sqrt{2\pi}b_n)^{1-k} \, n^k \, \pr(A_n) \, \to \, 2^{-k}\E(T^k)=2^{-k} k!. 
\end{equation}
For $x, x', y > 0$ the relation (\ref{pAn}) and the independence of $R_n$, 
$M_{n-k}$, and $M'_{n-k}$ imply that
\begin{align}
\label{pRn}
&\sqrt{k} (\sqrt{2 \pi} b_n)^{1-k} \, n^k \, 
\pr( R_n \geq - \log y, \, M_{n-k} \leq -\log x, \, M'_{n-k} \leq -\log x') \nonumber \\[.03in]
&\qquad \to y^{k} \, e^{-x-x'} \text{ as } n \to \infty. 
\end{align} 
We claim that, conditional on the event $A_n$, 
\begin{equation}
\label{eqn:avg-mn-mn-joint}
	(R_n, M_{n-k}, M'_{n-k}) \Rightarrow ( -\log G, -\log (G+Y), -\log(G+Y'))
\end{equation}
where $G,Y,Y'$ are mutually independent, $Y,Y' \sim \text{Exp}(1)$, and $G \sim \text{Gamma}(k,2)$, 
with density $2^k x^{k-1} e^{-2x} / (k-1)!$ for $x > 0$.  To see this, note that if $0 < y \leq \min\{ x, x' \}$
then
\begin{eqnarray*}
\lefteqn{ \pr( R_n \geq - \log y, \, M_{n-k} \leq -\log x, \, M'_{n-k} \leq -\log x' \, | \, A_n) } \\
& = & 
\pr( R_n \geq - \log y, \, M_{n-k} \leq -\log x, \, M'_{n-k} \leq -\log x' ) \, \pr(A_n)^{-1} \\
& \to &
\frac{2^k}{k!} y^{k} \, e^{-x-x'} \mbox{ as } n \to \infty.
\end{eqnarray*}
where in the last, limiting, step we have used (\ref{pAn}) and (\ref{pRn}).
On the other hand, for $G,Y,Y'$ distributed as above, for the same values 
of $x, x', y$ 
\begin{align*}
\pr( G \leq y, \, G+Y \geq x, \, G+Y' \geq x') 
&= \int_{0}^{y} \frac{2^k}{(k-1)!}t^{k-1}e^{-2t} e^{-(x-t)}e^{-(x'-t)}dt
= \frac{2^k}{k!} y^{k} \, e^{-x-x'},
\end{align*}
in agreement with the previous display.  This establishes (\ref{eqn:avg-mn-mn-joint}).

It follows from (\ref{eq:ev}), (\ref{eqn:avg-mn-mn-joint}), and Lemma~\ref{lem:tail}\eqref{item:taila} that 
\begin{eqnarray*}
\pr(\cI_{k,n}) 
& = & 
\E \bigl( F_k( a_{n}^{-1} (R_n - M_{n-k}) ) \, F_k(a_{n}^{-1} (R_n - M'_{n-k}) ) \, \big| \, A_n \bigr)
\pr(A_n) \\[.1in]
& = & 
\E \bigl( F_k(a_n^{-1} \log(1+{Y}/G)) \, F_k(a_n^{-1} \log(1+{Y}'/G)) \bigr) (1 + o(1))
\pr(A_n) \\[.1in]
& = & 
\alpha_k^2 \, a_n^{-2(k-1)} \E \bigl( \big[ \log(1+{Y}/G)) \log(1+{Y}'/G)) \big]^{k-1} \bigr) (1 + o(1))
\pr(A_n) 
\end{eqnarray*}
where $\alpha_k$ is the constant defined in (\ref{alphak}).  Combining the last expression with (\ref{pAn}),
the expression (\ref{alphak}), and Stirling's formula, we find
\[
\lim_{n \to \infty}{n \choose k}{a_n}^{k-1} \pr({\cI_{k,n}}) 
\ = \ 
\frac{ (2 \pi)^{(k-1)/2} \alpha_k^2 }{\sqrt{k}2^k}\E\big( \big[ \log(1+{Y}/G) \log(1+{Y}'/G) \big]^{k-1} \big).
\]
In particular, 
\[
\pr(\cI_{k,n})= \frac{\theta_k}{{n\choose k}(\log n)^{(k-1)/2}}(1+o(1))
\] 
as $n$ tends to infinity, where
\begin{align}\label{eq:thetak}
\theta_{k}:=\frac{k^{2k+1/2}}{\pi^{(k-1)/2} 2^{2k-1}k!^2} 
\E\big( \big[ \log(1+{Y}/G) \log(1+{Y}'/G) \big]^{k-1} \big).
\end{align}

We now wish to find the asymptotic conditional distribution of the matrix itself given $\cI_k$. Recall that in matrix form, the ANOVA decomposition can be written as 
\begin{align}
	\vC  
	&= c_{..}\vone\vone' + \tilde{\vC}  
	+\begin{bmatrix} c_{1.}-c_{..} \\ c_{2.}-c_{..}\\\vdots\\ c_{k.}-c_{..}\end{bmatrix}
	 \vone'  + \vone  \begin{bmatrix} c_{.1}-c_{..} \\ c_{.2}-c_{..}\\\vdots\\ c_{.k} - c_{..}\end{bmatrix}' \label{eqn:anova-matrix}
\end{align}
 where $\tilde\vC = ((\tilde{c}_{ij}))$ with $\tilde{c}_{ij}= c_{ij}-c_{i\cdot}-c_{\cdot j} + c_{\cdot\cdot}, 1\leq i,j\leq k$ is independent of the event $\cI_{k,n}$. This immediately gives the second term $\tilde \vZ$ in the structure theorem. Now note that by \eqref{eqn:ik-decomp} on $\cI_{k,n}$ the row sums satisfy
\[ \sqrt{k}\max_{1\leq j\leq k}\{ c_{\cdot\cdot}-c_{\cdot j}\} \leq \sqrt{k}c_{\cdot\cdot}- (b_n + a_{n}^{-1}M_{n-k}) \]
Here the term on the left has distribution $\max_{1\leq j\leq k} \set{\bar Z - Z_i}$ where the $Z_i$ are i.i.d.~standard gaussian random variables and $\bar{Z} = \avg(\set{Z_i}_{1\leq i\leq k})$. Further by the ANOVA decomposition, this random variable is independent of the term on the right which by \eqref{eqn:avg-mn-mn-joint} is of order $\Theta_P(a_n^{-1})$. In fact,  \eqref{eqn:avg-mn-mn-joint} implies that conditional on the event $\{ k^{1/2}c_{\cdot\cdot}\geq b_n  + a_{n}^{-1}\max\{M_{n-k},M'_{n-k}\} \}$, the random variable $a_n(\sqrt{k}c_{\cdot\cdot}- (b_n + a_{n}^{-1}M_{n-k}))$ converges in distribution to $\log(1+Y/G)$.
Thus for the third term in the ANOVA decomposition in \eqref{eqn:anova-matrix}, on the event $\cI_{k,n}$ intuitively one is looking at the distribution of $(Z_1 -\bar Z, Z_2-\bar Z, \ldots, Z_k - \bar Z)$ conditional on $\max_{1\leq j\leq k} \set{\bar Z - Z_i} \leq \Theta_P(a_n^{-1})$ which is exactly the type of event Lemma \ref{lem:tail}\eqref{item:tailb} is geared to tackle. An identical argument applies to the last term in \eqref{eqn:anova-matrix}.  

Define the random variables 
\[
X_{n}=a_{n}(\sqrt{k}c_{\cdot\cdot} - b_n),\quad Y_{n}=M_{n-k},\quad Y'_{n}=M'_{n-k}
\]
and the random vectors 
\begin{align*}
\vV^{(n)} &=a_{n}\sqrt{k}(c_{1\cdot}-c_{\cdot\cdot},\ldots,c_{k\cdot}-c_{\cdot\cdot}),\quad
\vV^{'(n)} =a_{n}\sqrt{k}(c_{\cdot 1}-c_{\cdot\cdot},\ldots, c_{\cdot k}-c_{\cdot\cdot}).
\end{align*}
Note that all these random variables are independent. From equation \eqref{eqn:ik-decomp} we have
\[
\cI_{k,n}= \{X_{n}\ge \max\{Y_{n},Y'_{n}\}, \max_{j}V_{j}^{(n)} \le X_n-Y_n, \max_{j}V_j^{'(n)} \le X_n-Y'_n\}
\]
Define the new random vectors 
\[
\vU^{(n)}=\vV^{(n)}/(X_n-Y_n), \vU^{'(n)}=\vV^{'(n)}/(X_n-Y'_n),
\]
For any  compactly supported continuous function $\psi:\dR^{3+2k}\to \dR$ we have
\begin{align*}
&\E(\psi(X_{n},Y_{n},Y'_{n},\vU^{(n)},\vU^{'(n)}); \cI_{k,n})\\
&=  \int_{\{x\ge \max\{y,y'\}, \max_j\{u_j\} \le 1,  \max_j\{u'_j\} \le 1\}} \psi(x,y,y',\vu,\vu') f_{X_n}(x) f_{Y_n}(y) f_{Y'_n}(y')  \\
&\qquad\cdot f_{\vV^{(n)}}((x-y)\vu) f_{\vV^{'(n)}}((x-y')\vu') (x-y)^{k-1}(x-y')^{k-1} dx dy dy' \prod_{i=1}^{k-1}du_{i} du'_{i}
 \end{align*}
 where $f_{Z}$ is the density of $Z$.  Note that the density of $X_n$ is 
\begin{align*}
f_{X_n}(x)&=\frac{\sqrt{k}}{\sqrt{2\pi}a_n}\exp\bigr(-\frac{k(x+a_{n}b_{n})^2}{2a_n^2}\bigr) \\
&= \frac{\sqrt{k}}{\sqrt{2\pi}a_n}\exp(-{kb_n^2}/{2})\cdot\exp(-kx^2/2a_n^2-kxb_n/a_n)
\end{align*}
for $x\in\dR$. Moreover
\[
\frac{\sqrt{k}}{\sqrt{2\pi}a_n}\exp(-{kb_n^2}/{2})= \frac{\sqrt{k}}{k!}(\sqrt{2\pi}a_n)^{k-1}{n\choose k}^{-1}(1+o(1)).
\]
 Also density of $Y_{n}$ is 
\[
f_{Y_n}(y)= \frac{n-k}{a_n}\phi(b_n+y/a_n)(1-\bar{\Phi}(b_n+y/a_n))^{n-k-1}\to e^{-y}e^{-e^{-y}}\text{ as }n\to\infty
\]
for $y\in\dR$. Similarly
\[
f_{Y'_n}(y')= f_{Y_n}(y')\to e^{-y'}e^{-e^{-y'}}\text{ as }n\to\infty
\]
for $y'\in\dR$. Now using Lemma~\ref{lem:tail}\eqref{item:tailb} we have
\[
a_{n}^{k-1}f_{\vV^{(n)}}(\vt)\to c_{k}
\]
as $n\to\infty$ for any fixed $\vt$ satisfying $t_{1}+t_{2}+\cdots+t_{k}=0$ for some constant $c_{k}$ depending only on $k$.  
A simple calculation now shows that
\begin{align*}
&\E(\psi(X_{n},Y_{n},Y'_{n},\vU^{(n)},\vU^{'(n)})\mid \cI_{k,n})\\
&=  \frac{1}{\pr(\cI_{k,n})} \int_{\{x\ge \max\{y,y'\}, \max_j\{u_j\} \le 1, \max_j\{u'_j\} \le 1\}} \psi(x,y,y',\vu,\vu') f_{X_n}(x) f_{Y_n}(y) f_{Y'_n}(y')  \\
&\qquad\cdot f_{\vV^{(n)}}((x-y)\vu) f_{\vV^{'(n)}}((x-y')\vu') (x-y)^{k-1}(x-y')^{k-1}dx dy dy' \prod_{i=1}^{k-1}du_{i} du'_{i}
\\
&\to c'_k\int_{\{x\ge \max\{y,y'\}, \max_j\{u_j\} \le 1, \max_j\{u'_j\} \le 1, \sum u_j=\sum u'_j=0 \}} \psi(x,y,y',\vu,\vu')\\
&\quad  e^{-kx} e^{-y}e^{-e^{-y}} e^{-y'}e^{-e^{-y'}} (x-y)^{k-1}(x-y')^{k-1} dx dy dy' \prod_{i=1}^{k-1}du_idu'_i\\
&= c'_k\int_{\{g>0, t>0, t'>0, \max_j\{u_j\} \le 1, \max_j\{u'_j\} \le 1, \sum u_j=\sum u'_j=0 \}} \psi(-\log g,-\log(g+t),-\log(g+t'),\vu,\vu')\\
&\quad  g^{k-1} e^{-2g-t-t'}(\log(1+t/g)\log(1+t'/g))^{k-1} dgdtdt' \prod_{i=1}^{k-1}du_idu'_i
\end{align*}
as $n\to\infty$ for some constant $c'_{k}$ depending only on $k$.

Thus the conditional distribution of $a_n(c_{\cdot\cdot} - b_n/\sqrt{k}, c_{1\cdot}-c_{\cdot\cdot},\ldots,c_{k\cdot}-c_{\cdot\cdot}, c_{\cdot 1}-c_{\cdot\cdot},\ldots, c_{\cdot k}-c_{\cdot\cdot})$ converges in distribution to  $k^{-1/2}( -\log G, (kU_1-1) \log(1+T/G), \ldots, (kU_k-1) \log(1+T/G), (kU'_1-1) \log(1+T'/G), (kU'_k-1) \log(1+T'/G)) $ where  $(U_{1},U_{2},\ldots,U_{k})$, $(U'_{1},U'_{2},\ldots,U'_{k})$ are i.i.d.~from Dirichlet$(1,1,\ldots,1)$ distribution and $(G,T,T')$ has joint density
 \[
\propto (\log(1+t/g)\log(1+t'/g))^{k-1}g^{k-1}e^{-t-t'-2g},\ g,t,t'\geq 0. 
\]
The $(\log(1+t/g)\log(1+t'/g))^{k-1}$ term is arising from the $(k-1)$-dimensional volume of the simplexes  $\{\max_{1\leq j\leq k}\{\bar{x}-x_j\}\leq \log(1+t/g)\}$ and $\{\max_{1\leq j\leq k}\{\bar{x}-x_j\}\leq \log(1+t'/g)\}$.}
\end{proof}

\section{Variance Asymptotics}
\label{sec:var-asymp}
The aim of this Section is to prove Theorem \ref{thm:var}, which describes the asymptotic 
behavior of the variance of $L_n(k)$.  We require several preliminary results that have
potential application to the analysis of similar local maxima. 

\vskip.1in

\subsection{Preliminary Results}

\begin{lem}\label{lem:bigmax}
Let $\vU$ be a $s\times t$ matrix of independent standard Gaussian entries. 
For fixed $\theta>0$ and $x, y \in \dR$ there is a constant $\eta(s,t,\theta)>0$
such that 
\begin{align*}
& \pr\left( \max_{1\leq i\leq s} u_{i\cdot}\geq \theta b_{n} + x/a_{n}, \ 
                \max_{1\leq j\leq t} u_{\cdot j} \geq  \theta  b_{n} + y/a_{n} \right) \\
& =  (\eta(s,t,\theta)+o(1)) e^{- st\theta((t-1)x+(s-1)y) / (st-1) }  \
n^{- \frac{st(s+t-2) \theta^2}{st-1}} \
(\log n)^{ \frac{st(s+t-2)\theta^2}{2(st-1)} -1}
\end{align*}
where $u_{i\cdot}, u_{\cdot j}$ are, respectively, the average of the $i$-th row and $j$-th column of $\vU$.
\end{lem}

The heuristic idea behind the proof of Lemma~\ref{lem:bigmax} is the following. If both the maximum row average and maximum column average are bigger than $z$, there will be at least one row (say $i_*$-th row) and one column (say $j_{*}$-th column) with average bigger that $z$. The joint density of the $i_{*}$-th row and $j_{*}$-th column is proportional to $\exp(-(\sum_{i\neq i_*} u_{ij_*}^2 + \sum_{j\neq j_*} u_{i_*j}^2 + u_{i_*j_*}^2 )/2)$. If we minimize $\sum_{i\neq i_*} u_{ij_*}^2 + \sum_{j\neq j_*} u_{i_*j}^2 + u_{i_*j_*}^2 $ under the constraint that $\sum_{i} u_{ij_*}\geq t z,  \sum_{j} u_{i_*j}\geq s z$,  the minimum is achieved at 
\begin{align*}
u_{ij_*} = \frac{(st-s)z}{st-1} \text{ for } i\neq i_*, &\qquad
u_{ij_*} = \frac{(st-t)z}{st-1} \text{ for } j\neq j_{*}\\
\text{and } 
u_{i_*j_*} &= \frac{(2st-s-t)z}{st-1}.
\end{align*}
Plugging in these values in the exponent results in the value $st(s+t-2)z^2/(st-1)$. When $z=\theta b_{n}$, we have 
\[
\exp(- \frac{st(s+t-2)z^2}{2(st-1)})\approx n^{-\frac{st(s+t-2)\theta^{2}}{st-1}},
\]
which is the leading order in the probability. The complete proof is given below. 

\begin{proof}[Proof of Lemma~\ref{lem:bigmax}]
Fix $\theta > 0$ and $x,y \in \bR$, and define $\alpha_n=\sqrt{st}(\theta b_{n} + x/a_{n}) $ and 
$\beta_n=\sqrt{st} (\theta b_{n} + y/a_{n})$.  We wish to bound the probability
\[
p_n : =  
\pr \left(\max_{1\leq i\leq s} u_{i\cdot}\geq \theta b_{n} + x/a_{n}, \, \max_{1\leq j\leq t} u_{\cdot j} \geq  \theta b_{n} + y/a_{n} \right)
\]
\noindent
Let $Z,Z_{1},Z_{2},\ldots, Z_{s},Z'_{1},Z'_{2},\ldots,Z'_{t}$ be independent standard Gaussian random variables,
and define 
\[
V_{s}=\max_{1\leq i\leq s} (Z_{i}-\bar{Z}) \ \mbox{ and } \  V'_{t}= \max_{1\leq j\leq t} (Z'_{i}-\bar{Z'}).
\]
It is easy to see that 
\[
\left( \max_{1 \leq i \leq s} u_{i\cdot} - u_{\cdot\cdot}, \, \max_{1 \leq j \leq t} u_{\cdot j} - u_{\cdot\cdot}, \, u_{\cdot\cdot} \right) 
\ \equald \  
(t^{-1/2} V_{s}, \, s^{-1/2} V'_{t}, \, (st)^{-1/2}Z)
\]
\vskip.04in
\noindent
and it then follows from a routine calculation that
\[
p_n = \pr( V_{s} \geq (\alpha_n- Z)/\sqrt{s}, \, V'_{t}\geq (\beta_n- Z)/\sqrt{t}).
\]
Note that $V_s$ has the same distribution as $\min_{1 \leq i \leq s} (\bar{Z} - Z_i)$, and that a similar relation
holds for $V'_t$.  Thus Lemma~\ref{lem:tail}\eqref{item:tailc} implies that
\[
p_n = g_{s}g_{t}\sqrt{st}
\E \left( h_s(\alpha_n-Z) h_t(\beta_n-Z)
\exp \left(-\frac{(\alpha_n-Z)^2}{2(s-1)}  - \frac{(\beta_n-Z)^2}{2(t-1)}  \right) \right)
\]
where the expectation is with respect to $Z$ and 
\[
h_l(x) :=  \frac{1}{g_l l^{1/2}}\exp\left(\frac{x^2}{2(l-1)}\right) \pr\left( \sum_{i=1}^l Z_{i}-l\cdot \min_{1\le i\le l}Z_i \geq x l^{1/2}\right) 
\]
is a bounded continuous function satisfying $\lim_{x\to\infty} xh_l(x) =1$ for $l\in\{s,t\}$. 
One may easily check  that
\begin{align*}
&\frac{(\alpha_n-z)^2}{s-1}  + \frac{(\beta_n-z)^2}{t-1} + z^2\\
&\qquad= \frac{st-1}{(s-1)(t-1)}\left(z - \frac{(t-1)\alpha_n+(s-1)\beta_n}{st-1}\right)^2 + \frac{t\alpha_n^2+s\beta_n^2-2\alpha_n \beta_n}{st-1}.
\end{align*}
Note that the last term above does not depend on $z$.  Define
\[
q_n \  := \
\exp\left(-\frac{t\alpha_n^2+s\beta_n^2-2\alpha_n\beta_n}{2(st-1)}\right).
\]
  Using the last two displays and the fact that 
$\alpha_n, \beta_n \sim \theta \sqrt{2 st \log n}$, we find
\begin{align*}
p_n 
&= q_n\cdot g_s g_t \sqrt{st} \int_{\bR} h_s\left(\frac{(s-1)(t\ga_n-\gb_n)}{st-1}-z\right) h_t\left(\frac{(t-1)(s\gb_n-\ga_n)}{st-1}-z\right)\\[.1in]
&\qquad\qquad\qquad\qquad\qquad\qquad\qquad \exp\left(-\frac{(st-1)z^2}{2(s-1)(t-1)}\right) dz\\[.1in]
&=  \frac{q_n (st-1)^2 g_s g_t \sqrt{st}(1+o(1)) }{(s-1)(t-1)(t\alpha_n-\beta_n)(s\beta_n-\alpha_n)}\int_{\bR}\exp\left(-\frac{(st-1)z^2}{2(s-1)(t-1)}\right) dz\\[.1in]
&= \frac{\eta(s,t,\theta)(1+o(1))}{\log n}\cdot  q_n
\end{align*}
where $\eta(s,t,\theta)$ is a positive constant.
A straightforward calculation using the definition of $\alpha_n$ and $\beta_n$ shows that 
\[
q_n \ = \ \exp\left(- \frac{st\theta((t-1)x+(s-1)y)}{st-1}\right) \left(\frac{\sqrt{4\pi \log n}}{n}\right)^{\frac{st(s+t-2) \theta^2}{st-1}}(1+o(1)) 
\]
and the proof is complete.
\end{proof}

Another preliminary result needed for the correlation analysis is a joint 
probability estimate for two locally optimal matrices. For integers $s, t \in [k]$, let $\cB_{s,t,k}$ be 
the event that $\vW^k = \vW_{[k]\times[k]}$ is locally optimal as a sub matrix of 
$\vW_{([k] \cup[s+k+1,n]) \times ([k] \cup[t+k+1,n])}$ and the overlapping submatrix 
$\vW_{[s+1,s+k]\times[t+1,t+k]}$ is locally optimal as a sub matrix of  
$\vW_{[s+1,n]\times[t+1,n]}$ (see Figure \ref{fig:bst}). 

\begin{figure}[htbp]
\centering
\includegraphics[scale=.4,page=2]{Bst.pdf}%
\includegraphics[scale=.4,page=1]{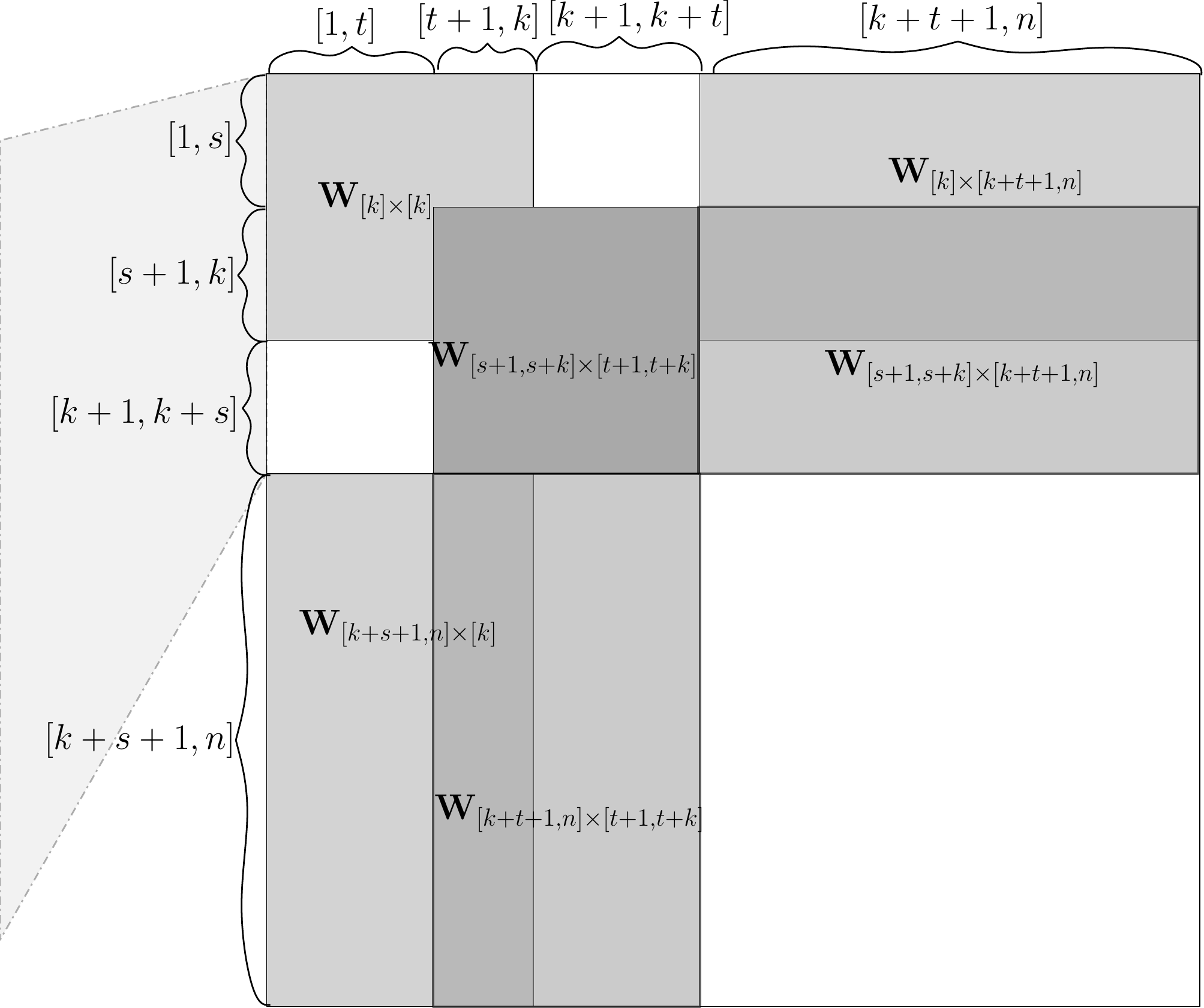}
\caption{A pictorial representation of the event $\cB_{s,t,k}$ and the block matrices $\vX_{i},1\leq i\leq 7$.}
\label{fig:bst}
\end{figure}

\vskip.1in

\begin{lem}\label{lem:st2}
Let $0 < s, t < k$. There exists a constant $\eta(s,t,k)>0$ such that 
\[
\pr(\cB_{s,t,k}) \ \leq \ \eta(s,t,k) \left(\frac{\sqrt{\log n}}{n}\right)^{2k -2k(k-s)(k-t)/(2k^2-st) }.
\]
\end{lem}

\vskip.1in



\begin{proof}[Proof of Lemma~\ref{lem:st2}]

Referring to figure \ref{fig:bst}, we define disjoint matrices $\vX_1,\ldots, \vX_7$ in the following way:
\begin{align*}
\vX_1 =\vW_{[s]\times[t]},\qquad \vX_2 &=\vW_{[s]\times[t+1,k]}, &\\
\vX_3 =\vW_{[s+1,k]\times[t]}, \quad \vX_4 &=\vW_{[s+1,k]\times[t+1,k]},\qquad \vX_5 =\vW_{[s+1,k]\times[k+1,k+t]}\\
 \vX_6 &=\vW_{[k+1,k+s]\times[t+1,k]},\quad \vX_7=\vW_{[k+1,k+s]\times[k+1,k+t]}.
\end{align*}
Let $S_{i} = \avg(\vX_i)$ and $\theta_i$ be the number of entries in $\vX_{i}$.  Clearly, 
\begin{align*}
\theta_1&=\theta_{7}=st  \qquad\qquad\qquad  \theta_{2}=\theta_{6}=s(k-t) \\
\theta_3&=\theta_5=(k-s)t  \quad\qquad \  \theta_{4}=(k-s)(k-t).
\end{align*}
The joint density of $(S_1,\ldots,S_7)$ is given by 
\[
g(s_1,\ldots,s_7) \ = \ \prod_{i=1}^7 \sqrt{\theta_i/2 \pi} \exp(-\theta_i s_i^2/2) .
\] 

Define random variables
\[
M_c=\max_{k+t+1\leq j\leq n}\sum_{i=1}^k W_{i,j}
\ \ \ \
M'_c=\max_{k+t+1\leq j\leq n}\sum_{i=1}^k W_{s+i,j} .
\]
Thus $M_c$ is the maximum column sum of the sub-matrix 
$\vW_{[k] \times [k+t+1,n]}$, and $M'_c$ is the maximum column sum 
of the sub-matrix $\vW_{[s+1,s+k] \times [k+t+1,n]}$. Similarly define
\[
M_r=\max_{k+s+1\leq i\leq n}\sum_{j=1}^k W_{i,j}
\ \ \ \ 
M'_r=\max_{k+s+1\leq i\leq n}\sum_{j=1}^k W_{t+i,j}
\]
to be the maximum row sum of the sub-matrix $\vW_{[k+s+1,n] \times [k]}$ and the 
maximum row sum of the sub-matrix $\vW_{[k+s+1,n] \times [t+1,t+k]}$, respectively. 
For a real number $x \in \dR$, let $\cD(x)$ be the set 
\begin{align*}
\cD(x) = \{(s_{1},s_{2},\ldots,s_{7})\in\dR^{7}\mid\ &t s_1+(k-t)s_2\geq x,\quad s s_1+(k-s)s_3\geq x,\\
& s s_2+(k-s)s_4\geq x,\quad t s_3+(k-t)s_4\geq x,\\
& (k-t) s_4+ts_5\geq x,\quad (k-s) s_4+ss_6\geq x,\\
& (k-s) s_5+ss_7\geq x,\quad (k-t) s_6+ts_7\geq x\}.
\end{align*}
Note that $\cD(x)$ is decreasing in $x$. It is easy to see that
\[
\cB_{s,t,k} \subseteq \{(S_{1},S_{2},\ldots,S_{7}) \in \cD(\min\{M_{r},M'_{r},M_{c},M'_{c}\})\}.
\]
Now as $(M_{r},M'_{r},M_{c},M'_{c})$ is independent of $(S_1,\ldots,S_7)$, 
$M_{r} \equald M'_{r}$, and $M_{c} \equald M'_{c}$ we have
\begin{equation}
\label{eqn:pbstk}
\pr(\cB_{s,t,k}) \leq 2 \E( f(M_r) + f(M_{c}) )
\end{equation}
where $f(x) := \pr((S_{1},S_{2},\ldots,S_{7}) \in \cD(x))$.  We claim that
\begin{align}\label{eq:fxbd}
f(x) \leq \exp\left(- \left(1 -\frac{(k-s)(k-t)}{2k^2-st}\right) x^2\right) \ \ \ \ x > 0
\end{align}
To see this, note first that by standard calculus one can check that $\sum_{i=1}^{7}\theta_{i}s_i^{2}$ 
is minimized over $(s_1,\ldots,s_7) \in \cD(x)$ at $a_{i},\ldots, a_{7}$, where 
\begin{align*}
a_{1}&=a_{7}=\frac{(3k-s-t)x}{2k^2-st} \qquad a_{2}=a_{6}=\frac{(2k-t)x}{2k^2-st} \\
a_{3}&=a_{5}=\frac{(2k-s)x}{2k^2-st}  \qquad \ \ \ a_{4}=\frac{2kx}{2k^2-st}. 
\end{align*}
Note that for $(s_1,\ldots,s_7) = (a_{1},\ldots,a_{7})$ all the inequalities defining $\cD(x)$ 
become equalities.  In particular, we have
\begin{align*}
f(x) &:=\pr((S_{1},S_{2},\ldots,S_{7})\in \cD(x))\\
&=  \int_{\cD(x)} \prod_{i=1}^7\sqrt{\theta_i/2\pi}\exp(-\theta_i s_i^2/2)ds_i\\
&= \int_{\cD(0)} \prod_{i=1}^7\sqrt{\theta_i/2\pi}\exp(-\theta_i (s_i+a_i)^2/2)ds_i\\
&= \prod_{i=1}^7\sqrt{\theta_i/2\pi}\exp(-\theta_i a_i^2/2) \int_{\cD(0)} \exp(-\sum_{i=1}^{7}\theta_i (s_i^2/2+a_is_i))ds_i.
\end{align*}
Further note that
\begin{align*}
\frac{2k^2-st}{x}\sum_{i=1}^{7}\theta_i a_is_i = &s(2k-s-t)(t s_1+(k-t)s_2 + (k-t) s_6+ts_7)\\
&
+ kt(s s_1+(k-s)s_3 +  (k-s) s_5+ss_7) \\[.1in]
&
+ s(k-t)(s s_2+(k-s)s_4 + (k-s) s_4+ss_6)\\[.1in]
&
+(k-s)(k-t)(t s_3+(k-t)s_4 + (k-t) s_4+ts_5)
\end{align*}
which is non-negative under $\cD(0)$. Thus we have
\begin{align*}
f(x) &\leq \prod_{i=1}^7\sqrt{\theta_i/2\pi}\exp(-\theta_i a_i^2/2) \int_{\cD(0)} \exp(-\sum_{i=1}^{7}\theta_i s_i^2/2)ds_i\\
&\leq \exp(-\sum_{i=1}^7\theta_i a_i^2/2).
\end{align*}
Simplifying we have
\[
\frac{1}{2}\sum_{i=1}^7\theta_i a_i^2= \left(1 -\frac{(k-s)(k-t)}{2k^2-st}\right) x^2.
\]
This proves the claim~\eqref{eq:fxbd}. 

Now note that $M_{r} \equald \sqrt{k} V_{n-k-s}$ and $M_{c}\equald \sqrt{k} V_{n-k-t}$ 
where $V_n=\max\{Z_1,Z_2,\ldots,Z_n\}$ is the maximum of $n$ independent ${\mathcal N}(0,1)$
random variables.  Combining (\ref{eq:fxbd}) and (\ref{eqn:pbstk}), we complete the proof by showing
that for any constant $\theta > 0$,
\[
\E(\exp(-\theta \max\{V_n,0\}^2 )) \leq \gamma(\theta) \exp(-\theta b_n^2)
\]
for some constant $\gamma(\theta) > 0$ where $b_{n}$ satisfies $e^{-b_n^2 / 2} = \sqrt{2 \pi} b_{n} / n$. 
Letting 
\[
\theta=(1-(k-s)(k-t)/(2k^2 -st)) \mbox{ and } b_n=\sqrt{2\log n} - \log(4\pi\log n)/\sqrt{8\log n} 
\]
then gives the asserted bound for $\pr(\cB_{s,t,k})$.
The following lemma completes the proof.
\end{proof}

\begin{lem}\label{lem:expmom}
Let $V_n: = \max\{ Z_1, Z_2, \ldots, Z_n \}$ be the maximum of $n$ independent
${\mathcal N}(0,1)$ random variables. For any constant $\theta>0$
\[
\E(\exp(-\theta \max\{V_n,0\}^2 )) \, \leq \, \gamma(\theta) \exp(-\theta b_n^2)
\]
for some constant $\gamma(\theta) > 0$ for all $n \geq 1$ where 
$b_{n}=\sqrt{2\log n} - \log(4\pi\log n)/\sqrt{8\log n}$. 
\end{lem}

\begin{proof}[Proof of Lemma~\ref{lem:expmom}]
Define $X_{n}=b_{n}(b_{n}-V_{n})$.  Then
\begin{align*}
\E(\exp(-\theta \max\{V_n,0\}^2 )) 
& \leq \pr(V_n < 0) +  \E(\exp( -\theta (b_n -  X_n/b_n)^2)\ind\{V_n\geq 0\}) \\[.1in]
& \leq 2^{-n} + \exp(-\theta b_n^2 ) \E(\exp(2\theta X_n) \ind\{V_n\geq 0\}).
\end{align*}
It is easy to see that $2^{-n} \exp(\theta b_n^2 )$ is uniformly bounded in $n$, and it suffices to 
show that the same is true of $\E(\exp(2 \theta X_{n}) \ind \{V_n\geq 0\})$. 
For each $c > 0$ it is clear that 
$\E(\exp(\theta X_{n})\ind\{X_n\leq c\}) \leq \exp(\theta c)$ for every $n$. 
Moreover, $V_{n} \geq 0$ implies $X_n \leq b_n^2$, so it suffices to bound 
\begin{equation}
\label{ethetaxn}
\E(\exp(\theta X_{n})\ind\{c\leq X_n\leq b_n^2\})
\end{equation}
for any fixed $c > 0$.  (An appropriate choice of $c$ is given below.)
To this end, note that
\[
\pr(X_{n}\geq x) \, = \, (1-\bar{\Phi}(b_n-x/b_n))^n \, \leq \, \exp(-n \bar{\Phi}(b_n-x/b_n) ).
\]
Using the bound $\bar{\Phi}(u)\geq u^2/(\sqrt{2\pi}(1+u^2))e^{-u^2/2}$ from Lemma~\ref{lem:gauss_tail} we have
\[
n\bar{\Phi}(b_n-x/b_n) \geq \frac{ne^{-b_n^2/2}}{\sqrt{2\pi}b_n}\cdot \frac{b_n^2-x}{1+(b_n-x/b_n)^2}e^{-x^2/2b_n^2} \cdot 
 e^{-x}. 
\]
Clearly $ne^{-b_n^2/2}/(\sqrt{2\pi}b_n)=1+o(1)$.  Let 
\[
B:=\inf_{x\in[0,b_n]}\frac{b_n^2-x}{1+(b_n-x/b_n)^2}e^{-x^2/2b_n^2} \, > \, 0
\]
and define $C:=\min\{B, \theta/e\}$.  It follows from the calculation above that
\[
\pr(X_{n}\geq x) \leq\exp(-Ce^x )
\]
for all $x \in[0,b_n]$.  

In order to bound the expectation in (\ref{ethetaxn}) we will identify an appropriate constant
$c^*$ and break the interval $[c^*,b_{n}^{2}]$ into subintervals where the contribution of each subinterval can be easily bounded. Define $t:=2\theta/C$. Let $x_{0} = b_n^2$ and let $x_{i+1} = \log(t x_i)$ for $i \geq 0$. 
Note that $x_1 \leq b_{n}$ for $n$ sufficiently large. 
Let $c_{*}$ be the largest solution to the equation $x = \log(tx)$ so that $tc_{*}=e^{c_*}$.  The definition of $C$ ensures that $t > e$, therefore the equation $x=\log(tx)$ has two solutions and moreover $c_{*} > 1$. 
It is easy to see that $x_{i}\to c_{*}$ as $i \to \infty$. 
Thus there exists $k$ such that $c^*< x_{k+1} < 2c_* \leq x_1, \ldots, x_k$,
and therefore
\begin{align*}
\E(\exp(\theta X_{n})\ind\{2c^*\leq X_n\leq b_n^2\}) &\leq \sum_{i=0}^k \E(\exp(\theta X_{n})\ind\{x_{i+1}\leq X_n\leq x_{i}\})\\
& \leq \sum_{i=0}^k \exp(\theta x_{i})\pr(X_n\geq x_{i+1}) \\
&\leq \sum_{i=0}^k \exp(\theta x_{i} - Ce^{x_{i+1}})
\leq  \sum_{i=0}^{k} \exp(-\theta x_{i}) 
\end{align*}
where in the last inequality we have used the definition of $x_{i+1}$.  
Using convexity we have  $e^{x}> e^{c_*}(1+x-c_*)=tc_*(1+x-c_*)$ for all $x> c_{*}$. It follows from the definition that $x_{i}= e^{x_{i+1}}/t >  c_*(1+x_{i+1}-c_*)$, and therefore 
$x_{i}-c_{*} > c_{*}(x_{i+1}-c_{*})\geq c_{*}^{k-i}(x_k-c_{*})\geq c_{*}^{k+1-i}$ for $i = 0,\ldots,k-1$. Hence

\hspace*{1.5in}$\displaystyle
\sum_{i=0}^{k} \exp(-\theta x_{i})  \le \sum_{i=0}^{k} \exp(-\theta c_* - \theta c_{*}^{k+1-i})=O(1). $
\end{proof}

\vskip.1in

\subsection{Variance Bound~:~Proof of Theorem~\ref{thm:var}}
Let 
\[
p_n=p_{n}(k) = \pr(\vW^k \text{ is locally optimal as a sub matrix of } \vW^n).
\] 
Theorem~\ref{thm:st1} shows that $p_{n}=(1+o(1)) \theta_{k} {n\choose k}^{-1} (\log n)^{-(k-1)/2}$.
By symmetry we may write
\begin{align*}
 \var(L_{n}(k))
& = \sum_{\gl,\gc\in\sS_{n}(k)} \cov(\ind\{\vW_{\gl}\text{ is locally optimal}\}, \ind\{\vW_{\gc}\text{ is locally optimal}\}) \\[.1in]
& = {n \choose k}^{2} \sum_{s=0}^{k} \sum_{t=0}^{k} {k \choose s} {k \choose t} {n-k \choose s}{n-k \choose t} \cdot \\[.1in]
& \ \ \ \ \ \ \ \cov( \ind\{\vW^k \text{ is locally optimal} \}, \ind\{\vW_{[s+1,s+k]\times[t+1,t+k]}\text{ is locally optimal}\}).
\end{align*}
For $0 \leq s,t \leq k$ define the quantity
\begin{align}\label{eq:vnst}
v_n(s,t)
\, := \, {n\choose k}^{2}{k\choose s}{k\choose t} &{n-k\choose s}{n-k\choose t} \cdot
\\[.1in] 
& \hskip-.9in
\cov( \ind\{ \vW^k \text{ is locally optimal}\},
\ind\{\vW_{[s+1,s+k]\times[t+1,t+k]}\text{ is locally optimal}\}) \notag
\end{align}

We first analyze the case $k=1$, which is relatively straightforward.
When $k=1$ we have $p_n = 1 / (2n-1)$, so that
\[
v_{n}(0,0)= n^2p_n(1-p_n)=\frac{2n^3}{(2n-1)^2}=\frac{1}{2}n(1+o(1))
\]
and 
\[
v_{n}(0,1)=v_{n}(1,0)= - n^2(n-1)p_n^2=-\frac{n^2(n-1)}{(2n-1)^2}=-\frac{1}{4}n(1+o(1)).
\]
Now note that
\begin{eqnarray*}
\lefteqn{ \pr( w_{11} \text{ is locally optimal}, w_{22} \text{ is locally optimal} ) } \\[.1in] 
& = &
2 \pr\left( w_{11}=\max_{i=1,2 ; 1\leq j\leq n} \{w_{ij}, w_{ji}\}, w_{22} = \max_{1 \leq j \leq n} \{w_{2j}, w_{j2}\} \right) \\[.1in]
& = &
2 / ((4n-4)(2n-1)) .
\end{eqnarray*}
Therefore
\[
v_{n}(1,1) \, = \, 
n^2(n-1)^2 \left( \frac{2}{4(2n-1)(n-1)} - p_n^2 \right) 
\, = \, 
\frac{n^2(n-1)}{2(2n-1)^2} = \frac{1}{8}n(1+o(1))
\]
Combining the previous relations yields
\[
\var(L_{n}(1))=\frac{n^2(n-3)}{2(2n-1)^2} = \frac{1}{8}n(1+o(1)).
\]
as desired.

We now establish the variance asymptotics of $L_n(k)$ for fixed $k \geq 2$.  Our argument considers diferent cases, 
depending on the values of $s$ and $t$.  We find that
the dominant contribution comes from the case $s=t=k$, \ie~when the matrices under consideration are 
share no rows or columns. In particular, $v_{n}(k,k)\approx n^{2k^2/(k+1)}=n^{2k-2+2/(k+1)}$ with logarithmic corrections. 

\vskip.2in

\noindent{\bf Case $\mathbf1$.} $s=t=0$:
In this case the matrices are the same, and therefore
\[
0 < v_n(0,0) = {n \choose k}^{2} p_{n} (1-p_{n}) = O(n^k (\log n)^{-(k-1)/2}) .
\]

\vskip.1in

\noindent{\bf Case $\mathbf2$.} $s=0, t>0$ or $s>0, t=0$: 
In this case the matrices have identical row or column sets, but do not overlap. 
It is clear that both matrices cannot be locally optimal at the same time, so the covariance of the indicators is $-p_n^2$,
and the contribution to the overall variance is $|v_n(s,t)| = O(n^{2k+s+t} \, p_n^2) = O(n^k(\log n)^{1-k})$.

\vskip.1in

\noindent{\bf Case $\mathbf3$.} $0<s, t<k$: 
In this case the two submatrices of interest have $k-s > 0$ common rows and $k-t > 0$ common columns. 
Lemma~\ref{lem:st2} implies that
\begin{align*}
0 & \leq \cov(\ind\{\vW^k \text{ is locally optimal}\},
\ind\{\vW_{[s+1,s+k] \times[t+1,t+k]} \text{ is locally optimal}\})\\[.05in]
& \leq \eta(s,t,k) (\log n/n^2)^{k - k(k-s)(k-t)/(2k^2-st) }
\end{align*}
and we therefore obtain the bound 
\begin{equation}
\label{vnst}
0 \, \leq \, v_{n}(s,t) \, = \, O(n^{s+t +2k (k-s)(k-t)/(2k^2-st) } \, (\log n)^{k - k(k-s)(k-t)/(2k^2-st)})
\end{equation}
Note that 
\[
\frac{2k (k-s)(k-t)}{2k^2-st} 
\ = \ 
\frac{2k}{\frac{k(3k-s-t)}{(k-s)(k-t)}-1}
\ \leq \ 
\frac{2k}{\frac{k(3k-s-t)}{(k-(s+t)/2)^2}-1}.
\]
Thus, defining $\theta := (s+t) / 2k$, we find that
\[
s+t +\frac{2k (k-s)(k-t)}{2k^2-st} 
\ \leq \
2k \left(\theta + \frac{(1-\theta)^2}{2-\theta^2} \right) 
\ =\ 
2k \left(\frac{3-\theta^3}{2-\theta^2} -1 \right).
\]
The derivative
\[
\frac{d}{d\theta} \frac{3-\theta^3}{2-\theta^2} 
\ = \ 
\frac{6(1-\theta)+\theta^3}{(2-\theta^2)^2}
\] 
is positive, so $(3-\theta^3) / (2-\theta^2)$ is a strictly increasing function of $\theta$,
which takes values in $[1/k, 1-1/k]$.  Thus for all $1 \leq s,t \leq k-1$ the bound 
(\ref{vnst}) on $v_n(s,t)$ is maximized when $s = t = k-1$, and in this case
\begin{align*}
s+t +\frac{2k (k-s)(k-t)}{2k^2-st}  
& \ = \ 2k-2 + \frac{2k}{k^2+2k-1}\\
&\ =\ \frac{2k^2}{k+1} - \frac{2(k-1)}{(k+1)(k^2+2k-1)}.
\end{align*}
Thus, for $0<s,t<k$ we have
\[
0\leq v_{n}(s,t) = O\left(n^{\frac{2k^2}{k+1} - \frac{2(k-1)}{(k+1)(k^2+2k-1)}} \, (\log n)^{k-\frac{k}{k^2+2k-1}} \right).
\]

\vskip.1in

\noindent{\bf Case $\mathbf 4$.} $s=t=k$: 
In this case the two submatrices of interest have no common rows or columns.
We will show that 
\[
v_{n}(k,k)= (\nu_k+o(1))n^{2k^2/(k+1)}(\log n)^{-k^{2}/(k+1)}
\]
for some constant $\nu_{k}>0$.  Define events
\[
\cI_{n-k} = \big\{ \vW_{[k] \times [k]} \mbox{ is locally optimal as a submatrix of } 
\vW_{([k] \cup [2k+1,n]) \times ([k] \cup [2k+1,n])} \big\}
\] 
and 
\[
\cI'_{n-k} = \big\{ \vW_{[k+1,2k] \times[k+1,2k]} \mbox{ is locally optimal as a submatrix of } 
\vW_{[k+1,n]\times[k+1,n]} \big\}. 
\]
These two events are independent and 
$\pr(\cI_{n-k}) = \pr(\cI'_{n-k}) = p_{n-k}$. 
Let $\vW^{*} = (w^*_{ij})_{k\times k}$ and $\vW^{**} = (w^{**}_{ij})_{k\times k}$ denote the 
matrices $\vW_{[k] \times [k]}$ and $\vW_{[k+1,2k] \times [k+1,2k]}$, respectively, conditional on 
the event $\cI_{n-k} \cap \cI'_{n-k}$.  
Finally, let 
\[
\vC:=\vW_{[k]\times[k+1,2k]} \mbox{ and } \vC':=\vW_{[k+1,2k]\times[k]} 
\] 
be the submatrices capturing the dependence between the local optimality of 
$\vW^{*}$ and $\vW^{**}$ in the full matrix $\vW^n$. 
See figure~\ref{fig:stk} for an illustration of the submatrices
under study.

\begin{figure}[hbtf]
\includegraphics[width=6cm,page=3]{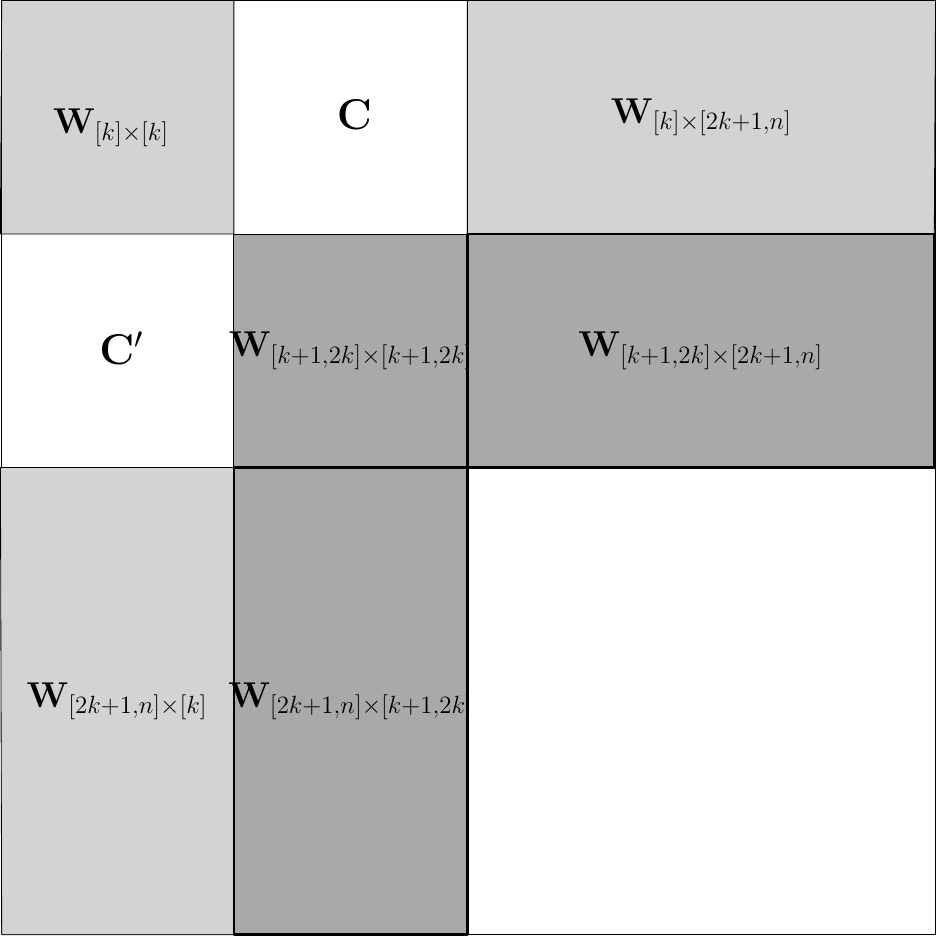}
\label{fig:stk}
\caption{The event $\cI_{k,n-k}$ corresponds to the matrix $\vW_{[k]\times[k]}$ being optimal in the 
light gray region and the event $\cI'_{k,n-k}$ corresponds to the matrix $\vW_{[k+1,2k]\times[k+1,2k]}$ 
being optimal in the dark gray region.}
\end{figure}

Let $\cI_n$ be the event that $\vW_{[k] \times [k]}$ is locally optimal in $\vW^n$, and let $\cI_n'$ be the event that
$\vW_{[k+1,2k] \times [k+1,2k]}$ is locally optimal in $\vW^n$.  
By conditioning on $\cI_{n-k} \cap \cI^\prime_{n-k}$ we have 
\begin{eqnarray}
\pr( \cI_n \cap \cI_n' )  
& = & p_{n-k}^{2} \, 
\E \big( \ind \left\{ \min w^{*}_{i \cdot} \geq \max c'_{i \cdot} ,  \min w^{*}_{\cdot j} \geq \max c_{\cdot j} \right\}  
\nonumber \\[.04in]
\label{pnk2}
& & \hskip.6in \cdot \ind \left\{ \min w^{**}_{i \cdot} \geq \max c_{i \cdot}, \min w^{**}_{\cdot j} \geq \max c'_{\cdot j} \right\} \big) 
\end{eqnarray}
Here and in what follows minima and maxima are taken over appropriate row or column index sets of size $k$.
The standard ANOVA decomposition ensures that the random variables 
\[
\max c_{i\cdot} - c_{\cdot\cdot}
\ \ \ 
\max c_{\cdot j} - c_{\cdot\cdot} 
\ \ \ 
\max c'_{i\cdot} - c'_{\cdot\cdot} 
\ \ \ 
\max c'_{\cdot j} - c'_{\cdot\cdot}
\ \ \ 
c_{\cdot\cdot}
\ \ \ 
c'_{\cdot\cdot}
\]
are mutually independent.  Now let $d_{\cdot\cdot}$ and $d'_{\cdot\cdot}$ be independent copies of 
$c_{\cdot\cdot}$ and $c'_{\cdot\cdot}$, respectively.  One may readily verify that the random triple
\[
(\vW^{**},d_{\cdot \cdot}+\max c_{i\cdot}-c_{\cdot\cdot}, \max c'_{\cdot j} -c'_{\cdot\cdot}+d'_{\cdot\cdot})
\]
is an independent copy of the triple $(\vW^{*},\max c'_{i\cdot}, \max c_{\cdot j})$. 
Therefore
\begin{align*}
p_n 
&= p_{n-k} \pr\left(   \min w^{*}_{i \cdot} \geq \max c'_{i \cdot}, \, \min w^{*}_{\cdot j} \geq \max c_{\cdot j}  \right)\\[.1in]
&=  p_{n-k} \pr\left(  \min w^{**}_{i \cdot} \geq \max c_{i \cdot} - c_{\cdot\cdot} + d_{\cdot\cdot}, \, 
                                 \min w^{**}_{\cdot j} \geq \max c'_{\cdot j} - c'_{\cdot\cdot} + d'_{\cdot\cdot} \right)
\end{align*}
and, using independence of the triples,
\begin{align*}
p_{n}^{2} 
\ = \ 
p_{n-k}^2 \, \E \big( &\ind\left\{ \min w^{*}_{i \cdot} \geq \max c'_{i \cdot}, \, \min w^{*}_{\cdot j} \geq \max c_{\cdot j} \right\}\\
& \cdot \ind\left\{ \min w^{**}_{i \cdot} \geq \max c_{i \cdot} - c_{\cdot\cdot} + d_{\cdot\cdot}, \, 
                             \min w^{**}_{\cdot j} \geq \max c'_{\cdot j} - c'_{\cdot\cdot} + d'_{\cdot\cdot} \right\} \big).
\end{align*} 
Combining the previous two equations with (\ref{pnk2}), we find that  
\begin{align*}
&\cov( \cI_n, \cI_n' ) \ = \ \pr( \cI_n \cap \cI_n') - p_n^2 \\
& = p_{n-k}^{2} \E\biggl( \ind \{ \min w^{*}_{i\cdot} \geq \max c'_{i\cdot}, \, \min w^{*}_{\cdot j} \geq \max c_{\cdot j}\} \\
&\qquad\qquad\qquad\cdot \Big[ \ind \{ \min w^{**}_{i\cdot} \geq \max c_{i\cdot}, \,  \min w^{**}_{\cdot j} \geq \max c'_{\cdot j}\} \\
&\qquad\qquad - \ind\{ \min w^{**}_{i\cdot} \geq \max c_{i\cdot} - c_{\cdot\cdot}+d_{\cdot\cdot}, 
                                   \, \min w^{**}_{\cdot j} \geq \max c'_{\cdot j} -c'_{\cdot\cdot}+d'_{\cdot\cdot}\} \Big] \biggr).
\end{align*}
Now define random variables
\begin{align*}
E:=  \min w^{*}_{i\cdot} - \max (c'_{i\cdot}-c'_{\cdot\cdot}), &\qquad F:= \min w^{*}_{\cdot j} - \max (c_{\cdot j}-c_{\cdot\cdot}), \\
G:= \min w^{**}_{i\cdot} - \max (c_{i\cdot} - c_{\cdot\cdot}), &\qquad
H:= \min w^{**}_{\cdot j} - \max (c'_{\cdot j} -c'_{\cdot\cdot}).
\end{align*}
Note that $E,F,G,H$ are independent of $c_{\cdot\cdot},c'_{\cdot\cdot},d_{\cdot\cdot},d'_{\cdot\cdot}$ and that, by construction, 
$c_{\cdot\cdot} \stackrel{d}{=} d_{\cdot \cdot}$ and $c'_{\cdot\cdot}  \stackrel{d}{=} d'_{\cdot \cdot}$. Thus we have
\begin{eqnarray*}
\lefteqn{ p_{n-k}^{-2} \, \cov( \cI_n, \cI_n' ) } \\[.1in]
& = & \E( \ind \{ E \geq  c'_{\cdot \cdot} ,  F \geq c_{\cdot \cdot} \} \,
( \ind \{ G \geq  c_{\cdot \cdot},  H \geq  c'_{\cdot \cdot} \} - \ind\{ G \geq d_{\cdot \cdot},  H \geq d'_{\cdot \cdot}\})) \\[.13in]
& = & \E\Big\{ \E\big[ \pr(c_{\cdot\cdot}\leq \min\{F,G\})\pr(c'_{\cdot\cdot}\leq \min\{E,H\})\\
& & \qquad - \pr(c_{\cdot\cdot}\leq F)\pr(c_{\cdot\cdot}\leq G)\pr(c'_{\cdot\cdot}\leq E)\pr(c'_{\cdot\cdot}\leq H) 
       \mid E,F,G,H \, \big] \Big\} \\[.1in]
& = & \E\Big\{ \E\big[ \pr(c_{\cdot\cdot} \leq \min\{F,G\}) \pr(c'_{\cdot\cdot} \leq \min\{E,H\})\\
& & \qquad (1 - \pr(c_{\cdot\cdot} \leq \max\{F,G\}) \pr(c'_{\cdot\cdot} \leq \max\{E,H\}) ) \mid E,F,G,H \, \big] \Big\} \\[.1in]
& = & \E\Big\{ \E\big[ \pr(c_{\cdot\cdot} \leq \min\{F,G\})\pr(c'_{\cdot\cdot} \leq \min\{E,H\} )\\
& & \qquad \big( \pr(c_{\cdot\cdot} \geq  \max\{F,G\}) \, + \, 
       \pr(c_{\cdot\cdot}\leq  \max\{F,G\}) \pr(c'_{\cdot\cdot}\geq  \max\{E,H\} ) \big) \mid E,F,G,H \big] \Big\} ,
\end{eqnarray*}
where in the last step we used the elementary identity $1 - P(A) P(B) = P(A^c) + P(A) P(B^c)$.
The Structure Theorem \ref{thm:st1} ensures that $a_{n}(w^{*}_{i\cdot} - b_{n}/\sqrt{k})$ and 
$a_{n}(w^{*}_{\cdot j} - b_{n}/\sqrt{k})$, and the analogous quantities involving $w^{**}$, are tight. 
Thus the previous display yields
\begin{eqnarray*}
p_{n-k}^{-2} \, \cov( \cI_n, \cI_n' ) 
& = & 
(2+o(1)) \pr(c_{\cdot\cdot}\geq \max\{F,G\}) \\[.06in]
& = &
(2+o(1)) \pr( \max c_{\cdot j}\geq \min w^{*}_{\cdot j},\max c_{i\cdot}\geq \min w^{**}_{i\cdot}).
\end{eqnarray*}
Applying Lemma~\ref{lem:bigmax} with $s=t=k$ and $\theta = 1 / \sqrt{k}$ we have 
\begin{align*}
v_{n}(k,k) \ & = \  {n\choose k}^{2}{n-k\choose k}^{2} p_{n-k}^{2} (2+o(1)) 
\pr( \max c_{\cdot j} \geq \min w^{*}_{\cdot j}, \, \max c_{i \cdot} \geq \min w^{**}_{i \cdot}) \\[.1in]
& = \ \Theta\big( n^{2k - 2k/(k+1)}(\log n)^{k/(k+1)-1-(k-1)}\big)\
 = \ \Theta\big( (n/\sqrt{\log n})^{2k^2/(k+1)} \big).
\end{align*}

\vskip.1in

\noindent{\bf Case $\mathbf 5$.} $s<k$ and $ t=k$:  In this case note that
\begin{align*}
&\pr(\vW_{[k]\times[k]}\text{ is locally optimal}, \vW_{[s+1,s+k]\times[k+1,2k]}\text{ is locally optimal})\\
&\leq \pr(\vW_{[k]\times[k]}\text{ is row optimal}, \vW_{[s+1,s+k]\times[k+1,2k]}\text{ is row optimal})
= {n\choose k}^{-2}.
\end{align*}
Thus we have 
$
|v_{n}(s,k)|= O(n^{2k+s+k-2k})=O(n^{2k-2})
$
 for $s\leq k-2$. We need to consider the case $t=k,s=k-1$ separately as $2k-1>2k^{2}/(k+1)$. However, using a similar analysis done in case $4$ and the fact that $\pr(\max c_{i \cdot }\geq \max w^{**}_{i\cdot}) =O(\sqrt{\log n}/n)$ where $\vC=\vW_{[k-1]\times[k+1,2k]}$ we have 
\[
|v_{n}(k-1,k)|=O(n^{2k+2k-1-2k-1}\sqrt{\log n}) = O(n^{2k-2}\sqrt{\log n}).
\]
Note that in the case when $s=k-1,t=k$, the number of sub matrix  pairs and covariance term balance each other in a subtle way.

\vskip.1in 

\noindent{\bf Case $\mathbf 6$.} $s=k$ and $ t<k$: Similar to Case $5$.  \\
 
Combining everything we finally have
\[
\var(L_{n}(k))=(\nu_k+o(1)) (n/\sqrt{\log n})^{2k^2/(k+1)}
\]
for some constant $\nu_{k}>0$ where the $o(1)$ term decays like \\[.1in]
\hspace*{1.5in}$\displaystyle
(\log n/n^2)^{\frac{k-1}{(k+1)(k^2+2k-1)}}(\log n)^{2k-1}.$\hfill\qed

\subsection{Local versus Global Optima}
\begin{proof}[Proof of Corollary \ref{cor:glob-loc}]
Fix numbers $c_n > 0$ such that $c_n \, a_n \to \infty$.
To simplify what follows, let 
$\tilde{L}_n(k) = L_{n}(k: k^{-1/2} b_n - c_n)$.
Note that $0 \leq \tilde{L}_n(k) / L_n(k) \leq 1$ for each $n$, so it suffices to show 
that the expected value of the ratio tends to one.  
Abbreviating ``locally optimal'' by ``loc-opt'', elementary calculations show that
\begin{align*}
\E \tilde{L}_{n}(k)  &= 
\sum_{\gl \in \sS_{n}(k)} 
\pr \big(\vW_{\gl} \text{ loc-opt } \vW^n \ \text{and} \, \avg(\vW_{\gl}) \geq  k^{-1/2} b_n - c_n \big) \\[.1in]
& = 
\sum_{\gl \in \sS_{n}(k)} 
\pr \big(\vW_{\gl} \text{ loc-opt } \vW^n \big) 
\pr \big( \avg(\vW_{\gl}) \geq k^{-1/2} b_n - c_n \, \big| \, \vW_{\gl} \text{ loc-opt } \vW^n \big)  \\[.1in]
& = 
{n \choose k}^2 \pr(\cI_{k,n}) 
\pr \big( \avg(\vW^k) \geq k^{-1/2} b_n - c_n \, \big| \, \cI_{k,n} \big)  \\[.1in]
& = 
\E L_n(k) \cdot \pr \big( \avg(\vW^k) \geq k^{-1/2} b_n - c_n \, \big| \, \cI_{k,n} \big)  .
\end{align*}
Rearranging, we have
\begin{equation}
\label{ratioe}
\frac{\E \tilde{L}_{n}(k) }{\E L_n(k)}
\ = \ 
\pr \big( \avg(\vW^k) \geq k^{-1/2} b_n - c_n \, \big| \, \cI_{k,n} \big)  
\end{equation}
It follows from Theorem \ref{thm:st1} that, conditional on $\cI_{k,n}$,
\[
\avg(\vW^k) \ = \ \frac{b_n}{k^{1/2}} + \frac{\avg(R_n)}{k^{1/2} \, a_n}
\]
where $\avg(R_n) \probd -\log(G)$, and in particular, $\avg(R_n) = O_P(1)$.  Thus our 
assumption on $c_n$ ensures that $\E \tilde{L}_{n}(k) / \E L_n(k) \to 1$ as $n$ tends 
to infinity.  To complete the proof, note that
\begin{eqnarray*}
\left| \, \E \left( \frac{\tilde{L}_{n}(k)}{L_n(k)} \right) - \frac{\E \tilde{L}_{n}(k) }{\E L_n(k)} \, \right|
& \leq &
\E \left| \, \frac{\tilde{L}_{n}(k)}{L_n(k)} \cdot \frac{\E L_n(k) - L_{n}(k) }{\E L_n(k)} \, \right| \\[.1in]
& \leq &
\E \left| \, \frac{L_{n}(k) - \E L_n(k)}{\E L_n(k)} \, \right| \leq 
\frac{ \var(L_{n}(k))^{1/2} }{ \E L_n(k) } 
\end{eqnarray*}
where in the last two steps we have made use of the fact that $\tilde{L}_n(k) / L_n(k) \leq 1$
and Jensen's inequality.  It follows from Theorems \ref{thm:mean} and \ref{thm:var} that the
final term above tends to zero with increasing $n$, and this completes the proof.
\end{proof}

\vskip.3in

\section{Proof of the Central limit theorem}
\label{sec:clt} 
The last section analyzed first and second order properties of the number of local optima $L_n(k)$. The aim of this section is to prove the Central Limit Theorem~\ref{thm:clt} for $L_n(k)$, for fixed $k\geq 2$. For submatrix $\gl=I\times J\in\sS_n(k)$ define 
\[
\cI_{\gl}:=\ind\{\vW_{\gl} \text{ is locally optimal for } \vW_{[n]\times[n]}\}.
\]
Write 
$
L:=L_{n}(k)=\sum_{\gl\in\sS_{n}(k)}\cI_{\gl}
$
for the total number of locally optimal sub matrices of size $k\times k$. To emphasize the dependence on the underlying matrix $\vW:= \vW_{[n]\times [n]}$, when necessary we will write $\cI_{\gl}(\vW),L(\vW)$ instead of $\cI_{\gl},L$ respectively.  

Let 
\[
p_{n}=\E(\cI_{\gl}),\ \mu=\E(L)={n\choose k}^{2} p_{n} \text{ and }\gs^{2}=\var(L).
\]
From Theorem~\ref{thm:mean} and Theorem~\ref{thm:var} we have
\[
\mu= \frac{\theta_k n^k}{k!(\log n)^{(k-1)/2}}(1+o(1)) \text{ and } \gs^{2}=\frac{\nu_k n^{2k^2/(k+1)}}{(\log n)^{k^2/(k+1)}}(1+o(1))
\]
for some constant $\theta_{k},\nu_{k}>0$.  Thus
\begin{align}\label{eq:vmratio}
\frac{\gs}{\mu} =(1+o(1))\frac{\alpha_k}{n^{k/(k+1)}(\log n)^{1/(2k+2)}}=o(1).
\end{align}
where $\alpha_k = k!\nu_k/\theta_k> 0$.
Let $\vW'=((w'_{ij}))$ be an i.i.d.~copy of the underlying matrix $\vW$. 
For any fixed submatrix $\gl=I\times J\in\sS_{n}(k)$, define 
\begin{align*}
w^{\gl}_{ab}=
\begin{cases}
w'_{ab} &\text{ if either $a\in I$ or $ b\in J$}\\
w_{ab} &\text{ if $a\notin I$ and $b\notin J$},
\end{cases}
\end{align*}
$\vW^{\gl}=((w^{\gl}_{ij}))$ and $L^{\gl}:=L(\vW^{\gl})$. Thus we replace {\bf all $n$} entries for the row set and column set of $\lambda$ by independent and identical entries $w_{ab}^\lambda$.  If $\gl$ is chosen uniformly at random from $\sS_{n}(k)$, it is easy to see that $\vW^{\gl}$ and $\vW$ form an exchangeable pair. However we will not use the exchangeable pair approach for Stein's method as the conditional error $\E(L^{\gl}-L\mid \vW)$ is not linear with $L$.  Recall from the discussion on Stein's method in Section \ref{sec:stein}, in order to prove that $\hat L = (L-\mu)/\sigma$, one needs to bound $|\E(g'(\hat L) - \hat L g(\hat L))|$ for $g$ in the class of functions $\cD'$ in \eqref{eqn:class-fn-stein}. We will use a direct argument to bound this quantity.   

First note that $\cI_{\gl}(\vW)$ is independent of $L^{\gl}$. Thus  for any twice differentiable function $f$, we have
\begin{align*}
\E((L- \mu)f(L))
&=\sum_{\gl} \E (\cI_{\gl}f(L) - p_n f(L) ) \\
 &=\sum_{\gl}  \E( \cI_\gl (f(L) - f(L^\gl)))\\
&= \sum_{\gl}  \E\bigl( \cI_{\gl} ( (L- L^\gl) f'(L) -  \frac{1}{2}(L - L^\gl)^2 f''(L^\gl_*) )\bigr)
\end{align*}
where $L^\gl_*$ is a random variable.
In particular with $\hat{L}=(L-\mu)/\gs$ and $f(x)=g((x-\mu)/\gs)$ we have
\begin{align*}
|\E(\hat{L}g(\hat{L}) - g'(\hat{L}))|
& \leq \frac{||g'||_{\infty}}{\gs^2} \E\bigl|\sum_{\gl} \cI_\gl\E\bigl( L- L^\gl\mid \vW\bigr) - \gs^2\bigr|
 + \frac{||g'||_{\infty}}{2\gs^3} \E\sum_{\gl} \cI_{\gl}(L - L^\gl)^2.
\end{align*}
Note that by symmetry
\begin{align}
\E\sum_{\gl} \cI_{\gl}(L - L^\gl)^2 = \mu \E( (L - L^{\gl_0})^2\mid \cI_{\gl_0} )
\end{align}
where $\gl_{0}={[k]\times[k]}$ and for simplicity we write $E(\cdot\mid \cI_{\gl_0} ):= E(\cdot\mid \cI_{\gl_0}=1 ) $. Thus using Lemma~\ref{lem:stein} we have
\begin{align}\label{eq:stein1}
d_{\cW}(\hat{L},\text{N}(0,1)) \leq \frac{1}{\gs^2} \E\bigl|\sum_{\gl} \cI_\gl\E\bigl( L- L^\gl\mid \vW\bigr) - \gs^2\bigr|+ \frac{\mu}{\gs^3} \E( (L - L^{\gl_0})^2\mid \cI_{\gl_0} ).
\end{align}

Recall that, for $\gl,\gc\in\sS_{n}(k)$,  $|\gl\cap\gc|=(s,t)$ implies that $\gl$ and $\gc$ share $s$ many rows and $t$ many columns. For {fixed} $\gl\in\sS_{n}(k)$, define 
\[
\sS_{\gl}(s,t):=\{\gc\in\sS_{n}(k)\mid |\gl\cap\gc|=(k-s,k-t)\},\qquad {0}\leq s,t\leq k.
\]
Thus $\sS_{\gl}(s,t)$ consists of the set of submatrices which are $s$ rows and $t$ columns \emph{different }from $\lambda$. Write
\[
S_{\gl}(s,t):= \sum_{\gc\in \sS_{\gl}(s,t)} (\cI_\gc -  \cI_{\gc}(\vW^{\gl})) 
\]
so that we have $L-L^{\gl}=\sum_{0\leq s,t\leq k} S_{\gl}(s,t)$. Clearly 
\[
|\sS_{\gl}(s,t)|={k\choose s}{k\choose t}{n-k\choose s}{n-k\choose t}=O(n^{s+t}).
\] 
Let 
\[
u_{n}(s,t):= \E( S_{\gl}(s,t) \mid \cI_{\gl})
\]
By symmetry, this term is the same for all $\gl$. Recall from \eqref{eq:vnst} that the variance of $L_n(k)$ could be expressed as $\gs^2 =\sum_{s,t}v_{n}(s,t)$ where 
\begin{align*}
v_n(s,t):= {n\choose k}^{2}{k\choose s}{k\choose t} &{n-k\choose s}{n-k\choose t} \cov(\ind\{\vW_{[k]\times[k]}\text{ is locally optimal}\},\notag\\
&\ind\{\vW_{[s+1,s+k]\times[t+1,t+k]}\text{ is locally optimal}\})
\end{align*} 
A simple conditioning argument shows that $v_{n}(s,t)=\mu u_{n}(s,t)$. Now let us consider the first term in the bound \eqref{eq:stein1}.
\begin{align*}
\E&\bigl|\sum_{\gl} \cI_\gl\E\bigl( L - L^\gl\mid \vW\bigr) - \gs^2\bigr|\\
&\leq \sum_{s=0}^{k}\sum_{t=0}^{k} \E| \sum_{\gl \in\sS_{n}(k)}  \cI_{\gl} \E(S_{\gl}(s,t)\mid \vW) - \mu u_n(s,t)|\\
&\leq \sum_{s=0}^{k}\sum_{t=0}^{k} ( |\sS_{n}(k)|\cdot \E|\cI_{\gl_0} \E(S_{\gl_0}(s,t) - u_n(s,t)\mid \vW) |+ |u_n(s,t)| \cdot \E |L - \mu| )\\
&\leq \mu \sum_{s=0}^{k}\sum_{t=0}^{k} \E(| \E(S_{\gl_0}(s,t)\mid \vW) - u_n(s,t)| \mid \cI_{\gl_0} )+ \gs\sum_{s=0}^{k}\sum_{t=0}^{k} |u_n(s,t)|.
\end{align*}

Similarly for the second term in \eqref{eq:stein1} we have
\begin{align*}
\sqrt{\E( (L - L^{\gl_0})^2\mid \cI_{\gl_0} )} &\leq \sum_{s=0}^{k}\sum_{t=0}^{k} \sqrt{\E(S_{\gl_0}(s,t)^2 \mid \cI_{\gl_0}) }\\
&\leq \sum_{s=0}^{k}\sum_{t=0}^{k} (|u_n(s,t)| + \sqrt{\var(S_{\gl_0}(s,t) \mid \cI_{\gl_0}) }).
\end{align*}

The proof of the variance estimate in Theorem~\ref{thm:var} shows that $u_{n}(s,t)\geq 0$ for $st>0$ and $u_{n}(s,t)=-|\sS_{\gl_0}(s,t)|p_n$ for $st=0, s+t>0$. In particular we have 
\[
\sum_{s=0}^{k}\sum_{t=0}^{k} |u_n(s,t)| = \frac{1}{\mu} \sum_{s=0}^{k}\sum_{t=0}^{k} |v_n(s,t)| \leq \frac{c\gs^2}{\mu} 
\]
for some constant $c>0$. Combining, the bound \eqref{eq:stein1} reduces to
\begin{align}
d_{\cW}(\hat{L},\text{N}(0,1)) 
&\leq \sum_{s=0}^k\sum_{t=0}^k \frac{\mu}{\gs^2}  \E(| \E(S_{\gl_0}(s,t)\mid \vW) - u_n(s,t)| \mid \cI_{\gl_0} ) +\frac{c\gs}{\mu}\notag\\
&+ \left(\sqrt{\frac{c^2\gs}{\mu}} + \sum_{s=0}^k \sum_{t=0}^{k}\sqrt{\frac{\mu}{\gs^3}\var(S_{\gl_0}(s,t) \mid \cI_{\gl_0}) }\right)^2.\label{eq:stein2}
\end{align}

From \eqref{eq:vmratio} it follows that $\gs/\mu\to 0$ as $n\to\infty$. Moreover, for $st=0$ we have $|S_{\gl_0}(s,t)|\leq 1$ a.s. Note that 
\begin{align*}
\frac{\gs^2}{\mu} = n^{k-2+2/(k+1)+o(1)}
\text{ and }\frac{\gs^3}{\mu} = n^{2k-3+3/(k+1)+o(1)}. 
\end{align*}

Thus the case $st = 0$ is negligible and we are left to prove that
\begin{align*}
\gC_{1} &:= \sum_{s=1}^k\sum_{t=1}^k \frac{\mu}{\gs^2}  \E(| \E(S_{\gl_0}(s,t)\mid \vW) - u_n(s,t)| \mid \cI_{\gl_0} ) \to 0\\
\gC_{2}&:= \sum_{s=1}^k \sum_{t=1}^{k}\sqrt{\frac{\mu}{\gs^3}\var(S_{\gl_0}(s,t) \mid \cI_{\gl_0}) }\to0
\end{align*}
as $n\to\infty$. Clearly
\begin{align}\label{eq:varubd}
 \E(| \E(S_{\gl_0}(s,t)\mid \vW) - u_n(s,t)| \mid \cI_{\gl_0} ) \leq \sqrt{\var(S_{\gl_0}(s,t) \mid \cI_{\gl_0}) }.
\end{align}

Recall that,
\[
S_{\gl_{0}}(s,t):= \sum_{\gc\in \sS_{\gl_{0}}(s,t)} (\cI_\gc -  \cI_\gc(\vW^{\gl_0})) 
\]
where 
\[
\sS_{\gl}(s,t)=\{\gc\in\sS_{n}(k)\mid |\gl\cap\gc|=(k-s,k-t)\}.
\]
We start with the term $\gC_{1}$. We consider different cases depending on the values of $s,t$. Note that, $\E(S_{\gl_0}(s,t) \mid \cI_{\gl_0})=u_{n}(s,t)\ll \gs^2/\mu$ for $st<k^2$. Thus, heuristically for $st<k^2$, the contribution in $\gC_{1}$  should be $\ll 1$ as $n\to\infty$. Obviously the nontrivial case is when $s=t=k$. 
\\

\noindent{\bf Case 1.} $st>0, s+t\leq 2k-2$: In this case we have $u_{n}(s,t)\geq 0$ and thus
 \[
 \E( \cI_{\gc}(\vW^{\gl_0} )\mid \cI_{\gl_0}) \leq \E( \cI_\gc \mid \cI_{\gl_0})
\] 
for $\gc\in \sS_{\gl_0}(s,t)$. 
Now we have
\begin{align*}
&\E(| \E(S_{\gl_0}(s,t)\mid \vW) - u_n(s,t)| \mid \cI_{\gl_0} )\\
 &\leq \sum_{\gc\in\sS_{\gl_{0}}(s,t)} \E(\cI_\gc +  \cI_{\gc}(\vW^{\gl_{0}} )\mid \cI_{\gl_{0}}) + |u_n(s,t)|\\
 &= \sum_{\gc\in\sS_{\gl_{0}}(s,t)} \E(\cI_\gc +  \cI_{\gc}(\vW^{\gl_0}) \mid \cI_{\gl_0}) + \sum_{\gc\in\sS_{\gl_{0}}(s,t)} \E(\cI_\gc -  \cI_{\gc}(\vW^{\gl_0})\mid \cI_{\gl_0}) \\
 &= 2 {k\choose s}{k\choose t} {n-k\choose s}{n-k\choose t} \pr(\cI_{[s+1,s+k]\times[t+1,t+k]}\mid \cI_{[k]\times[k]}).
\end{align*}
Now using the results in case 3 and 5 from the proof of Theorem~\ref{thm:var} we have
\begin{align}\label{eq:condprob1}
\pr(\cI_{[s+1,s+k]\times[t+1,t+k]}\mid \cI_{[k]\times[k]}) \leq n^{-k + 2k(k-s)(k-t)/(2k^2-st) + o(1)}
\end{align}
and 
\[
2{k\choose s}{k\choose t} {n-k\choose s}{n-k\choose t} \pr(\cI_{[s+1,s+k]\times[t+1,t+k]}\mid \cI_{[k]\times[k]}) \leq \eps_{n}\gs^{2}/\mu
\]
where 
\[
\eps_{n}:= O((\log n/n^2)^{\frac{k-1}{(k+1)(k^2+2k-1)}}(\log n)^{2k-1}).
\]
Thus we have
\[
\frac{\mu}{\gs^2}\E(| \E(S_{\gl_0}(s,t)\mid \vW) - u_n(s,t)| \mid \cI_{\gl_0} ) \leq \eps_n.
\]

\noindent{\bf Case 2:} $s+t=2k-1$: This corresponds to the set of matrices which have exactly one row in common with $\lambda_0$ and no columns, or vice-vera. Without loss of generality assume the former case (the later is dealt with identically) so that $s=k-1,t=k$. By \eqref{eq:varubd} it is enough to prove that 
\begin{align*}
\E(S_{\gl_0}(k-1,k)^2\mid \cI_{\gl_0}) &\ll \gs^4/\mu^2.
\end{align*}
Note that
\begin{align*}
\E(S_{\gl_0}(k-1,k)\mid \cI_{\gl_0}) &=v_{n}(k-1,k)/\mu \ll \gs^2/\mu.
\end{align*}
We will write 
\[
\hat{\cI}_{\gc}:= \cI_\gc(\vW^{\gl_0}) \text{ and }\pr_{\gl_0}(\cdot)=\pr(\cdot\mid \cI_{\gl_{0}}).
\]

Note that, any matrix in $\sS_{\gl_{0}}(k-1,k)$ is contained in the sub matrix $[n]\times[k+1,n]$ with exactly one row with index in $[k]$. For two matrix indices $\gc,\gc'\in \sS_{{\gl_0}}(k-1,k)$ define $\cN(\gc,\gc')=(\ell,r,c)$ where $\ell=1$ if $\gc,\gc'$ share  a row in $[k]$ and $0$ otherwise; $r$ is the number of common rows between $\gc,\gc'$ in $[k+1,n]$ and $c$ is the number of common columns between $\gc,\gc'$. Note that 
\begin{align*}
&|\{(\gc,\gc')\mid \cN(\gc,\gc')=(\ell,r,c) \}| \\
&= k( (k-1)\ind\{\ell=0\} + \ind\{\ell=1\} ){n-k\choose k}{k\choose c}{n-2k\choose k-c}\\
&\qquad\qquad\qquad\qquad\qquad\qquad{n-k\choose k-1}{k-1\choose r}{n-2k+1\choose k-1-r}\\
&= O(n^{4k-2-r-c}).
\end{align*}
Thus we have
\begin{align*}
&\E(S_{\gl_0}(k-1,k)^2\mid \cI_{\gl_0}) \\
&\leq O_{k}(1) \sum_{\ell=0}^{1}\sum_{r=0}^{k-1}\sum_{c=0}^{k} n^{4k-2-r-c} \E( (\cI_\gc -  \hat{\cI}_{\gc}) (\cI_{\gc_{\ell,r,c}} - \hat{\cI}_{\gc_{\ell,r,c}})  \mid \cI_{\gl_{0}})
\end{align*}
where $\cN(\gc,\gc_{\ell,r,c})=(\ell,r,c)$.  Now for $ \gc,\gc'\in\sS_n(k)$  we have
\begin{align*}
&\E((\cI_\gc -  \hat{\cI}_{\gc}) \cdot (\cI_{\gc'} - \hat{\cI}_{\gc'})  \mid \cI_{\gl_{0}})\\
&= \pr_{\gl_0}(\cI_\gc \hat{\cI}_{\gc}^{c} \cI_{\gc'} \hat{\cI}_{\gc'}^{c}) 
- \pr_{\gl_0}(\cI_\gc \hat{\cI}_{\gc}^{c} \cI_{\gc'}^{c} \hat{\cI}_{\gc'})
- \pr_{\gl_0}(\cI_\gc^{c} \hat{\cI}_{\gc} \cI_{\gc'} \hat{\cI}_{\gc'}^{c})
+ \pr_{\gl_0}(\cI_\gc^{c} \hat{\cI}_{\gc} \cI_{\gc'}^{c} \hat{\cI}_{\gc'}).
\end{align*}

For $rc=0, r+c>1$ the contribution in $\E(S_{\gl_0}(k-1,k)^2\mid \cI_{\gl_0})  $ is bounded by $ O_{k}(1)n^{4k-2 -r-c}n^{-2k}\leq  O_{k}(1)n^{2k-4}$. To see this, consider the first term in the above equation. Here we require both $\gamma, \gamma^\prime$ to be locally optimal, in particular column optimal and thus must possess the largest $k$ row sums in their respective column set, each of which has probability (even conditioning on $\cI_{\lambda_0}$) of at most than $1/{n-2k \choose k}$. When $r+c=0$, one can prove that (using the method used in the proof of Theorem~\ref{thm:var} for $s=t=k$) 
\[
\E( (\cI_\gc -  \hat{\cI}_{\gc}) (\cI_{\gc_{\ell,r,c}} - \hat{\cI}_{\gc_{\ell,r,c}})  \mid \cI_{\gl_{0}})=O( n^{-2k - 2})
\]
and for $r+c=1$ 
\[
\E( (\cI_\gc -  \hat{\cI}_{\gc}) (\cI_{\gc_{\ell,r,c}} - \hat{\cI}_{\gc_{\ell,r,c}})  \mid \cI_{\gl_{0}})=O( n^{-2k - 1}).
\]
The $n^{-2k}$ term comes from the probability that both $\gc$ and $\gc_{\ell,r,c}$ are locally optimal and the $1/n$ improvement is coming from the fact that $\E( \cI_\gc -  \hat{\cI}_{\gc}  \mid \cI_{\gl_{0}})=O( n^{-k - 1}).$ Thus for $rc=0$, the total contribution is $O(n^{2k-4})$. 
When $rc\in [1,k(k-1)], \ell\in\{0,1\}$, the contribution  is 
\[
n^{4k-2-r-c}n^{-2k+\frac{2k(r+\ell)c}{2k^2-(k-r-\ell)(k-c)} }.
\]
The maximum power occurs for $r=c=\ell=1$ so that the contribution is bounded by
\[
n^{2k-4+\frac{2k(1+\ell)}{2k^2-(k-1-\ell)(k-1)}}\leq n^{2k-4 + 4k/(k^2+3k-2)}
\]
and $4k/(k^2+3k-2) = 4/(k+1) - 8(k-1)/((k+1)(k^2+3k-2))$. Thus combining everything we have
\begin{align*}
\E(S_{\gl_0}(k-1,k)^2\mid \cI_{\gl_0})
&=O(n^{2k-4 + 4/(k+1) - 8(k-1)/((k+1)(k^2+3k-2)) })\\
&= O(n^{- 8(k-1)/((k+1)(k^2+3k-2))}\cdot \gs^4/\mu^2).
\end{align*}

\noindent{\bf Case 3.} $s=t=k$. This corresponds to the set of matrices which have no common rows or columns with $\lambda_0$. We move to the proof of 
\[
\var(S_{\gl_0}(k,k)\mid \cI_{\gl_0})\ll \gs^4/\mu^2.
\]
Note that, any matrix in $\sS_{\gl_{0}}(k,k)$ is contained in the sub matrix $[k+1,n]\times[k+1,n]$. Also we have
\begin{align*}
&|\{\gc,\gc'\in\sS_{\gl_0}(k,k)\mid |\gc\cap\gc'|=(r,c) \}| \\
&= {n-k\choose k}{k\choose r}{n-2k\choose k-r}{n-k\choose k}{k\choose c}{n-2k\choose k-c}
= O(n^{4k-r-c}).
\end{align*}
Thus we have
\begin{align*}
&\var(S_{\gl_0}(k,k)^2\mid \cI_{\gl_0}) \\
&\leq O(1) \sum_{\ell=0}^{1}\sum_{r=0}^{k-1}\sum_{c=0}^{k} n^{4k-r-c} \cov( \cI_\gc -  \hat{\cI}_{\gc}, \cI_{\gc_{r,c}} - \hat{\cI}_{\gc_{r,c}}  \mid \cI_{\gl_{0}})
\end{align*}
where $\cN(\gc,\gc_{r,c})=(r,c)$.  
For $r=c=k$, the total contribution in the variance is 
\[
O(n^{2k}n^{-k-1})\leq O(n^{2k-4+3/(k+1)}). 
\]
Note that here $n^{-k-1}$ term comes from the fact that $\cI_{\gc}$ has probability $n^{-k}$ and after changing the elements in the first $k$ rows and $k$ columns $\gc$ is no longer locally optimal implies one of the new rows or columns beat $\gc$ which has probability $1/n$.  In particular, similar to the variance calculation for $L_{n}$, for all $rc=0, r+c>1$ the contribution is
\[
\leq n^{4k-r-c -2k-2}\leq n^{2k-4}. 
\]
and for all $rc\geq 1$ the contribution is
\begin{align*}
&n^{4k-r-c}n^{-2k+2krc/(2k^2-(k-r)(k-c)) - 2/(1+\max\{r,c\}/k)} \\
&\qquad\leq n^{ - 2(k-1)/((k+1)(k^2+2k-1))}\gs^4/\mu^2
\end{align*}
where the largest exponent occurs for $r=c=1$.
Thus the only terms remaining to bound are when $r+c=1$ and $r+c=0$. We look at the $r+c=0$ case first. We want to bound 
\[
\sum_{\gc,\gc'\in\sS_{\gl_0}(k,k), |\gc\cap\gc'|=(0,0)} \cov(\cI_\gc -  \hat{\cI}_{\gc},\cI_{\gc'} - \hat{\cI}_{\gc'}  \mid \cI_{\gl_{0}}). 
\] 
Number of summands in the above sum is $O(n^{4k})$. Now after some simplification it is easy to see that we need to bound 
\[
\cov(\cI_{\gc}\hat{\cI}^c_{\gc}, \cI_{\gc'}\hat{\cI}_{\gc'}^{c}  \mid \cI_{\gl_{0}})
\]
which, by Lemma~\ref{lem:bigmax} can be bounded by
\[
n^{-2k-2 - 2k/(k+1)} = n^{-2k-4 + 2/(k+1)}.
\]
Thus the total contribution is 
\[
n^{4k-2k-4 +2/(k+1)}= n^{2k-4+4/(k+1)-2/(k+1)}=n^{-2/(k+1)}\gs^4/\mu^2.
\]
Similarly for the $r=1,c=0$ case the total contribution is 
\[
n^{4k-1} n^{-2k-2-1}=n^{2k-4}=n^{-4/(k+1)}\gs^4/\mu^2.
\]
Combining everything we have $\gC_{1}\to0$ as $n\to\infty$.

Now we show that 
\[
\gC_{2}=\sum_{s=1}^k \sum_{t=1}^{k}\sqrt{\frac{\mu}{\gs^3}\var(S_{\gl_0}(s,t) \mid \cI_{\gl_0}) }\to  0
\]
 as $n\to\infty$. Note that, $\E(S_{\gl_0}(s,t) \mid \cI_{\gl_0})=u_{n}(s,t)\leq \gs^2/\mu$ for all $s,t$. Heuristically for fixed $s,t$ the contribution in  $\gC_{2}$ should be $\le\sqrt{\mu/\gs^3\cdot \gs^4/\mu^2}= \sqrt{\gs/\mu}\to 0$ as $n\to\infty$. We leave the  proof  to the interested reader where the proof follows exactly the same steps used in case $1$--$3$ of the proof of $\gC_{1}\to0$. 
Combining everything finally we have the result that
\begin{align}
d_{\cW}(\hat{L},\text{N}(0,1)) \to 0
\end{align}
as $n\to\infty$.
\qed


\vskip.1in
\noindent{\bf Acknowledgments.}  PD is grateful for the hospitality of the Department of Statistics and Operations research, University of North Carolina, Chapel Hill, where much of the research was done. SB was partially supported by NSF grant DMS-1105581. PD was supported by Simons Postdoctoral Fellowship. AN was partially supported by NSF grant DMS-0907177.

\bibliographystyle{plain}
\bibliography{matrix}

%
%

\end{document}